\documentclass[11pt]{amsart}

\usepackage{graphicx}
\usepackage[english]{babel}
\usepackage[T1]{fontenc}
\usepackage[latin1]{inputenc}
\usepackage{amsfonts}
\usepackage{amssymb}
\usepackage{amsthm}
\usepackage{amsmath}
\usepackage{tikz}
\usepackage{float}
\usepackage{enumerate}
\usepackage{accents}
\usepackage{mathtools}
\usepackage{mathrsfs}
\usepackage{comment}
\usepackage{afterpage}
\usepackage{bbm}
\usepackage{dsfont}
\usepackage{stmaryrd}
\usepackage{hyperref}
\usepackage{mleftright}
\usepackage{subfig}
\usepackage{soul}
\usepackage{tikz-cd} 
\usepackage[foot]{amsaddr}
\usepackage{fullpage}

\theoremstyle{plain}
\newtheorem{thm}{Theorem}[section]
\newtheorem{lem}[thm]{Lemma}
\newtheorem{prop}[thm]{Proposition}
\newtheorem{cor}[thm]{Corollary}
\newtheorem*{thm*}{Theorem}
\newtheorem*{prop*}{Proposition}
\newtheorem*{cor*}{Corollary}
\newtheorem{thmintro}{Theorem}

\newtheorem{corintro}[thmintro]{Corollary}

\theoremstyle{definition}
\newtheorem{defn}[thm]{Definition}
\newtheorem{ex}[thm]{Example}
\newtheorem{rmk}[thm]{Remark}
\newtheorem{ass}[thm]{Assumption}
\newtheorem*{rmk*}{Remark}
\newtheorem*{conj*}{Conjecture}

\newtheorem*{quest*}{Question}

\newtheorem*{defn*}{Definition}

\newcommand{\acts}{\curvearrowright}
\newcommand{\ra}{\rightarrow}
\newcommand{\Ra}{\Rightarrow}

\newcommand{\cu}{\subseteq}

\newcommand{\wt}{\widetilde}

\newcommand{\x}{\times}
\renewcommand{\o}{\circ}
\newcommand{\id}{\mathrm{id}}

\newcommand{\mc}{\mathcal}
\newcommand{\mf}{\mathfrak}
\newcommand{\mscr}{\mathscr}

\newcommand{\R}{\mathbb{R}}
\newcommand{\Z}{\mathbb{Z}}
\newcommand{\N}{\mathbb{N}}

\newcommand{\s}{\sigma}
\newcommand{\eps}{\epsilon}
\newcommand{\Om}{\Omega}
\newcommand{\om}{\omega}

\renewcommand{\L}{\Lambda}
\newcommand{\g}{\gamma}
\newcommand{\G}{\Gamma}

\newcommand{\X}{\mc{X}}
\newcommand{\A}{\mc{A}}

\newcommand{\T}{\mc{T}}

\newcommand{\CAT}{{\rm CAT(0)}}
\DeclareMathOperator{\Fix}{Fix}

\DeclareMathOperator{\lk}{lk}
\DeclareMathOperator{\St}{st}

\DeclareMathOperator{\aut}{Aut}
\DeclareMathOperator{\out}{Out}

\DeclareMathOperator{\rk}{rk}

\newcommand{\Min}{\mathrm{Min}}

\begin{document}

\title{On automorphisms and splittings of special groups} 
\author[E. Fioravanti]{Elia Fioravanti}\address{Max Planck Institute for Mathematics, Vivatsgasse 7, 53111 Bonn, Germany}\email{fioravanti@mpim-bonn.mpg.de} 

\begin{abstract}
We initiate the study of outer automorphism groups of special groups $G$, in the Haglund--Wise sense. We show that $\out(G)$ is infinite if and only if $G$ splits over a co-abelian subgroup of a centraliser and there exists an infinite-order ``generalised Dehn twist''. Similarly, the coarse-median preserving subgroup $\out_{\rm cmp}(G)$ is infinite if and only if $G$ splits over an actual centraliser and there exists an infinite-order coarse-median-preserving generalised Dehn twist.

The proof is based on constructing and analysing non-small, stable $G$--actions on $\R$--trees whose arc-stabilisers are centralisers or closely related subgroups. Interestingly, tripod-stabilisers can be arbitrary centralisers, and thus are large subgroups of $G$.

As a result of independent interest, we determine when generalised Dehn twists associated to splittings of $G$ preserve the coarse median structure.
\end{abstract}

\maketitle
%\tableofcontents  (see below)

% compared to resubmitted version (02.09.2022), thanked referee in acknowledgements

\section{Introduction.}

It was first shown by Dehn in 1922 
% this is mentioned in the Primer; in Farb's collection of problems, a more precise reference is given in the bibliography to Chapter 7:
% Max Dehn, Lecture notes from Breslau, 1922, Archives of the University of Texas at Austin.
that mapping class groups of closed surfaces are generated by finitely many Dehn twists around simple closed curves \cite{Dehn}. Many decades later, one of the successes of Rips--Sela theory was the extension of this result to outer automorphism groups of all Gromov-hyperbolic groups \cite{Rips-Sela}. 

More precisely, whenever a group $G$ splits as an amalgamated product $G=A\ast_CB$, we can construct an automorphism $\varphi\in\aut(G)$ by defining $\varphi|_A$ as the identity and $\varphi|_B$ as the conjugation by an element of the centre of $C$. A similar construction can be applied to HNN splittings $G=A\ast_C$. We refer to group automorphisms obtained in this way as \emph{algebraic Dehn twists}. Indeed, when $G=\pi_1\Sigma$ for a closed surface $\Sigma$ and $C\simeq\Z$, algebraic Dehn twists are precisely the action on $\pi_1\Sigma$ of the usual homeomorphisms of $\Sigma$ known as Dehn twists.

When $G$ is a one-ended Gromov-hyperbolic group (without torsion), Rips and Sela showed that a finite index subgroup of $\out(G)$ is generated by finitely many algebraic Dehn twists arising from \emph{cyclic} splittings of $G$ \cite{Rips-Sela}. An analogous result was obtained by Groves for toral relatively hyperbolic groups $G$, where one must consider more generally all \emph{abelian} splittings of $G$ \cite{Groves-AGT09}.

Both results are proved by first constructing an isometric $G$--action on an $\R$--tree and then applying the Rips machine. However, the trees involved have a very specific structure --- they are \emph{superstable} and \emph{small} (i.e.\ with abelian arc-stabilisers) --- and thus they do not require the full power of Rips' techniques, which can handle \emph{stable} $G$--trees with arbitrary arc-stabilisers \cite{BF-stable}.

For this reason, it is natural to expect that the class of groups $G$ for which $\out(G)$ can be understood through Rips--Sela theory should be broader. The difficulty to overcome is that, when $G$ lacks strong hyperbolic features (mostly Gromov-hyperbolicity or relative hyperbolicity), it is generally hard to construct $G$--trees that simultaneously capture many significant features of the geometry of $G$. Nevertheless, in certain contexts, acylindrical hyperbolicity has been shown to suffice when addressing related questions, such as equational Noetherianity \cite{Groves-Hull,Groves-Hull-Liang} and the existence of (higher-rank) Makanin--Razborov diagrams \cite{Sela-new}.

It is worth remarking that the above results are no exception and, in fact, all classical applications of the Rips machine only require its most ``basic'' form for small superstable $G$--trees: from acylindrical accessibility \cite{Sela-acyl} and JSJ decompositions for finitely presented groups \cite{RS-Ann}, to the Hopf property \cite{Sela-Top} and the isomorphism problem for hyperbolic groups \cite{Sela-Ann,Dahmani-Guirardel-GAFA},
% also \cite{Dahmani-Groves,Dahmani-Guirardel-Duke} for relative hyperbolicity, but the latter can have parabolic arc-stabilisers
 to the elementary theory of free groups \cite{SelaI,SelaVI}.

\medskip
In this paper, we seek to obtain an analogous relationship between the structure of $\out(G)$ and the splittings of $G$ when $G$ is not relatively hyperbolic. 

We choose to focus on \emph{special groups} $G$, in the sense of Haglund and Wise \cite{Haglund-Wise-GAFA}. This is the remarkably broad class of  subgroups of right-angled Artin groups that are quasi-convex in the standard word metric. Little seems to be known on $\out(G)$ in this context, other than the fact that it is always a residually finite group \cite{AMS}.

Special groups are best known for their Gromov-hyperbolic examples --- hyperbolic $3$--manifold groups \cite{Kahn-Markovic,Bergeron-Wise,Agol-Doc}, hyperbolic free-by-cyclic groups \cite{Hagen-Wise-Duke,Hagen-Wise-GAFA}, finitely presented small cancellation groups \cite{Wise04} among many others --- and their role in Agol and Wise's resolution of Thurston's virtual fibering and virtual Haken conjectures \cite{Agol-ICM,Wise-ICM}.

However, special groups also admit many non-relatively-hyperbolic examples: for instance, finite-index subgroups of right-angled Artin and Coxeter groups, most non-geometric $3$--manifold groups \cite{Przytycki-Wise14,Hagen-Przytycki15,Przytycki-Wise18}, graph braid groups \cite{Crisp-Wiest}, cocompact diagram groups \cite{Guba-Sapir,Gen-diag1,Gen-diag2}, and the examples from \cite{Kropholler-Vigolo}. Various other non-hyperbolic groups are expected to be special (for instance, among free-by-cyclic groups), but for the moment this runs into the general difficulty of cocompactly cubulating groups without relying on Sageev's criterion \cite{Sag97,Sag95}.

The case when $G$ is a right-angled Artin group $\A_{\G}$ already demonstrates that $\out(G)$ can well be infinite even when $G$ does not split over an abelian subgroup \cite{Groves-Hull-ab}, suggesting that we will have to deal with non-small $G$--trees and automorphisms of $G$ that are more general than the algebraic Dehn twists defined above.

A particularly simple generating set for $\out(\A_{\G})$ was given by Laurence and Servatius \cite{Laurence,Servatius}. With this in mind, it is natural to consider the following generalisation of algebraic Dehn twists which provides a unified perspective on automorphisms of hyperbolic groups and right-angled Artin groups. If $G$ is a group, $H\leq G$ is a subgroup and $K\cu G$ is a subset, we denote by $Z_H(K)$ the \emph{centraliser} of $K$ in $H$.

\begin{defn*}[DLS automorphisms]
Let $G$ be a group. A \emph{Dehn--Laurence--Servatius automorphism} of $G$ is any of the following two kinds of automorphisms of $G$.
\begin{itemize}
\item Suppose that $G$ splits as an amalgamated product $A\ast_CB$. Each element $z\in Z_A(C)$ defines an automorphism $\s\in\aut(G)$ with $\s(a)=a$ for all $a\in A$ and $\s(b)=zbz^{-1}$ for all $b\in B$. We refer to $\s$ as a \emph{partial conjugation}.
\item Suppose that $G$ splits as an HNN extension $A\ast_C=\langle A,t \mid t^{-1}ct=\alpha(c),\ \forall c\in C\rangle$. Each $z\in Z_A(C)$ defines an automorphism $\tau\in\aut(G)$ with $\tau(a)=a$ for all $a\in A$ and $\tau(t)=zt$. We refer to $\tau$ as a \emph{transvection}.
\end{itemize}
\end{defn*}

DLS automorphisms generate a finite index subgroup of $\out(G)$ both when $G$ is hyperbolic (or toral relatively hyperbolic) and when $G$ is a right-angled Artin or Coxeter group. 

DLS automorphisms were previously introduced by Levitt \cite{Levitt-hyp} and they appear in even earlier work of Bass and Jiang \cite{Bass-Jiang}. DLS automorphisms are often simply called ``twists'' in the literature, but this terminology would be rather confusing in the present paper since, in the context of right-angled Artin groups, the word ``twist'' has come to refer only to a very specific type of DLS auto\-mor\-phism \cite{CSV}. We will stick to the latter convention and reserve the term ``twist'' for transvections induced by elements of the centre of $C$ (see below).

When $G$ is a special group with a fixed embedding in a right-angled Artin group $G\hookrightarrow\A_{\G}$, we can endow $G$ with a natural \emph{coarse median} structure $[\mu]$ \cite{Bow-cm}. This provides us with a notion of \emph{quasi-convexity} for subgroups $H\leq G$. In fact, a subgroup $H\leq G$ is quasi-convex with respect to $[\mu]$ if and only if its action on the universal cover of the Salvetti complex of $\A_{\G}$ is \emph{convex-cocompact}: $H$ stabilises a convex subcomplex, acting cocompactly on it. For this reason, we will speak interchangeably of ``quasi-convex'' and ``convex-cocompact'' subgroups of $G$.

The coarse median structure on $G$ also gives us a notion of \emph{orthogonality} between subgroups (denoted $\perp$, see Definition~\ref{pf defn}). Because of this, it is convenient to differentiate between two types of transvections that always display quite different behaviours. This distinction was first introduced for automorphisms of right-angled Artin groups in \cite{CSV}.

\begin{defn*}[Twists and folds]
Let $(G,[\mu])$ be a coarse median group. Suppose that $G$ splits as an HNN extension $A\ast_C$, where $C$ is quasi-convex with respect to $[\mu]$. Let $\tau\in\aut(G)$ be the transvection determined by an element $z\in Z_A(C)$.
\begin{itemize}
\item If $z$ lies in the centre of $C$, we say that $\tau$ is a \emph{twist}.
\item If instead $\langle z\rangle\perp C$, we say that $\tau$ is a \emph{fold}.
\end{itemize}
\end{defn*}

Fixing an embedding of $G$ in $\A_{\G}$ and the corresponding coarse median structure $[\mu]$, it is also interesting to study the subgroup $\out_{\rm cmp}(G)\leq\out(G)$ of coarse-median preserving automorphisms. This was introduced in previous work of the author \cite{Fio10a} and often makes up a significant portion of the whole automorphism group. For instance, $\out_{\rm cmp}(G)=\out(G)$ when $G$ is either Gromov-hyperbolic or a right-angled Coxeter group, while $\out_{\rm cmp}(G)$ is the group of \emph{untwisted} automorphisms when $G$ is a right-angled Artin group (which was studied e.g.\ in \cite{CSV,Hensel-Kielak}).

The results of \cite{Fio10a} show that, in various respects, $\out_{\rm cmp}(G)$ displays a much closer similarity to automorphisms of hyperbolic groups than the whole $\out(G)$. This pattern is confirmed in the present paper (compare Theorems~\ref{main cmp} and~\ref{main general}).

\medskip
We are now ready to state our two main theorems. Previous results of this type for hyperbolic and relatively hyperbolic groups include \cite{Paulin-arboreal,Levitt-hyp,Drutu-Sapir,Guirardel-Levitt-GGD15}, among many others. The correct extension to general special groups seems to require replacing \emph{abelian} subgroups with \emph{centralisers}.
% stress that the trivial group can be the centraliser of a finite subset?

\begin{thmintro}\label{main cmp}
Let $G$ be a special group. Then $\out_{\rm cmp}(G)$ is infinite if and only if $\out_{\rm cmp}(G)$ contains an infinite-order DLS automorphism $\varphi$ of one of the following forms:
\begin{enumerate}
\item $G$ splits as $A\ast_CB$ or $A\ast_C$, where $C$ is the centraliser of a finite subset of $G$, and $\varphi$ is a partial conjugation or fold associated to this splitting;
\item $G$ splits as $A\ast_C$, where $C=Z_G(g)$ for an element $g\in G$ such that $\langle g\rangle$ is convex-cocompact, and $\varphi$ is the twist determined by this splitting and the element $g$. % surface case, or simplicial and non--elliptic in all T_n
\end{enumerate}
\end{thmintro}

\begin{thmintro}\label{main general}
Let $G$ be a special group. Then $\out(G)$ is infinite if and only if $\out(G)$ contains an infinite-order DLS automorphism $\varphi$ either as in Types~(1) and~(2) of Theorem~\ref{main cmp}, or of the form:
\begin{enumerate}
\setcounter{enumi}{2}
\item $\varphi$ is a twist associated to an HNN splitting $G=A\ast_C$, where $C$ is the kernel of a nontrivial homomorphism $Z_G(x)\ra\Z$ for some $x\in G$. In addition, the stable letter of the HNN splitting can be chosen within $Z_G(x)$.
\end{enumerate}
\end{thmintro}

Thinking of right-angled Artin groups $\A_{\G}$, most of the Laurence--Servatius generators for $\out(\A_{\G})$ fall into Types~(1) and~(3) of the previous theorems (for the experts, we are only excluding \emph{graph automorphisms} and \emph{inversions}, which have finite order). Automorphisms of Type~(2) never occur for $\A_{\G}$, but many algebraic Dehn twists of hyperbolic groups are of this form.

We emphasise that, in many cases, a general DLS automorphism can turn out to be an inner automorphism of $G$, or to have an inner power. By contrast, Theorems~\ref{main cmp} and~\ref{main general} do provide DLS automorphisms with infinite order in $\out(G)$. In particular, this shows that $\out(G)$ and $\out_{\rm cmp}(G)$ can never be infinite torsion groups.

An immediate consequence of the above two theorems is the following.

\begin{corintro}\label{splitting cor}
Let $G$ be a special group.
\begin{enumerate}
\item If $\out_{\rm cmp}(G)$ is infinite, then $G$ splits over the centraliser of a finite subset of $G$.
\item If $\out(G)$ is infinite, then $G$ splits over the centraliser of a finite subset of $G$, or over the kernel of a homomorphism $Z_G(x)\ra\Z$ for some $x\in G$. 
\end{enumerate}
\end{corintro}

Note that all special groups split over ``some'' convex-cocompact subgroup, simply because they act properly on Salvetti complexes, hence on the associated products of trees. Such a splitting does not tell us anything about $\out(G)$ in general, so it is important that the splittings provided by Corollary~\ref{splitting cor} are over \emph{centralisers}, or subgroups thereof.

It seems that Corollary~\ref{splitting cor}(2) and part of Theorem~\ref{main general} can also be deduced from the work of Casals-Ruiz and Kazachkov \cite{CRK1,CRK2}. Indeed, if $\out(G)$ is infinite, the Bestvina--Paulin construction yields a nice $G$--action on an asymptotic cone of a right-angled Artin group. By \cite[Theorem~9.33]{CRK2}, it follows that $G$ can be embedded in a \emph{graph tower}, in the sense of \cite[Section~5]{CRK2}. Every graph tower $T$ admits particular splittings over centralisers of subsets of $T$. With significant additional work, these splittings can be translated into splittings of $G$ over the required subgroups of centralisers of subsets of $G$ (the main difficulty is passing from centralisers of subsets of $T$ to centralisers of subsets of $G$).

We emphasise that Corollary~\ref{splitting cor}(1) can fail if $\out(G)$ is infinite, but $\out_{\rm cmp}(G)$ is finite. The simplest example is provided by the right-angled Artin group $\A_{\G}$ with $\G$ as in Figure~\ref{intro fig}. Note that $G=\A_{\G}$ does not split over any centraliser of a subset, but it does split as an HNN extension over the subgroup $\langle b,c,f\rangle$, which is the kernel of a homomorphism $Z_G(a)\ra\Z$.

\begin{figure} 
\begin{tikzpicture}
\draw[fill] (-1,-0.5) -- (1,-0.5);
\draw[fill] (-1,0.5) -- (1,0.5);
\draw[fill] (-1,-0.5) -- (-1,0.5);
\draw[fill] (1,-0.5) -- (1,0.5);
\draw[fill] (-1,-0.5) -- (0,0.5);
\draw[fill] (-1,0.5) -- (0,-0.5);
\draw[fill] (-1,-0.5) circle [radius=0.04cm];
\draw[fill] (-1,0.5) circle [radius=0.04cm];
\draw[fill] (1,-0.5) circle [radius=0.04cm];
\draw[fill] (1,0.5) circle [radius=0.04cm];
\draw[fill] (0,-0.5) circle [radius=0.04cm];
\draw[fill] (0,0.5) circle [radius=0.04cm];
\node[below left] at (-1,-0.5) {$a$};
\node[above left] at (-1,0.5) {$b$};
\node[below right] at (1,-0.5) {$e$};
\node[above right] at (1,0.5) {$d$};
\node[below] at (0,-0.5) {$f$};
\node[above] at (0,0.5)  {$c$};
\end{tikzpicture}
\caption{}
%A graph $\Gamma$ such that ${\rm Out}(\mathcal{A}_{\Gamma})$ is infinite, but $\mathcal{A}_{\Gamma}$ does not split over any centralisers.
\label{intro fig} 
\end{figure}

With a bit more work, it is also possible to deduce from Theorems~\ref{main cmp} and~\ref{main general} the following result, which I found rather unexpected.

\begin{corintro}\label{infinite out finite cmp}
Let $G$ be a special group. Suppose that $\out(G)$ is infinite, but $\out_{\rm cmp}(G)$ is finite. Then there exists $x\in G$ such that the $G$--conjugacy class of the subgroup $Z_G(x)$ is preserved by a finite-index subgroup of $\out(G)$.
\end{corintro}
% when G is a RAAG, this is clear; just take a vertex with maximal star.

Theorems~\ref{main cmp} and~\ref{main general} (and their proof) provide significant evidence for the following conjecture, which has so far resisted all our attempts at a proof. We briefly illustrate the issue at the end of the introduction: it involves \emph{shortening} the $G$--action on an $\R$--tree (which we were able to do), while \emph{not lengthening} finitely many other $G$--trees (which appears to be quite delicate). 

\begin{conj*}
Let $G$ be a special group.
\begin{enumerate}
\item The DLS automorphisms appearing in Theorem~\ref{main cmp} generate a finite-index subgroup of $\out_{\rm cmp}(G)$. Moreover, finitely many such automorphisms suffice.
\item The DLS automorphisms appearing in Theorems~\ref{main cmp} and~\ref{main general} generate a finite-index subgroup of $\out(G)$. Moreover, finitely many such automorphisms suffice.
\end{enumerate}
In particular, $\out(G)$ and $\out_{\rm cmp}(G)$ are finitely generated.
\end{conj*}

For one-ended hyperbolic groups $G$, one can give a much more precise description of $\out(G)$: up to passing to finite index, it has a free abelian normal subgroup whose quotient is a finite product of mapping class groups \cite{Sela-GAFAII,Levitt-hyp}. A similar description also holds in the toral relatively hyperbolic case \cite{Guirardel-Levitt-GGD15}. 

The usual approach to these results relies heavily on the existence of JSJ decompositions. For general special groups $G$, it appears that one would need to consider splittings over a class of groups that is not closed under taking subgroups, so JSJ techniques seem rather hard to apply.

\medskip
Finally, we would like to highlight the following result, which is required in the proof of Theorem~\ref{main cmp}. It characterises which DLS automorphisms of a special group preserve the coarse median structure, provided that they originate from splittings over convex-cocompact subgroups. For a more general statement on (possibly non-special) cocompactly cubulated groups, see Theorem~\ref{CMP GDT 2}.

\begin{thmintro}\label{CMP GDT}
Let $G$ be a special group with a splitting $G=A\ast_CB$ or $G=A\ast_C$.
% such splittings *always* exist, because RAAGs act properly on products of trees; however, they might not give rise to automorphisms of G in general, since the centraliser of C can be trivial.
Suppose that $C$ is convex-cocompact in $G$. Then:
\begin{enumerate}
\item all partial conjugations and folds determined by this splitting are coarse-median preserving;
\item if $G=A\ast_C$ and $z\in Z_C(C)$ is such that $\langle z\rangle$ is convex-cocompact in $G$ and $Z_G(z)$ is contained in a conjugate of $A$, then the twist determined by $z$ is coarse-median preserving;
\item more generally, if for every $c\in Z_C(C)\setminus\{1\}$ the centraliser $Z_G(c)$ is contained in a conjugate of $A$, then all transvections determined by $G=A\ast_C$ are coarse-median preserving.
\end{enumerate}
\end{thmintro}

For instance, all DLS automorphisms determined by \emph{acylindrical} splittings of $G$ (over convex-cocompact subgroups) are coarse-median preserving. 
%In particular, this applies to all twists associated to cyclic splittings of $G$ over Morse subgroups.
By contrast, the reader can easily check that the twist of $\Z^2=\langle a,b\rangle$ fixing $a$ and mapping $b\mapsto ab$ is not coarse-median preserving (here $A=C=\langle a\rangle$).

Theorem~\ref{CMP GDT} greatly expands the class of automorphisms to which the techniques of \cite{Fio10a} can be applied. In particular, if $\varphi$ is a product of DLS automorphisms as in Theorem~\ref{CMP GDT}, then the subgroup $\Fix\varphi\leq G$ is finitely generated, undistorted, and cocompactly cubulated (see \cite[Theorem~B]{Fio10a}).

\medskip
{\bf On the proof of Theorems~\ref{main cmp} and~\ref{main general}.} Let $G$ be a special group with an infinite sequence of automorphisms $\phi_n\in\out(G)$. The core of the proof lies in the construction of a non-elliptic, stable $G$--tree $T_{\om}$ with ``nice'' arc-stabilisers. From there, the conclusion is technical, but relatively straightforward given the work of Bestvina--Feighn \cite{BF-stable} and Guirardel \cite{Gui-CMH,Gui-Fou}.

Arc-stabilisers of $T_{\om}$ will be centralisers when the $\phi_n$ are coarse-median preserving, and kernels of homomorphisms from centralisers to $\R$ in the general case. Here the expression ``centraliser'' always refers to centralisers of finite subsets of $G$. These properties ensure that $T_{\om}$ is stable, even though arc-stabilisers can be infinitely generated in general. In addition, $T_{\om}$ will not normally be superstable, and tripod-stabilisers can always be nontrivial centralisers.

The mere construction of the tree $T_{\om}$ is straightforward (for instance, it already appears in \cite{Gen-Rtrees}). What requires new ideas is analysing its arc-stabilisers, as we now discuss.

Fixing a convex-cocompact embedding in a right-angled Artin group $G\hookrightarrow\A_{\G}$, we obtain a $G$--action on the universal cover $\X_{\G}$ of the Salvetti complex. It is known that $\X_{\G}$ equivariantly embeds in a finite product of simplicial trees $\prod_{v\in\G}\T_v$, so we obtain a proper action of $G$ on this product. It follows that there exists $v\in\G$ such that the twisted trees $\T_v^{\phi_n}$ diverge, and we can define $T_{\om}$ as an ultralimit of these trees, suitably rescaled.

Arc-stabilisers for the actions $G\acts\T_v$ are quite nice: they are convex-cocompact in $G$, and they are the intersection between $G$ and the centraliser of a subset of $\A_{\G}$. However, they are usually not centralisers \emph{of subsets of $G$}. As a consequence, the moment we start twisting by automorphisms of $G$, we lose all control over their images, which can for instance stop being convex-cocompact in $G$. This compromises the study of arc-stabilisers of $T_{\om}$, since we cannot control those of the trees $\T_v^{\phi_n}$.

The key observation (Theorem~\ref{main perturbation}) is that sufficiently long arcs of $\T_v$ can be \emph{perturbed} so that their $G$--stabiliser (and even their ``almost-stabiliser'') becomes a centraliser in $G$. This rescues us, as automorphisms of $G$ will take centralisers to centralisers, and $\om$--intersections of centralisers are again centralisers.

The main steps of the proof are taken in Sections~\ref{perturbation sect}--\ref{Rips sect}. Section~\ref{perturbation sect} proves Theorem~\ref{main perturbation} on perturbations of arcs of $\T_v$. Section~\ref{limit sect} (and in particular Subsection~\ref{limit subsec}) uses this to obtain all necessary information on arc-stabilisers of $T_{\om}$. Finally, Section~\ref{Rips sect} considers geometric trees approximating $T_{\om}$, applies the Rips machine (blackboxed), and draws the required conclusions.

\medskip
{\bf On the Conjecture.} The classical Rips--Sela argument for hyperbolic groups is based on a well-known \emph{shortening argument} \cite{Rips-Sela,Reinfeldt-Weidmann}. We would like to emphasise that the tree $T_{\om}$ mentioned above can indeed always be shortened by a DLS automorphism of the form described in the statement of the two theorems. This requires a significant amount of work, which we have chosen to omit from this article, as it can be circumvented for a more direct proof of our main results.

The reason why the Conjecture remains unproven is that shortening a \emph{single} tree no longer suffices in this context. Recall that $G$ acts properly on the finite product of trees $\prod_{v\in\G}\T_v$. Each $\T_v$ gives rise to a (possibly elliptic) tree $T_{\om}(v)$ as above. When we shorten some $T_{\om}(v)$ by a DLS automorphism, it is possible that some other $T_{\om}(w)$ will ``get longer'', which deals a serious blow to this kind of approach. 

Excluding this eventuality would require some compatibility conditions between the trees $T_{\om}(v)$. For instance, if $G$ is a surface group and the $\phi_n$ are powers of a pseudo-Anosov with stable and unstable trees $T_{\pm}$, it seems that we cannot hope to shorten $T_+$ without lengthening $T_-$.

Perhaps a more successful strategy would be based on a theory of cospecial actions on median spaces, the first promising steps of which were taken in \cite[Subsection~9.4]{CRK2}.

\medskip
{\bf Structure of the paper.} Section~\ref{prelim sect} contains basic information on $\CAT$ cube complexes, coarse median groups and ultralimits. Within it, Subsection~\ref{cc subsect} proves a few new results on convex-cocompactness in cube complexes, though these will certainly not surprise experts. Section~\ref{special sect} studies centralisers in special groups and the kernels of their homomorphisms to abelian groups.

As discussed above, the proof of Theorems~\ref{main cmp} and~\ref{main general} is spread out over Sections~\ref{perturbation sect}--\ref{Rips sect}. The final argument and the proof of Corollary~\ref{infinite out finite cmp} are given at the end of Subsection~\ref{main argument sect}.

Theorem~\ref{CMP GDT} is proved in Section~\ref{CMP GDT sect}. For the latter, an ingredient we find of particular interest is the discussion of Guirardel cores of products of cube complexes in Subsection~\ref{Guirardel core sect}.

\medskip
{\bf Acknowledgements.} I warmly thank Daniel Groves for the many stimulating conversations on the content of this paper and the Conjecture. I am grateful to Anthony Genevois for pointing me to \cite{Crisp-Wiest}, to Vincent Guirardel for dispelling some of my misconceptions on geometric trees, and to Michah Sageev for suggesting the first paragraph of the proof of Proposition~\ref{splitting euclidean factor}. I also thank Montserrat Casals-Ruiz, Michael Hull, Gilbert Levitt, Ashot Minasyan and Zlil Sela for useful comments and interesting discussions related to this work. Finally, I am particularly grateful to the referee for their many helpful suggestions.

I thank the Max Planck Institute for Mathematics in Bonn, the University of Bonn and Ursula Hamenst\"adt for their hospitality and financial support while this work was being completed.

\tableofcontents

\section{Preliminaries.}\label{prelim sect}

\subsection{$\CAT$ cube complexes.}\label{CCC sect}

We refer the reader to \cite{CS,Sageev-notes,CFI,Fio1} for basic facts on $\CAT$ cube complexes. Here we simply fix terminology and notation, and recall a few standard results. Some of these are relevant also in the $1$--dimensional case of simplicial trees.

Let $X$ be a $\CAT$ cube complex.

\subsubsection{Halfspaces and hyperplanes.}

We denote by $\mscr{W}(X)$ and $\mscr{H}(X)$, respectively, the set of hyperplanes and halfspaces of $X$. If $\mf{h}$ is a halfspace, $\mf{h}^*$ denotes its complement. Two hyperplanes are \emph{transverse} if they are distinct and meet. Halfspaces $\mf{h},\mf{k}$ are transverse if they are bounded by transverse hyperplanes; equivalently, all four intersections $\mf{h}\cap\mf{k}$, $\mf{h}^*\cap\mf{k}$, $\mf{h}\cap\mf{k}^*$, $\mf{h}^*\cap\mf{k}^*$ are nonempty. We also say that a hyperplane is transverse to a halfspace if it is transverse to the hyperplane bounding it. Subsets $\mc{U},\mc{V}\cu\mscr{W}(X)$ are transverse if every element of $\mc{U}$ is transverse to every element of $\mc{V}$. 

If $A$ and $B$ are sets of vertices, $\mscr{H}(A|B)\cu\mscr{H}(X)$ is the subset of halfspaces $\mf{h}$ such that $A\cu\mf{h}^*$ and $B\cu\mf{h}$. Similarly, $\mscr{W}(A|B)\cu\mscr{W}(X)$ is the set of hyperplanes bounding the elements of $\mscr{H}(A|B)$. We say that the elements of $\mscr{W}(A|B)$ \emph{separate} $A$ and $B$.

\subsubsection{Metrics and geodesics.}

We always endow $X$ with its \emph{$\ell^1$ metric} (denoted by $d$), rather than the $\CAT$ metric. We will only be interested in distances between vertices of $X$ (possibly after passing to its cubical subdivision), in which case the $\ell^1$ metric coincides with the intrinsic path metric of the $1$--skeleton. The latter is also known as the \emph{combinatorial metric}. If $x$ and $y$ are vertices of $X$, we have $d(x,y)=\#\mscr{W}(x|y)$.

All geodesics in $X$ are implicitly assumed to be combinatorial geodesics contained in the $1$--skeleton and with their endpoints at vertices. For such a geodesic $\alpha$, we denote by $\mscr{W}(\alpha)\cu\mscr{W}(X)$ the set of hyperplanes dual to the edges of $\alpha$. We say that these are the hyperplanes \emph{crossed} by $\alpha$. We write $\ell(\alpha)$ for the \emph{length} of $\alpha$, which coincides with the cardinality of $\mscr{W}(\alpha)$.

\subsubsection{Convexity.}

If $Y\cu X$ is a convex subcomplex, we do not distinguish between hyperplanes of $Y$ and hyperplanes of $X$ separating vertices of $Y$. The set of such hyperplanes is denoted $\mscr{W}(Y)$. If $Y\cu X$ is a convex subcomplex, we denote its \emph{gate-projection} by $\pi_Y\colon X\ra Y$. For every vertex $x\in X$, the image $\pi_Y(x)$ is a vertex of $Y$ and it is the unique point of $Y$ that is closest to $x$. Gate-projections are $1$--Lipschitz and satisfy $\mscr{W}(x|\pi_Y(x))=\mscr{W}(x|Y)$.

If $Y,Z\cu X$ are convex subcomplexes, we say that $y\in Y$ and $z\in Z$ form a \emph{pair of gates} if $d(y,z)=d(Y,Z)$. Equivalently, $\pi_Y(z)=y$ and $\pi_Z(y)=z$, or again $\mscr{W}(y|z)=\mscr{W}(Y|Z)$. The projections $\pi_Y(Z)$ and $\pi_Z(Y)$ are also convex subcomplexes.

Note that all hyperplanes $\mf{w}\in\mscr{W}(X)$ are convex subcomplexes of the cubical subdivision of $X$. For this reason, they have a cellular structure that makes them into lower-dimension $\CAT$ cube complexes. In addition, we can consider the gate-projections $\pi_{\mf{w}}\colon X\ra\mf{w}$. 

\subsubsection{Isometries and actions.}

We denote by $\aut(X)$ the group of \emph{automorphisms} of $X$, i.e.\ isometries that take vertices to vertices. All actions on $X$ are assumed to be by automorphisms without explicit mention.

If $g\in\aut(X)$, we denote its translation length by $\ell_X(g)=\inf_{x\in X}d(x,gx)$. The \emph{minimal set} of $g$ is the subset $\Min(g,X)\cu X$ (or just $\Min(g)$) where $\ell_X(g)$ is realised. We write $\Fix(g,X)\cu X$ for the set of fixed points of $g$.

An action $G\acts X$ is \emph{without inversions} if there do not exist $g\in G$ and $\mf{h}\in\mscr{H}(X)$ with $g\mf{h}=\mf{h}^*$. Note that $\aut(X)$ acts on the cubical subdivision of $X$ without inversions. Given an action without inversions $G\acts X$, every $g\in G$ contains at least one vertex of $X$ in its minimal set. In particular, if $g$ does not have fixed vertices, then it admits an \emph{axis}: a $\langle g\rangle$--invariant geodesic along which $g$ translates nontrivially \cite{Haglund}.

A hyperplane $\mf{w}\in\mscr{W}(X)$ is \emph{skewered} by $g\in\aut(X)$ if it bounds a halfspace $\mf{h}$ with $g\mf{h}\subsetneq\mf{h}$. Given an action $G\acts X$, we keep the notation from \cite{Fio10b} and write:
\begin{align*}
\mc{W}_1(G,X)&:=\{\mf{w}\in\mscr{W}(X) \mid \text{$\exists g\in G$ skewering $\mf{w}$}\}, \\
\overline{\mc{W}}_0(G,X)&:=\{\mf{w}\in\mscr{W}(X) \mid \text{$\forall g\in G$, either $g\mf{w}=\mf{w}$, or $g\mf{w}$ is transverse to $\mf{w}$}\}.
\end{align*}
We write $\mc{W}_1(G)$ and $\overline{\mc{W}}_0(G)$ when the ambient cube complex is clear, or $\mc{W}_1(g)$ and $\overline{\mc{W}}_0(g)$ if $G=\langle g\rangle$. Note that a hyperplane in $\mc{W}_1(g)$ might only be skewered by a \emph{power} of $g$.

Consider an action $G\acts X$. We say that $X$ is \emph{$G$--essential} if $\mc{W}_1(G)=\mscr{W}(X)$. We say that $X$ is simply \emph{essential} if no halfspace of $X$ is at finite Hausdorff distance from the hyperplane bounding it. When $G$ acts cocompactly on $X$, these two notions of essentiality coincide.

We say that $X$ is \emph{$G$--hyperplane-essential} if every hyperplane $\mf{w}\in\mscr{W}(X)$ is $G_{\mf{w}}$--essential with its induced cubical structure. Here $G_{\mf{w}}$ denotes the subgroup of $G$ leaving $\mf{w}$ invariant. Again, we say that $X$ is simply \emph{hyperplane-essential} if every hyperplane of $X$ is an essential cube complex with its induced cubical structure. As before, $X$ is $G$--hyperplane-essential if and only if $X$ is hyperplane-essential, provided that $G$ acts cocompactly on $X$ (see e.g.\ \cite[Lemma~2.3]{FH}).

Given an action $G\acts X$ and a hyperplane $\mf{w}$ lying neither in $\mc{W}_1(G)$ nor in $\overline{\mc{W}}_0(G)$, exactly one of the two halfspaces bounded by $\mf{w}$ contains an entire $G$--orbit in $X$. Taking the intersection of all such halfspaces, we obtain a $G$--invariant convex subcomplex of $X$, which is nonempty as soon as $G$ satisfies weak assumptions. As a consequence, we obtain the following result (see Remark~3.16, Theorem~3.17, Proposition~3.23(2) and Corollary~4.6 in \cite{Fio10b} for more details).

\begin{prop}\label{convex core prop}
If $G\leq\aut(X)$ is finitely generated and acts on $X$ without inversions, then there exists a nonempty, $G$--invariant, convex subcomplex $\overline{\mc{C}}(G,X)\cu X$ such that:
\begin{enumerate}
\item there is a $G$--invariant splitting $\overline{\mc{C}}(G,X)=\overline{\mc{C}}_0(G,X)\x\mc{C}_1(G,X)$, where the sets of hy\-per\-planes dual to the two factors are precisely $\overline{\mc{W}}_0(G,X)$ and $\mc{W}_1(G,X)$;
\item the action $G\acts\overline{\mc{C}}_0(G,X)$ has fixed vertices, whereas $\mc{C}_1(G,X)$ is $G$--essential;
\item if $h\in\aut(X)$ normalises $G$, then $h$ preserves $\overline{\mc{C}}(G,X)$ and leaves invariant its two factors.
\end{enumerate}
\end{prop}

Again, we simply write $\overline{\mc{C}}(G)$ when the cube complex $X$ is clear, and $\overline{\mc{C}}(g)$ if $G=\langle g\rangle$.

\subsubsection{Median subalgebras and median morphisms.}

A \emph{median algebra} is a set $M$ equipped with a ternary operation $m\colon M^3\ra M$ invariant under permutations and satisfying:
\begin{itemize}
\item $m(a,a,b)=a$ for $a,b\in M$;
\item $m(m(a,x,b),x,c)=m(a,x,m(b,x,c))$ for $a,b,c,x\in M$.
\end{itemize}
A \emph{median subalgebra} is a subset $N\cu M$ with $m(N\x N\x N)\cu N$. A subset $C\cu M$ is \emph{convex} if $m(C\x C\x M)\cu C$. A subset $\mf{h}\cu M$ is a \emph{halfspace} if both $\mf{h}$ and its complement $\mf{h}^*:=M\setminus\mf{h}$ are convex and nonempty. A \emph{wall} of $M$ is an unordered pair $\{\mf{h},\mf{h}^*\}$, where $\mf{h}$ is a halfspace. Let $\mscr{W}(M)$ and $\mscr{H}(M)$ denote, respectively, the sets of walls and halfspaces of $M$.

If $A\cu M$ is a subset, we denote by $\langle A\rangle\cu M$ the median subalgebra generated by $A$. This is the intersection of all median subalgebras of $M$ that contain $A$.

Every $\CAT$ cube complex $X$ has a natural structure of median algebra given by its \emph{median operator} $m\colon X^3\ra X$. If $x,y,z\in X$ are vertices, $m(x,y,z)$ is also a vertex and it is uniquely determined by the following property: a halfspace of $X$ contains $m(x,y,z)$ if and only if it contains at least two among $x,y,z$. The definitions of convexity, halfspaces and hyperplanes/walls coincide for a cube complex $X$ and the median-algebra structure on its $0$--skeleton. In addition, note that the map $m\colon X^3\ra X$ is $1$--Lipschitz.

A map $f\colon M\ra N$ between median algebras is a \emph{median morphism} if, for all $x,y,z\in M$, we have
 $f(m(x,y,z))=m(f(x),f(y),f(z))$. If $f$ is a median morphism, then preimages of convex subsets are again convex. In particular, if $f$ is onto, preimages of halfspaces are halfspaces. When $f$ is onto, it also takes convex subsets to convex subsets. 
 
 If $X$ is a $\CAT$ cube complex and $C\cu X$ is a convex subcomplex, then the gate-projection to $C$ is a median morphism.

A median algebra is \emph{discrete} if any two of its points are separated by only finitely many walls. It was shown by Chepoi \cite{Chepoi} and later by Roller \cite{Roller} that every discrete median algebra is canonically isomorphic to the $0$--skeleton of a unique $\CAT$ cube complex. We will rely on this fact repeatedly in Subsection~\ref{Guirardel core sect}, referring to it as \emph{Chepoi--Roller duality}.

\subsubsection{Restriction quotients.}

Consider an action $G\acts X$ and a $G$--invariant set of hyperplanes $\mc{U}\cu\mscr{W}(X)$. Then there exists a unique action on a $\CAT$ cube complex $G\acts X(\mc{U})$ satisfying the following properties:
\begin{itemize}
\item there is a $G$--equivariant, surjective, median morphism $p\colon X\ra X(\mc{U})$;
\item if $\mf{h}$ is a halfspace of $X(\mc{U})$, then $p^{-1}(\mf{h})$ is a halfspace of $X$ bounded by a hyperplane in $\mc{U}$;
\item this establishes a $G$--equivariant bijection between hyperplanes of $X(\mc{U})$ and elements of $\mc{U}$. 
\end{itemize}
The cube complex $X(\mc{U})$ is known as the \emph{restriction quotient} of $X$ associated to $\mc{U}$. Restriction quotients were introduced by Caprace and Sageev in \cite[p.\ 860]{CS}.

\subsection{Euclidean factors.}

The goal of this subsection is to prove the following result, which will only be required in Subsection~\ref{earthquake cmp sect} in order to prove Theorem~\ref{CMP GDT}.

\begin{prop}\label{splitting euclidean factor}
Consider a product of $\CAT$ cube complexes $X\x L$, where $L$ is a quasi-line. If $G$ acts properly, cocompactly and faithfully on $X\x L$ preserving the splitting, then $G$ has a finite-index subgroup of the form $H\x\Z$, where $H$ acts trivially on $L$ and the $\Z$--factor acts trivially on $X$.
\end{prop}

This is similar to \cite[Corollary~2.8]{Nevo-Sageev}, whose proof however relies on \cite[Lemma~2.7]{Nevo-Sageev}, which appears to be false (if ``flat'' is to be interpreted in the $\CAT$ sense). For instance, consider the quasi-line obtained by stringing together countably many squares diagonally to form a chain, whose automorphism group is not discrete (it contains a direct product of countably many copies of $\Z/2\Z$). 

For this reason, we give an alternative proof here, based on the following two lemmas. Note that the first lemma can fail if we replace \emph{automorphisms} of $X$ with \emph{isometries} (e.g.\ for $X=\R^2$).

\begin{lem}\label{fi norm}
Let $G\leq\aut(X)$ act properly and cocompactly on the $\CAT$ cube complex $X$. Then $G$ has finite index in its normaliser within $\aut(X)$.
\end{lem}
\begin{proof}
Let $N$ be the normaliser of $G$ in $\aut(X)$. Fix a vertex $p\in X$ and let $N_p$ be the subgroup of $N$ fixing it. Since $N$ permutes the finitely many $G$--orbits of vertices, $N$ has a finite index subgroup of the form $G\cdot N_p$. We will prove the lemma by showing that $N_p$ is finite.

If $g\in G$ and $n\in N_p$, then $d(ngn^{-1}p,p)=d(gp,p)$. Since $G$ acts properly on $X$, all orbits of the conjugation action of $N_p$ on $G$ must be finite. Hence, since $G$ is finitely generated by the Milnor--Schwarz lemma, a finite-index subgroup of $N_p$ commutes with $G$. 

Now, let $F\cu X$ be a finite set of vertices meeting all $G$--orbits. Since $X$ is locally finite and $N_p$ takes vertices to vertices, a finite-index subgroup $N_0\leq N_p$ fixes $F$ pointwise. By the above paragraph, we can choose $N_0$ so that it commutes with $G$. Given $f\in F$, $g\in G$ and $n\in N_0$, we have $n\cdot gf=ngn^{-1}\cdot nf=gf$. This shows that $N_0$ fixes the $0$--skeleton of $X$ pointwise, so it is the trivial group. Since $N_0$ has finite index in $N_p$, this proves that $N_p$ is finite, as required.
\end{proof}

\begin{lem}\label{discrete factor}
Consider a product of $\CAT$ cube complexes $X\x Y$. Let $G$ act properly, cocompactly and faithfully on $X\x Y$ preserving the factors. If the image of $G$ in $\aut(Y)$ is discrete, then $G$ has a finite-index subgroup of the form $H\x K$, where $H$ acts trivially on $Y$ and $K$ acts trivially on $X$.
\end{lem}
\begin{proof}
Let $\rho_X\colon G\ra\aut(X)$ and $\rho_Y\colon G\ra\aut(Y)$ be the homomorphisms corresponding to the actions on the two factors. Note that, since $X\x Y$ admits a proper cocompact action, it is locally finite; in particular, $Y$ is locally finite. Thus, since $\rho_Y(G)$ is discrete, it acts on $Y$ with finite vertex-stabilisers. This shows that $\ker\rho_Y$ is commensurable to the $G$--stabiliser of a vertex of $Y$, so $\ker\rho_Y$ acts cocompactly on $X$. By Lemma~\ref{fi norm}, $\rho_X(\ker\rho_Y)$ has finite index in $\rho_X(G)$. This shows that the subgroup $\rho_X^{-1}\rho_X(\ker\rho_Y)=\ker\rho_X\cdot\ker\rho_Y$ has finite index in $G$. Since both kernels are normal in $G$ and they have trivial intersection, this is a direct product.
\end{proof}

The first paragraph of the following proof was suggested to me by Michah Sageev, as it has the advantage of only requiring basic $\CAT$ geometry. Alternatively, one can also use panel collapse \cite[Theorem~A]{Hagen-Touikan} to find a subcomplex of $L$ isomorphic to $\R$.

\begin{proof}[Proof of Proposition~\ref{splitting euclidean factor}]
Let $P\cu L$ be the union of all geodesic lines in $L$ for the $\CAT$ metric. By \cite[Theorem~II.2.14]{BH}, we have a $G$--invariant splitting $P=P_0\x\R$, where $P_0$ is compact. Thus, since $G$ must fix a point of $P_0$, there exists a $G$--invariant $\CAT$--line $L_0\cu L$. Note that $G$ acts on $L_0$ with discrete orbits (e.g.\ because the projection to $L_0$ of the set of vertices of $L$ in a bounded neighbourhood of $L_0$ is $G$--invariant, and $L$ is locally finite). 

Now, since $G$ acts discretely on $L_0\simeq\R$, we can apply Lemma~\ref{discrete factor} to the $G$--action on $X\x L_0$ (modulo its finite kernel). It follows that the image of $G$ in $\aut X$ is discrete and so we can apply Lemma~\ref{discrete factor} again, this time to the whole product $X\x L$. This yields the required conclusion.
%Using panel collapse \cite[Theorem~A]{Hagen-Touikan}, we obtain a $G$--invariant subset $L_0\cu L$ admitting a structure of essential, hyperplane-essential $\CAT$ cube complex with the same $0$--skeleton as a subcomplex of the cubical subdivision of $L$. Since hyperplanes of quasi-lines are bounded and $L_0$ is a hyperplane-essential quasi-line, $L_0$ must be isomorphic to the standard cubulation of $\R$.
%
%Now, since $\aut L_0$ is discrete, we can apply Lemma~\ref{discrete factor} to the $G$--action on $X\x L_0$ (modulo its finite kernel). It follows that the image of $G$ in $\aut X$ is discrete and so we can apply Lemma~\ref{discrete factor} again, this time to the whole product $X\x L$. This yields the required conclusion.
\end{proof}

\subsection{Convex-cocompactness.}\label{cc subsect}

Fix a proper cocompact action without inversions on a $\CAT$ cube complex $G\acts X$ throughout this subsection.
% careful about inversions!

\begin{defn}\label{cc defn}
A subgroup $H\leq G$ is \emph{convex-cocompact} with respect to the action $G\acts X$ (or just \emph{in $X$}) if there exists an $H$--invariant convex subcomplex $Y\cu X$ on which $H$ acts cocompactly.
\end{defn}

As observed in \cite[Lemma~3.2]{Fio10a}, $H$ is convex-cocompact if and only if the action on $\mc{C}_1(H)$ is cocompact and $H$ is finitely generated. Thus, we can always take $Y$ to be $H$--essential in Definition~\ref{cc defn}, using Proposition~\ref{convex core prop}.

We will often need to quantify convex-cocompactness:

\begin{defn}\label{quant cc defn}
A subgroup $H\leq G$ is \emph{$q$--convex-cocompact} if there exists an $H$--invariant convex subcomplex $Y\cu X$ on which $H$ acts with exactly $q$ orbits of vertices. 
% exactly or at most?
\end{defn}

Since the number of $H$--orbits is minimised by $H$--essential convex subcomplexes, we can always take $Y$ in Definition~\ref{quant cc defn} to be $H$--essential.
% the essential core is a restriction quotient of any $H$--invariant convex subcomplex

\begin{rmk}\label{fi overgroups rmk}
Let $H\leq G$ be $q$--convex-cocompact and let $N$ be the maximum cardinality of the $G$--stabiliser of a vertex of $X$. Then $H$ has index $\leq qN$ in all its finite-index overgroups within $G$.

Indeed, suppose $H$ has finite index $d$ in a subgroup $H'\leq G$. Note that a hyperplane of $X$ is skewered by an element of $H$ if and only if it is skewered by an element of $H'$, so $\mc{C}_1(H)=\mc{C}_1(H')$. Since $\mc{C}_1(H')$ equivariantly embeds in $X$, the $H'$--stabiliser of any vertex of $\mc{C}_1(H')$ has cardinality $\leq N$. Now, if $H'$ acts on $\mc{C}_1(H')$ with $k$ orbits of vertices, then $H$ acts on $\mc{C}_1(H)$ with at least $kd/N$ orbits of vertices, hence $q\geq kd/N\geq d/N$.
\end{rmk}

\begin{lem}\label{cc intersection}
Let $H,K\leq G$ be subgroups that leave invariant convex subcomplexes $Y,Z\cu X$, respectively, and act cocompactly on them. Then $H\cap K$ acts cocompactly on $\pi_Y(Z)$.
\end{lem}
\begin{proof}
We split the proof into the following three claims.

\smallskip
{\bf Claim~1:} \emph{for every ball $B\cu X$, only finitely many distinct $G$--translates of $Y$ and $Z$ meet $B$.}

\smallskip\noindent
\emph{Proof of Claim~1.}
Suppose this is not the case for a ball $B\cu X$. Then, since $B$ contains only finitely many vertices, there are infinitely many, pairwise distinct translates $g_nY$ all containing the same vertex $p\in B$. Since $H\acts Y$ is cocompact, there exists a compact subset $Q\cu Y$ and elements $h_n\in H$ with $h_ng_n^{-1}p\in Q$. Since $G\acts X$ is proper, the set $F=\{h_ng_n^{-1}\}$ is finite. Hence $g_n^{-1}\in h_n^{-1}F$ and $g_n\in F^{-1}\cdot H$, contradicting the fact that the set $\{g_nY\}$ is infinite.  
\hfill$\blacksquare$

\smallskip
{\bf Claim~2:} \emph{for every ball $B\cu X$, only finitely many distinct $G$--translates of $\pi_Y(Z)$ meet $B$.}

\smallskip\noindent
\emph{Proof of Claim~2.}
Consider $g\in G$ such that $g\cdot\pi_Y(Z)=\pi_{gY}(gZ)$ intersects $B$. Then $gY$ intersects $B$, while $gZ$ intersects the neighbourhood of $B$ of radius $d(Y,Z)$. By Claim~1, there are only finitely many possibilities for the sets $gY$ and $gZ$. It follows that only finitely many sets of the form $\pi_{gY}(gZ)$ intersect $B$. 
\hfill$\blacksquare$

\smallskip
Let $L\leq G$ be the $G$--stabiliser of $\pi_Y(Z)$. Claim~2 and \cite[Lemma~2.3]{Hagen-Susse} imply that $L\acts\pi_Y(Z)$ is cocompact.

\smallskip
{\bf Claim~3:} \emph{a finite-index subgroup of $L$ leaves $Y$ and $Z$ invariant.}

\smallskip\noindent
\emph{Proof of Claim~3.}
By Claim~1, only finitely many distinct $G$--translates of $Y$ contain $\pi_Y(Z)$. By the same argument, $Z$ is among the finitely many $G$--translates of $Z$ that contain $\pi_Y(Z)$ in their neighbourhood of radius $d(Y,Z)$. Observing that $L$ permutes these translates of $Y$ and $Z$, we conclude that a finite-index subgroup of $L$ must preserve both $Y$ and $Z$.
\hfill$\blacksquare$

\smallskip
Let $G_Y,G_Z\leq G$ denote the $G$--stabilisers of $Y$ and $Z$. Since $H\acts Y$ is cocompact, $H$ has finite index in $G_Y$. Similarly, $K$ has finite index in $G_Z$. It follows that $H\cap K$ has finite index in $G_Y\cap G_Z$, which has finite index in $L$ by Claim~3. We have already observed that $L\acts\pi_Y(Z)$ is cocompact, so this implies that $H\cap K\acts\pi_Y(Z)$ is cocompact.
\end{proof}

Given a subgroup $H\leq G$, we denote by $N_G(H)\leq G$ its normaliser.

\begin{lem}\label{almost normalisers}
If $H,K\leq G$ are convex-cocompact in $X$, there exists a finite subset $F\cu G$ such that:
\[\{g\in G \mid gHg^{-1}\leq K\} = K\cdot F\cdot N_G(H).\]
\end{lem}
\begin{proof}
Let $\mscr{C}(q)$ be the collection of $q$--convex-cocompact subgroups of $G$. 

\smallskip
{\bf Claim:} \emph{for each $q\geq 1$, only finitely many $K$--conjugacy classes of subgroups of $K$ lie in $\mscr{C}(q)$.}

\smallskip\noindent
\emph{Proof of claim.}
Let $Y\cu X$ be a $K$--invariant, $K$--cocompact, convex subcomplex. Let $Y_0\cu Y$ be a finite set of vertices meeting every $K$--orbit. 

Consider a subgroup $L\leq K$ lying in $\mscr{C}(q)$. Then there exists an $L$--invariant convex subcomplex $Z\cu X$ on which $L$ acts with $\leq q$ orbits of vertices. Replacing $Z$ with its gate-projection to $Y$, we can assume that $Z\cu Y$. Conjugating $L$ by an element of $K$, we can assume that $Z$ meets $Y_0$.

Now, the $q$--neighbourhood of $Y_0$ in $Y$ contains a set of vertices $Z_0\cu Z$ meeting every $L$--orbit in $Z$. By \cite[Theorem~I.8.10]{BH}, $L$ is generated by the elements $\{g\in L\mid d(gZ_0,Z_0)\leq 1\}$. 

Summing up, every subgroup of $K$ lying in $\mscr{C}(q)$ is $K$--conjugate to a subgroup generated by a subset of the finite set $\{g\in G\mid d(gY_0,Y_0)\leq 2q+1\}$. This proves the claim.
\hfill$\blacksquare$

\smallskip
Choose $q'$ such that $H\in\mscr{C}(q')$. Then, for every $g\in G$, we have $gHg^{-1}\in\mscr{C}(q')$. The claim implies that $K$ contains only finitely many subgroups of this form up to $K$--conjugacy, and the lemma follows.
\end{proof}

\begin{defn}
An action on a $\CAT$ cube complex $H\acts X$ is \emph{non-transverse} if there do not exist a hyperplane $\mf{w}\in\mscr{W}(X)$ and an element $h\in H$ such that $\mf{w}$ and $h\mf{w}$ are transverse.
\end{defn}

Recall from Proposition~\ref{convex core prop} that $N_G(H)$ leaves invariant the convex subcomplex $\overline{\mc{C}}(H)$ and its splitting $\overline{\mc{C}}_0(H)\x\mc{C}_1(H)$.

\begin{lem}\label{C_0 cocompact}
Let $H\leq G$ be convex-cocompact in $X$. Suppose that $H$ acts non-transversely on $X$. Then the action $N_G(H)\acts\overline{\mc{C}}_0(H)$ is cocompact.
\end{lem}
% if H does *not* act non-transversely, \cite[Remark~3.9]{Fio10b} still shows that a finite index subgroup of H fixes every hyperplane of \overline{\mc{C}}_0(H), and the normaliser of this finite index subgroup of H *does* act cocompactly on \overline{\mc{C}}_0(H)
\begin{proof}
Let $\mscr{T}(X)$ be the set of tuples $(\mf{w}_1,\dots,\mf{w}_k)$ of pairwise-transverse hyperplanes of $X$. Since $H$ acts non-transversely on $X$, each hyperplane of $\overline{\mc{C}}_0(H)$ is left invariant by $H$. Thus, maximal cubes of $\overline{\mc{C}}_0(H)$ are in one-to-one correspondence with maximal $H$--fixed tuples in $\mscr{T}(X)$.

Let us show that $N_G(H)$ acts cofinitely on the set of fixed points of $H$ in $\mscr{T}(X)$. By the previous paragraph, this implies the lemma.

For every tuple $(\mf{w}_1,\dots,\mf{w}_k)$ in $\mscr{T}(X)$, its stabiliser $G_{\mf{w}_1}\cap\dots\cap G_{\mf{w}_k}$ acts cocompactly on the intersection $\mf{w}_1\cap\dots\cap\mf{w}_k$ (see e.g.\ \cite[Lemma~2.3]{FH}), so it is convex-cocompact in $X$. Lemma~\ref{almost normalisers} implies that there exists a finite set $F\cu G$ such that:
\[\{g\in G\mid H\text{ preserves } g\mf{w}_1,\dots,g\mf{w}_k\}= N_G(H)\cdot F\cdot (G_{\mf{w}_1}\cap\dots\cap G_{\mf{w}_k}).\]
It follows that every $G$--orbit in $\mscr{T}(X)$ contains only finitely many $N_G(H)$--orbits of elements fixed by $H$. Since the action $G\acts\mscr{T}(X)$ is cofinite, this shows that there are only finitely many $N_G(H)$--orbits of fixed points of $H$ in $\mscr{T}(X)$, as required.
\end{proof}

\begin{ex}
Lemma~\ref{C_0 cocompact} (and Corollary~\ref{cc normalisers}) % except its first part
can fail if $H$ does not act non-transversely. 

For instance, let $G=\Z^2\rtimes\langle h\rangle$ act on the standard cubulation of $\R^3$, with $\Z^2$ generated by unit translations in the $x$-- and $y$--directions, respectively, and $h(x,y,z)=(y,x,z+1)$. Taking $H=\langle h\rangle$, the space $\overline{\mc{C}}_0(H)$ is naturally identified with the $xy$--plane, but $N_G(H)$ is generated by $h$ and $(x,y,z)\mapsto (x+1,y+1,z)$.
\end{ex}

\begin{cor}\label{cc normalisers}
Let $H\leq G$ be convex-cocompact in $X$. If $H\acts X$ is non-transverse, then:
\begin{enumerate}
\item $N_G(H)$ has a finite-index subgroup of the form $H\cdot K$, where $H$ and $K$ commute and $H\cap K$ is finite (thus, if $G$ is virtually torsion-free, $N_G(H)$ is virtually a product $H\x K$);

\smallskip
\item there exists a point $p\in\overline{\mc{C}}(H)$ such that the fibre $\overline{\mc{C}}_0(H)\x\{\ast\}$ through $p$ is $K$--invariant and $K$--cocompact, while the fibre $\{\ast\}\x\mc{C}_1(H)$ through $p$ is $H$--invariant and $H$--cocompact;

\smallskip
\item the action $N_G(H)\acts\overline{\mc{C}}(H)$ is cocompact, hence $N_G(H)$ is convex-cocompact in $X$.
\end{enumerate}
\end{cor}
\begin{proof}
Recall that both $\mc{C}(H)\cu X$ and its splitting $\overline{\mc{C}}_0(H)\x\mc{C}_1(H)$ are preserved by $N_G(H)$. The action $H\acts\overline{\mc{C}}_0(H)$ has a fixed point, so we have an $H$--invariant fibre $\{p_0\}\x\mc{C}_1(H)$. The $H$--action on this fibre is cocompact (see e.g.\ \cite[Lemma~3.2(3)]{Fio10a}) and proper, since it equivariantly embeds in $X$. Let $p=(p_0,p_1)$ be any point in this fibre. 

Consider the proper cocompact action $H\acts\mc{C}_1(H)$. A finite index-subgroup $N\leq N_G(H)$ preserves the $H$--orbit of $p_1$. If $N_1\leq N$ is the subgroup fixing $p_1$, then $N=H\cdot N_1$, the intersection $H\cap N_1$ is finite, and a finite-index subgroup $K\leq N_1$ commutes with $H$. This can all be shown exactly as in the proof of Lemma~\ref{fi norm}. Part~(1) follows.

By Lemma~\ref{C_0 cocompact}, $N_G(H)$ acts cocompactly on $\overline{\mc{C}}_0(H)$. The same holds for the finite-index subgroup $K\cdot H$. Since $H$ is elliptic in $\overline{\mc{C}}_0(H)$, a $K$--orbit coincides with a $K\cdot H$--orbit and so it is coarsely dense in $\overline{\mc{C}}_0(H)$. Note that $\overline{\mc{C}}_0(H)$ is locally finite, since it embeds in $X$. Thus, $K$ acts cocompactly on $\overline{\mc{C}}_0(H)$, hence on the fibre $\overline{\mc{C}}_0(H)\x\{p_1\}$. This proves part~(2), and part~(3) follows immediately.
\end{proof}

\subsection{Coarse medians.}

Coarse medians were introduced by Bowditch in \cite{Bow-cm}. We present the following equivalent definition from \cite{NWZ1}. We write ``$x\approx_Cy$'' with the meaning of ``$d(x,y)\leq C$''.

\begin{defn}\label{coarse median space defn}
Let $X$ be a metric space. A \emph{coarse median} on $X$ is a permutation-invariant map $\mu\colon X^3\ra X$ for which there exists a constant $C\geq 0$ such that, for all $a,b,c,x\in X$, we have:
\begin{enumerate}
\item $\mu(a,a,b)=a$;
\item $\mu(\mu(a,x,b),x,c)\approx_C\mu(a,x,\mu(b,x,c))$;
\item $d(\mu(a,b,c),\mu(x,b,c))\leq Cd(a,x)+C$.
\end{enumerate}
\end{defn}

In accordance with \cite[Subsection~2.6]{Fio10a}, we also introduce the following.

\begin{defn}
Two coarse medians $\mu_1,\mu_2$ are at \emph{bounded distance} if $\mu_1(x,y,z)\approx_C\mu_2(x,y,z)$ for some $C\geq 0$ and all $x,y,z\in X$. A \emph{coarse median structure} on $X$ is the equivalence class $[\mu]$ of coarse medians at bounded distance from $\mu$. A \emph{coarse median space} is a metric space with a coarse median structure.
\end{defn}

\begin{defn}\label{cmp defn}
Let $(X,[\mu])$ be a coarse median space. A coarsely Lipschitz map $f\colon X\ra X$ is \emph{coarse-median preserving} if $f(\mu(x,y,z))\approx_C\mu(f(x),f(y),f(z))$ for some $C\geq 0$ and all $x,y,z\in X$. 
\end{defn}

Recall that $\CAT$ cube complexes have a natural structure of median algebra, hence one of coarse median space. The following is a simple observation.

\begin{lem}\label{cmp criterion}
Let $(X,m)$ be a $\CAT$ cube complex. A map $\Phi\colon X^{(0)}\ra X^{(0)}$ is coarse-median preserving if and only if there exists a constant $C\geq 0$ such that, whenever $x,y,p\in X$ are vertices with $p=m(x,y,p)$, the set $\mscr{W}(\Phi(p)|\Phi(x),\Phi(y))$ contains at most $C$ hyperplanes.
\end{lem}
\begin{proof}
Suppose that $\Phi$ is coarse-median preserving and $C$ is the constant in Definition~\ref{cmp defn}. Then, if $p=m(x,y,p)$, we have $\Phi(p)\approx_C m(\Phi(p),\Phi(x),\Phi(y))$. Hyperplanes separating these two points are precisely those in the set $\mscr{W}(\Phi(p)|\Phi(x),\Phi(y))$, which then has cardinality at most $C$.

Conversely, suppose that $\Phi$ is a map satisfying $\#\mscr{W}(\Phi(p)|\Phi(x),\Phi(y))\leq C$ for all $x,y,p\in X$ with $p=m(x,y,p)$. Consider arbitrary points $x',y',z'\in X$ and their median $m'=m(x',y',z')$. Then the set $\mscr{W}(\Phi(m') | m(\Phi(x'),\Phi(y'),\Phi(z')))$ is contained in the union
\[  \mscr{W}(\Phi(m')| \Phi(x'),\Phi(y'))\cup\mscr{W}(\Phi(m')|\Phi(y'),\Phi(z'))\cup\mscr{W}(\Phi(m')|\Phi(z'),\Phi(x')),\]
where each of the three sets has cardinality at most $C$ by our assumption on $\Phi$. It follows that $\Phi(m')\approx_{3C}m(\Phi(x'),\Phi(y'),\Phi(z'))$, showing that $\Phi$ is coarse-median preserving.
\end{proof}

\begin{defn}
A \emph{coarse median group} is a pair $(G,[\mu])$ where $G$ is a finitely generated group and $[\mu]$ is a coarse median structure with respect to the word metrics on $G$. Contrary to \cite{Bow-cm}, we additionally require all left multiplications by elements of $G$ to be coarse-median preserving.
\end{defn}

Let $(G,[\mu])$ be a coarse median group. Note that all automorphisms of $G$ are quasi-isometries with respect to the word metrics on $G$. We denote the set of coarse-median preserving automorphisms by $\aut(G,[\mu])$, or simply $\aut_{\rm cmp}(G)$ when the coarse median structure is clear. 

Note that $\aut_{\rm cmp}(G)\leq\aut(G)$ is a subgroup containing all inner automorphisms, so it descends to a subgroup $\out_{\rm cmp}(G)\leq\out(G)$.

All hyperbolic groups and mapping class groups are coarse median groups \cite{Bow-cm}. However, the main example of interest for this paper is provided by cocompactly cubulated groups, as this provides structures of coarse median group on all special groups.

\begin{ex}\label{cubical coarse medians}
Every proper cocompact action on a $\CAT$ cube complex $G\acts X$ induces a canonical structure of coarse median group on $G$. It suffices to pull back to $G$ the median operator of $X$ via any $G$--equivariant quasi-isometry $G\ra X$. The result is independent of all choices involved.

Note however that \emph{different actions} on $\CAT$ cube complexes can induce different coarse median structures on $G$. This is particularly evident for free abelian groups $\Z^n$ with $n\geq 2$ (corresponding to changes of basis). An exception is provided by hyperbolic groups, as they always admit a unique coarse median structure (see e.g.\ \cite[Theorem~4.2]{NWZ1}).
\end{ex}

\begin{defn}\label{qc defn}
Let $(X,[\mu])$ be a coarse median space. A subset $A\cu X$ is \emph{quasi-convex} if there exists $C\geq 0$ such that $\mu(A\x A\x X)$ is contained in the $C$--neighbourhood of $A$.
\end{defn}

\begin{rmk}\label{cc vs qc}
Let $G\acts X$ be a proper cocompact action on a $\CAT$ cube complex, and let $[\mu_X]$ be the induced coarse median structure on $G$. Then a subgroup $H\leq G$ is quasi-convex with respect to $[\mu_X]$ if and only if it is convex-cocompact in $X$. See for instance \cite[Lemma~3.2]{Fio10a}.
\end{rmk}

\begin{rmk}
Let $(G,[\mu])$ be a coarse median group. If $H\leq G$ is quasi-convex and $\varphi$ is a coarse-median preserving automorphism of $G$, then $\varphi(H)$ is again quasi-convex.
\end{rmk}

In coarse median groups we also have the following notion of orthogonality of subgroups, which was referenced in the definition of twists and folds in the Introduction.

\begin{defn}\label{pf defn}
Let $(G,[\mu])$ be a coarse median group. Two subgroups $H,K\leq G$ are \emph{orthogonal} (written $H\perp K$ or $H\perp_{[\mu]} K$) if the set $\{\mu(1,h,k) \mid h\in H,\ k\in K\}$ is finite.
\end{defn}
%Notation: always use $[\mu]$ with square brackets as subscript of \perp.

\begin{rmk}\label{pf vs qc rmk}
Orthogonal subgroups have finite intersection. The converse holds for quasi-convex subgroups.
\end{rmk}

\begin{lem}\label{pf lemma}
Suppose that $G$ admits a proper cocompact action on a $\CAT$ cube complex $X$. Let $[\mu_X]$ be the induced coarse median structure.\begin{enumerate}
\item If $H,K\leq G$ commute and $H\perp K$, then $\mc{W}_1(H)\cu\overline{\mc{W}}_0(K)$ and $\mc{W}_1(K)\cu\overline{\mc{W}}_0(H)$.
\item If $H,K\leq G$ are as in Corollary~\ref{cc normalisers}, then $H\perp K$.
\end{enumerate}
\end{lem}
\begin{proof}
Part~(2) is immediate from Corollary~\ref{cc normalisers}(2) and the definition of orthogonality. Regarding part~(1), it suffices to show that, for every $h\in H$, we have $\mc{W}_1(h)\cu\overline{\mc{W}}_0(K)$. 

Since $h$ and $K$ commute, we have $k\mc{W}_1(h)=\mc{W}_1(h)$ for every $k\in K$. In addition, for every $\mf{w}\in\mc{W}_1(h)$, each $k\in K$ takes the side of $\mf{w}$ containing a positive semi-axis of $h$ to the side of $k\mf{w}$ containing a positive semi-axis of $h$. Thus, either $k\mf{w}$ and $\mf{w}$ intersect,
% i.e. transverse *or* equal
or $k$ skewers $\mf{w}$.

If no element of $K$ skewers an element of $\mc{W}_1(h)$, this shows that $\mc{W}_1(h)\cu\overline{\mc{W}}_0(K)$, as required. If instead some $k\in K$ skewers a hyperplane $\mf{w}\in\mc{W}_1(h)$, then $\langle k\rangle\cdot\mf{w}\cu\mc{W}_1(h)\cap\mc{W}_1(k)$. In this case, $\mc{W}_1(h)\cap\mc{W}_1(k)$ is infinite, so $\mu(1,h^n,k^n)$ diverges for $n\ra+\infty$, violating the fact that $H\perp K$.
\end{proof}

\subsection{Ultralimits.}\label{ultrafilter sect}

For a detailed treatment of ultrafilters and ultralimits, the reader can consult \cite[Chapter~10]{DK}. Here we briefly recall only one basic construction.

Fix a non-principal ultrafilter $\om$ on $\N$. Consider a sequence $G\acts X_n$ of isometric actions on metric spaces, with a sequence of basepoints $o_n\in X_n$. Let $S\cu G$ be a finite generating set.

We say that the sequence $(G\acts X_n,o_n)$ \emph{$\om$--converges} if, for every generator $s\in S$, we have $\lim_{\om}d(o_n,so_n)<+\infty$. In this case, the \emph{$\om$--limit} is the isometric action $G\acts X_{\om}$ constructed as follows. Points of $X_{\om}$ are sequences $(x_n)$ with $x_n\in X_n$ and $\lim_{\om}d(x_n,o_n)<+\infty$, where we identify sequences $(x_n)$ and $(x_n')$ if $\lim_{\om}d(x_n,x_n')=0$. The $G$--action on $X_{\om}$ is defined by $g(x_n):=(gx_n)$.

If a sequence of actions on $\R$--trees $G\acts T_n$ $\om$--converges to an action $G\acts T_{\om}$ (for some choice of basepoints), then $T_{\om}$ is a complete $\R$--tree. Note that the action $G\acts T_{\om}$ will almost always fail to be minimal, even if all actions $G\acts T_n$ are. 

This construction will play a major role in Subsections~\ref{tame sect} and~\ref{limit subsec}.

\section{Special groups and RAAGs.}\label{special sect}

A group is usually said to be \emph{special} if it is the fundamental group of a compact special cube complex \cite{Haglund-Wise-GAFA,Sageev-notes}. For our purposes, it is more convenient to use the following, entirely equivalent characterisation.

\begin{defn}
A group $G$ is \emph{special} if and only if $G$ is a convex-cocompact subgroup of a right-angled Artin group $\A_{\G}$ with respect to the action on the universal cover of the Salvetti complex.
\end{defn}

Note that special groups are torsion-free.

\subsection{Notation and basic properties.}\label{special sect prelim}

In the rest of the paper, we employ the following notation.
\begin{itemize}
\item We denote right-angled Artin groups by $\A_{\G}$ and universal covers of Salvetti complexes by $\X_{\G}$. As customary, we identify the $0$--skeleton of $\X_{\G}$ with $\A_{\G}$.
\item We have a map $\g\colon\mscr{W}(\X_{\G})\ra\G^{(0)}$ that pairs each hyperplane of $\X_{\G}$ with its label.
\item For each $v\in\G^{(0)}$, we denote by $\pi_v\colon\X_{\G}\ra\T_v$ the restriction quotient associated to the set of hyperplanes $\g^{-1}(v)\cu\mscr{W}(\X_{\G})$. This is a simplicial tree with an $\A_{\G}$--action.
\item If $g\in\A_{\G}$, we denote by $\G(g)\cu\G$ the set of labels appearing on one (equivalently, all) axis of $g$ in $\X_{\G}$.
%\G(g) is empty iff g=1
Equivalently, $\G(g)$ is the set of $v\in\G$ for which $g$ is loxodromic in the tree $\T_v$.
Note that $\G(g)\cu\g(\mscr{W}(1|g))$, though this might not be an equality if $g$ is not cyclically reduced.
\item If $K\leq\A_{\G}$ is a subgroup, we also write $\G(K):=\bigcup_{g\in K}\G(g)$.
\item We do not distinguish between subgraphs $\Delta\cu\G$ and their $0$--skeleton. If $\Delta\cu\G$, we write: 
\begin{align*}
\Delta^{\perp}&=\bigcap_{v\in\Delta}\lk v, & \Delta_{\perp}&=\bigcap_{v\in\Delta}\St v.
\end{align*}
\end{itemize}

\begin{rmk}\label{perp rmk}
\begin{enumerate}
\item[]
\item For every $\Delta\cu\G$, the centraliser of $\A_{\Delta}$ in $\A_{\G}$ is $\A_{\Delta_{\perp}}$. 
\item We have $\Delta_{\perp}=\Delta^{\perp}\sqcup\{c_1,\dots,c_k\}$, where the $c_i$ are those vertices of $\Delta$ such that $\Delta\cu\St c_i$.
\end{enumerate} 
\end{rmk}

We record here a few basic lemmas for later use.

\begin{lem}\label{aligned2}
Consider $a,b\in\A_{\G}$ such that $1,a,ab$ lie on a geodesic of $\X_{\G}$ in this order. Then:
\[\g(\mscr{W}(1|a))\cap\g(\mscr{W}(1|b))\cu\G(a)\cup\G(b)\cup\G(ab).\]
\end{lem}
\begin{proof}
Consider $v\in\g(\mscr{W}(1|a))\cap\g(\mscr{W}(1|b))$ and suppose that $v\not\in\G(a)\cup\G(b)$. Write $a=xa'x^{-1}$ and $b=yb'y^{-1}$ as reduced words, with $a',b'$ cyclically reduced. Since $1,a,ab$ lie on a geodesic, the word $xa'x^{-1}yb'y^{-1}$ spells a geodesic in $\X_{\G}$.

Since $v\not\in\G(a)\cup\G(b)$, we must have $v\in\g(\mscr{W}(1|x))\cap\g(\mscr{W}(1|y))$. Thus, there exist halfspaces:
\begin{align*}
\mf{h}_1&\in\mscr{H}(1|x), & \mf{h}_2&\in\mscr{H}(xa'x^{-1}yb'|xa'x^{-1}yb'y^{-1}),
\end{align*}
bounded by hyperplanes labelled by $v$. Since $xa'x^{-1}yb'y^{-1}$ spells a geodesic, we have $\mf{h}_2\subsetneq\mf{h}_1$. Note that $ab\cdot\mf{h}_1$ lies in $\mscr{H}(xa'x^{-1}yb'y^{-1}|xa'x^{-1}yb'y^{-1}x)$. In addition, since $x^{-1}y$ is a sub-path of a geodesic, it is itself a geodesic, hence $y^{-1}x$ also spells a geodesic. This shows that $ab\cdot\mf{h}_1\subsetneq\mf{h}_2\subsetneq\mf{h}_1$. In conclusion, $ab$ skewers a hyperplane labelled by $v$, so $v\in\G(ab)$.
\end{proof}

\begin{lem}\label{commutation criterion}
Consider $g,h\in\A_{\G}$ and $x\in\X_{\G}$. If $h$ fixes $\mscr{W}(x|gx)$ pointwise, then $g$ and $h$ commute.
\end{lem}
\begin{proof}
Recall that $\Min(g)\cu\X_{\G}$ is convex. Replacing $x$ with its gate-projection to $\Min(g)$ can only shrink the set $\mscr{W}(x|gx)$, so we can assume that $x$ is on an axis of $g$. Conjugating $g$ and $h$ by $x$, we can further assume that $x=1$, i.e.\ that $g$ is cyclically reduced. Now, the conclusion is straightforward.
\end{proof}

\begin{lem}\label{increasing labels}
Consider $g,h\in\A_{\G}$.
\begin{enumerate}
\item There exists $k\in\langle g,h\rangle$ with $\G(g)\cup\G(h)\cu\G(k)$.
\item If $g$ is cyclically reduced and $h\not\in\A_{\G(g)}$, then there exists $k\in\langle g,h\rangle$ with $\G(k)\not\cu\G(g)$.
\end{enumerate}
\end{lem}
\begin{proof}
In order to prove part~(1), note that an element $x\in\A_{\G}$ is loxodromic in the tree $\T_v$ if and only if $v\in\G(x)$. Thus $\langle g,h\rangle$ acts without a global fixed point on all trees $\T_v$ with $v\in\G(g)\cup\G(h)$. It follows (for instance, by \cite[Theorem~5.1]{CU}) that there exists $k\in\langle g,h\rangle$ that is loxodromic in all these trees, that is, $\G(g)\cup\G(h)\cu\G(k)$.

We now prove part~(2). We can assume that $\G(h)\cu\G(g)$, otherwise we can take $k=h$. Since $g$ is cyclically reduced, the vertex set of $\Min(g)\cu\X_{\G}$ is contained in $\A_{\G(g)}\x\A_{\G(g)^{\perp}}$.

Observe that $\Min(h)$ and $\A_{\G(g)}\x\A_{\G(g)^{\perp}}$ are disjoint. Indeed, suppose that a vertex $x\in\X_{\G}$ lies in their intersection. Since $x\in\Min(h)$, we have $x^{-1}hx\in\A_{\G(h)}\leq\A_{\G(g)}$. Hence $h$ lies in $\A_{\G(g)}$, since $x\in\A_{\G(g)}\x\A_{\G(g)^{\perp}}$. This contradicts the assumption that $h\not\in\A_{\G(g)}$.

Now, since $\Min(h)$ and $\A_{\G(g)}\x\A_{\G(g)^{\perp}}$ are disjoint and convex, there exists a hyperplane $\mf{w}$ separating them. Choosing $\mf{w}$ closest to $\A_{\G(g)}\x\A_{\G(g)^{\perp}}$, we can assume that $w:=\g(\mf{w})$ does not lie in $\G(g)$. It follows that, in the tree $\T_w$, the elements $g$ and $h$ are both elliptic, with disjoint sets of fixed points (which are just the projections to $\T_w$ of $\Min(g)$ and $\Min(h)$). Thus, $gh$ is loxodromic in $\T_w$, which implies that $w\in\G(gh)$.
\end{proof}

\subsection{Label-irreducible elements.}\label{LI sect}

The following notion will play a fundamental role in the rest of the paper. We recall here a few observations from \cite[Subsection~3.2]{Fio10a}.

\begin{defn}
An element $g\in\A_{\G}\setminus\{1\}$ is \emph{label-irreducible} if the subgraph $\G(g)\cu\G$ does not split as a nontrivial join.
\end{defn}

Recall that, if $\mc{G}_1$ and $\mc{G}_2$ are graphs, then their \emph{join} $\mc{G}_1\ast\mc{G}_2$ is the graph obtained by adding to the disjoint union $\mc{G}_1\sqcup\mc{G}_2$ edges between every vertex of $\mc{G}_1$ and every vertex of $\mc{G}_2$.

\begin{rmk}\label{LI properties}
The following are straightforward properties of label-irreducibles.
\begin{enumerate}
\item An element $g$ is label-irreducible if and only if the subgroup $\langle g\rangle$ is convex-cocompact in $\X_{\G}$.
\item Every $g\in\A_{\G}$ can be written as $g=g_1\cdot\ldots\cdot g_k$, where $g_1,\dots,g_k$ are pairwise-commuting label-irreducibles with $\langle g_i\rangle\cap\langle g_j\rangle=\{1\}$ for $i\neq j$. This decomposition is unique up to permutation, so we refer to the $g_i$ as the \emph{label-irreducible components} of $g$. Here $\G(g)=\G(g_1)\sqcup\dots\sqcup\G(g_k)$ is precisely the maximal join-decomposition of $\G(g)$. 
\item Elements $g,h\in\A_{\G}$ commute if and only if every label-irreducible component of $g$ commutes with every label-irreducible component of $h$.
\item If two label-irreducible elements $g,h\in\A_{\G}$ commute, then either $\G(g)\cu\G(h)^{\perp}$ or $\langle g,h\rangle\simeq\Z$ (for instance, this follows from Remark~\ref{pf vs qc rmk} and Lemma~\ref{pf lemma}(1)).
\item If $g=g_1\cdot\ldots\cdot g_k$ is the decomposition of $g$ into label-irreducibles, then the centraliser of $g$ in $\A_{\G}$ splits as:
\[C_1\x\dots\x C_k\x P,\]
where $C_i\leq\A_{\G}$ is the maximal cyclic subgroup containing $g_i$, and $P\leq\A_{\G}$ is parabolic. If $g$ is cyclically reduced, then $P=\A_{\G(g)^{\perp}}$.
\item Let $G\leq\A_{\G}$ be convex-cocompact in $\X_{\G}$. Consider an element $g\in G$ and its decomposition into label-irreducibles $g=g_1\cdot\ldots\cdot g_k$, where $g_i\in\A_{\G}$. Then $G\cap\langle g_i\rangle\neq\{1\}$ for each $1\leq i\leq k$ (see for instance \cite[Lemma~3.16]{Fio10a}). In fact, if $G$ is $q$--convex-cocompact, then $G$ contains a power of each $g_i$ with exponent $\leq q$ (see \cite[Remark~3.17]{Fio10a}).
\end{enumerate}
\end{rmk}

\begin{lem}\label{long min intersection}
Consider $g,h\in\A_{\G}$ and $v\in\G$. Suppose that $g$ is loxodromic in $\T_v$ with axis $\alpha$. If $\Min(h,\T_v)$ intersects $\alpha$ in an arc of length $>4\dim\X_{\G}\cdot\max\{\ell_{\T_v}(g),\ell_{\T_v}(h)\}$, then $h\alpha=\alpha$.
\end{lem}
\begin{proof}
Note that exactly one of the label-irreducible components of $g$ is loxodromic in $\T_v$. In addition, this component has the same axis and the same translation length as $g$. Thus, we can assume that $g$ is label-irreducible. In this case, \cite[Corollary~3.14]{Fio10a} shows that $g$ and $h$ commute, so it is clear that $h$ preserves the axis of $g$.
\end{proof}

Recall that, given a group $G$, a subgroup $H\leq G$ and a subset $K\cu G$, we denote by $Z_H(K)$ the centraliser of $K$ in $H$, i.e.\ the subgroup of elements of $H$ that commute with all elements of $K$.

\begin{rmk}\label{cmp and straight proj rmk}
Let $G\leq\A_{\G}$ be convex-cocompact and let $g\in G$ be label-irreducible.
\begin{enumerate}
\item If $\varphi\in\aut(G)$ is coarse-median preserving (for the coarse median structure induced by $\A_{\G}$), then $\varphi(g)$ is again label-irreducible. This follows from Remark~\ref{LI properties}(1) and Remark~\ref{cc vs qc}.
\item We can define the \emph{straight projection} $\pi_g\colon Z_G(g)\ra\Z$ as the only homomorphism that is surjective, with convex-cocompact kernel, and with $\pi_g(g)>0$. 

Recall that $Z_{\A_{\G}}(g)=C\x P$, where $P\leq\A_{\G}$ is parabolic and $C\leq\A_{\G}$ is the maximal cyclic subgroup containing $g$. The subgroup $Z_G(g)\leq Z_{\A_{\G}}(g)$ is virtually $\langle g\rangle\x (G\cap P)$. Thus, $\pi_g$ is simply the restriction to $Z_G(g)$ of the coordinate projection $C\x P\ra C$, suitably shrinking the codomain to ensure that $\pi_g$ is surjective. In particular, we have $\ker\pi_g=G\cap P$.

If $\varphi\in\aut(G)$ is coarse-median preserving, note that $\pi_{\varphi(g)}=\pi_g\o\varphi^{-1}$.
\end{enumerate}
\end{rmk}

We conclude this subsection with a couple of definitions that will be needed later on.

\begin{defn}
A subgroup $H\leq\A_{\G}$ is \emph{full} if it is closed under taking label-irreducible components.
\end{defn}

\begin{rmk}\label{full rmk}
If $H\leq\A_{\G}$ is full, then $H$ is generated by the label-irreducibles that it contains. 
\end{rmk}

Observe that, for every group $G$ and every subset $A\cu G$, we have $Z_GZ_GZ_G(A)=Z_G(A)$.

\begin{defn}\label{centralisers defn}
Let $G$ be a group. We say that a subgroup $H\leq G$ is a \emph{centraliser in $G$} if $H=Z_GZ_G(H)$. Equivalently, there exists a subset $A\cu G$ such that $H=Z_G(A)$.
\end{defn}

\begin{rmk}
Centralisers in $\A_{\G}$ are full, by Remark~\ref{LI properties}(3).
\end{rmk}

\subsection{Parabolic subgroups.}

Recall the following standard terminology.

\begin{defn}
A subgroup $P\leq\A_{\G}$ is \emph{parabolic} if $P=g\A_{\L}g^{-1}$ for some $\L\cu\G$ and $g\in\A_{\G}$.
\end{defn}

The following alternative characterisation of parabolic subgroups will be needed in Subsection~\ref{om-intersection sect}.

\begin{prop}\label{parabolic prop}
A subgroup $H\leq\A_{\G}$ is parabolic if and only if it satisfies the following property. For every cyclically reduced element $a\in\A_{\G}$, written as a reduced word $a_1\dots a_n$ with $a_i\in\G^{\pm}$, and for every $g\in\A_{\G}$ with $gag^{-1}\in H$, we have $ga_ig^{-1}\in H$ for every $i$.
\end{prop}
\begin{proof}
We first show that parabolics have this property. Since the property is invariant under conjugation, it suffices to verify it for subgroups of the form $\A_{\L}$ with $\L\cu\G$. If $gag^{-1}\in\A_{\L}$, then $\G(a)=\G(gag^{-1})\cu\L$ and, since $a$ is cyclically reduced, each $a_i$ must lie in $\L$. Observing that elements of $\A_{\L}$ are $\A_{\G}$--conjugate if and only if they are $\A_{\L}$--conjugate, we see that $g\in\A_{\L}\cdot Z_{\A_{\G}}(a)$. Since $a$ is cyclically reduced, $Z_{\A_{\G}}(a)=\bigcap_iZ_{\A_{\G}}(a_i)$, hence $ga_ig^{-1}\in\A_{\L}$ for every $i$.

Conversely, let $H\leq\A_{\G}$ be a subgroup satisfying the property. By Lemma~\ref{increasing labels}(1), there exists $x\in H$ with $\G(x)=\G(H)$. Up to conjugating $H$, we can assume that $x$ is cyclically reduced. Our property then yields $\A_{\G(H)}\leq H$. If the reverse inclusion did not hold, Lemma~\ref{increasing labels}(2) would yield an element $y\in H$ with $\G(y)\not\cu\G(H)$, which is impossible. Thus $H=\A_{\G(H)}$.
\end{proof}

\begin{defn}
A \emph{parabolic stratum} is a subset of $\X_{\G}$ of the form $g\A_{\Delta}$ for some $\Delta\cu\G$ and $g\in\A_{\G}$ (we identify as usual the $0$--skeleton of $\X_{\G}$ with $\A_{\G}$).
\end{defn}

A parabolic stratum can equivalently be defined as the set of points of $\X_{\G}$ that one can reach starting at a given vertex $g\in\X_{\G}$ and only crossing edges with label in a given subgraph $\Delta\cu\G$.

\begin{rmk}\label{parabolic stratum rmk}
Here are a few straightforward properties of parabolic strata.
\begin{enumerate}
\item Intersections of parabolic strata are parabolic strata. Gate-projections of parabolic strata to parabolic strata are parabolic strata.
\item If $\mc{P}$ is a parabolic stratum and $g\in\A_{\G}$ is an element with $g\mc{P}\cap\mc{P}\neq\emptyset$, then $g\mc{P}=\mc{P}$.
\item Stabilisers of parabolic strata are parabolic subgroups of $\A_{\G}$.
\item For every $g\in\A_{\G}$, there exists a parabolic stratum $\mc{P}$ such that the hyperplanes of $\mc{P}$ are precisely the hyperplanes of $\X_{\G}$ that are preserved by $g$, namely the elements of $\overline{\mc{W}}_0(g,\X_{\G})$. It follows that, for every subgroup $H\leq\A_{\G}$, there exists a parabolic stratum $\mc{P}$ whose hyperplanes are precisely those in $\overline{\mc{W}}_0(H,\X_{\G})$.
% just gate-project
\end{enumerate}
\end{rmk}

\begin{lem}\label{cc conjugates}
If $H\leq\A_{\G}$ is convex-cocompact and $gHg^{-1}\leq H$ for some $g\in\A_{\G}$, then $gHg^{-1}=H$.
\end{lem}
\begin{proof}
Let $Z\cu\X_{\G}$ be an $H$--essential convex subcomplex. Since $g^{-1}Z$ is $g^{-1}Hg$--invariant and $H\leq g^{-1}Hg$, the finite set $\mscr{W}(Z|g^{-1}Z)$ is $H$--invariant, hence it is contained in $\overline{\mc{W}}_0(H,\X_{\G})$.

By Remark~\ref{parabolic stratum rmk}(4), there exists a parabolic stratum $\mc{P}\cu\X_{\G}$ such that the hyperplanes of $\mc{P}$ are precisely those preserved by $H$. By the previous paragraph, we can choose $\mc{P}$ so that it intersects both $Z$ and $g^{-1}Z$. Note that $\mc{P}$ is acted upon vertex-transitively by its stabiliser $P\leq\A_{\G}$, so there exists $x\in P$ such that $xg^{-1}Z\cap Z\neq\emptyset$. By Lemma~\ref{commutation criterion}, $x$ commutes with $H$. Thus, replacing $g$ with $gx^{-1}$, we can assume that $g^{-1}Z\cap Z\neq\emptyset$ without altering $gHg^{-1}$. Since $g^{-1}Z\cap Z$ is $H$--invariant and $Z$ is $H$--essential, we deduce that $Z\cu g^{-1}Z$. Hence $gZ\cu Z$. 

Now, pick a vertex $y\in Z$. Since $H$ is convex-cocompact, it acts cocompactly on $Z$, hence there exist integers $1\leq m<n$ such that $g^my$ and $g^ny$ are in the same $H$--orbit. Since $\A_{\G}$ acts freely on $\X_{\G}$, this implies that $g^{n-m}\in H$. In particular, $g^{n-m}$ normalises $H$, so we cannot have $gHg^{-1}\lneq H$.
\end{proof}

\begin{defn}\label{G--parabolic defn}
Let $G\leq\A_{\G}$ be convex-cocompact. A subgroup of $G$ is \emph{$G$--parabolic} if it is of the form $G\cap P$ with $P\leq\A_{\G}$ parabolic. To avoid confusion, the prefix $G$-- will never be omitted.
\end{defn}

\begin{lem}\label{intersecting stratum}
Let $G\leq\A_{\G}$ act cocompactly on a convex subcomplex $Y\cu\X_{\G}$. For every parabolic subgroup $P\leq\A_{\G}$, there exists a parabolic stratum $\mc{P}'$ stabilised by a parabolic subgroup $P'\leq P$ such that $G\cap P'=G\cap P$ and $\mc{P}'\cap Y\neq\emptyset$.
\end{lem}
\begin{proof}
Consider the gate-projection $\pi_Y\colon\X_{\G}\ra Y$. Let $\mc{P}$ be a parabolic stratum stabilised by $P$ and pick a point $p\in\pi_Y(\mc{P})$. Define $\mc{P}'$ as the parabolic stratum that contains $p$ and satisfies $\g(\mscr{W}(\mc{P'}))=\g(\mscr{W}(\mc{P}))\cap\g(\mscr{W}(Y|\mc{P}))^{\perp}$. It is easy to see that $\pi_Y(\mc{P})=Y\cap\mc{P}'$.

Let $P'\leq\A_{\G}$ be the parabolic subgroup associated to $\mc{P}'$. Since $\mc{P}'$ crosses the same hyperplanes as a sub-stratum of $\mc{P}$, we have $P'\leq P$, hence $G\cap P'\leq G\cap P$. By Lemma~\ref{cc intersection}, $G\cap P'$ acts cocompactly on $\pi_Y(\mc{P}')=Y\cap\mc{P}'$. This set coincides with $\pi_Y(\mc{P})$, so it is preserved by $G\cap P$, hence $G\cap P'$ has finite index in $G\cap P$. In particular, every element of $G\cap P$ has a power that lies in $P'$. Since $P'$ is parabolic, this implies that $G\cap P\leq P'$, and hence $G\cap P'=G\cap P$. 
\end{proof}

\begin{cor}\label{fin many conj parabolics}
If $G\leq\A_{\G}$ is convex-cocompact, there are only finitely many $G$--conjugacy classes of $G$--parabolic subgroups.
\end{cor}
\begin{proof}
Let $Y\cu\X_{\G}$ be a convex subcomplex on which $G$ acts cocompactly. By Lemma~\ref{intersecting stratum}, every $G$--parabolic subgroup is of the form $G\cap P$ for a parabolic subgroup $P\leq\A_{\G}$ whose parabolic stratum $\mc{P}$ intersects $Y$. There are only finitely many $G$--orbits of such parabolic strata, hence only finitely many $G$--conjugacy classes of such subgroups of $\A_{\G}$. 
\end{proof}

\begin{lem}\label{special cc}
Let $G\leq\A_{\G}$ be a $q$--convex-cocompact subgroup for $q\geq 1$. Let $H\leq G$ be an arbitrary convex-cocompact subgroup. Then:
\begin{enumerate}
\item $N_G(H)$ has a finite-index subgroup that splits as $H\x K$, where $K\leq G$ is $G$--parabolic;  
\item $Z_G(H)$ acts on the set $\mc{W}_1(G)\cap\overline{\mc{W}}_0(H)$ with at most $2q\cdot\#\G^{(0)}$ orbits;
%The proof of part~(2) does not actually require convex-cocompactness of $H$.
\item every $G$--parabolic subgroup of $G$ is $q$--convex-cocompact in $\A_{\G}$.
\end{enumerate}
\end{lem}
\begin{proof}
Choose convex subcomplexes $Z\cu Y\cu\X_{\G}$, where $Z$ is $H$--invariant and $H$--essential, while $Y$ is $G$--invariant and $G$--essential.

We prove part~(3) first. Lemma~\ref{intersecting stratum} shows that $G$--parabolic subgroups of $G$ are always of the form $G\cap P$, where $P$ is the stabiliser of a stratum $\mc{P}$ that intersects $Y$. Observe that points of $\mc{P}\cap Y$ in the same $G$--orbit are also in the same $(G\cap P)$--orbit. Indeed, if $x$ and $gx$ lie in $\mc{P}\cap Y$ for some $g\in G$, then $g\mc{P}\cap\mc{P}\neq\emptyset$, hence $g\mc{P}=\mc{P}$ and $g\in G\cap P$. This proves part~(3).

We now discuss the rest of the lemma. Remark~\ref{parabolic stratum rmk}(4) provides a parabolic stratum $\mc{P}\cu\X_{\G}$ whose hyperplanes are precisely the elements of $\overline{\mc{W}}_0(H)$. We can choose $\mc{P}$ so that $\mc{P}\cap Z\neq\emptyset$.
% because, since the reduced core is a product 0x1, there exist $H$--invariant hyperplanes adjacent to $Z$
Then the elements of $\mc{W}_1(G)\cap\overline{\mc{W}}_0(H)$ are precisely the hyperplanes of the intersection $\mc{P}\cap Y$, so we have a splitting $\overline{\mc{C}}(H,Y)=Z\x(\mc{P}\cap Y)$. Recall that $N_G(H)$ preserves $\overline{\mc{C}}(H,Y)$ along with its two factors.

Let $P\leq\A_{\G}$ be the stabiliser of $\mc{P}$. By part~(3), the $G$--parabolic subgroup $G\cap P$ acts on $\mc{P}\cap Y$ with at most $q$ orbits of vertices. In particular, since vertices of $\X_{\G}$ have degree $2\#\G^{(0)}$, there are at most $2q\cdot\#\G^{(0)}$ orbits of hyperplanes of $\mc{P}\cap Y$. By Lemma~\ref{commutation criterion}, $G\cap P$ is contained in $Z_G(H)$. This proves parts~(1) and~(2), taking $K=G\cap P$.
\end{proof}

\subsection{Semi-parabolic subgroups.}

\begin{defn}\label{AP defn}
A subgroup $H\leq\A_{\G}$ is \emph{semi-parabolic} if it is conjugate to a subgroup of the form $\langle a_1,\dots,a_k\rangle\x\A_{\Delta}$, where:
\begin{itemize}
\item the $a_i$ are cyclically reduced, label-irreducible and not proper powers;  
% the not-a-proper-power condition was added last-minute for Remark~\ref{AP chains rmk}
\item we have $\G(a_i)\cu\Delta^{\perp}$ for all $i$, and $\G(a_i)\cu\G(a_j)^{\perp}$ for all $i\neq j$.
\end{itemize}
We can always assume that $\A_{\Delta}$ has trivial centre, as this can be added to the $a_i$.
\end{defn}

We say that a subgroup $H\leq\A_{\G}$ is \emph{closed under taking roots} if, whenever $g^n\in H$ for some $g\in\A_{\G}$ and $n\geq 1$, we actually have $g\in H$.

Semi-parabolic subgroups are always convex-cocompact, full and closed under taking roots.

\begin{lem}\label{almost parabolic subgroups}
A subgroup $H\leq\A_{\G}$ is semi-parabolic if and only if it splits as $H=A\x P$, where $P$ is parabolic and $A$ is abelian, full and closed under taking roots.
\end{lem}
\begin{proof}
It is clear that semi-parabolic subgroups admit such a splitting. Conversely, suppose that $H\leq\A_{\G}$ is an arbitrary subgroup with a splitting $A\x P$ as in the statement. 

Since $A$ is full, it has a basis of label-irreducible elements $g_1,\dots,g_k$. The fact that $A$ is closed under taking roots implies that none of the $g_i$ can be a proper power. Since the $g_i$ commute, we must have $\G(g_i)\cu\G(g_j)^{\perp}$ for all $i\neq j$, by Remark~\ref{LI properties}(4). Since $P$ is parabolic, we can conjugate $H$ and assume that $H=\langle g_1,\dots,g_k\rangle\x\A_{\Delta}$ for some $\Delta\cu\G$. Since $g_i$ is label-irreducible and commutes with $\A_{\Delta}$, we have $\G(g_i)\cu\Delta^{\perp}$, again by Remark~\ref{LI properties}(4).

We are left to further conjugate $H$ in order to ensure that the $g_i$ are all cyclically reduced. Write $g_i=x_ia_ix_i^{-1}$ as a reduced word with $a_i$ cyclically reduced. Since $g_i$ commutes with $\A_{\Delta}$, we have $g_i\in\A_{\Delta_{\perp}}$ (Remark~\ref{perp rmk}(1)). In particular, $x_1\in\A_{\Delta_{\perp}}$ commutes with $\A_{\Delta}$ and, conjugating $H$ by $x_1$, we can assume that $g_1=a_1$. 

Now, since $g_1$ is cyclically reduced and $g_2$ is a label-irreducible commuting with $g_1$, Remark~\ref{LI properties}(5) shows that $x_2$ commutes with $g_1$. Thus, conjugating $H$ by $x_2$, we can assume that $g_2=a_2$ without affecting $g_1=a_1$. Repeating this procedure, we can ensure that all $g_i$ are cyclically reduced. 
\end{proof}

\begin{cor}\label{AP intersections}
Intersections of semi-parabolic subgroups are again semi-parabolic.
\end{cor}
\begin{proof}
Let $H_1,H_2\leq\A_{\G}$ be two semi-parabolic subgroups. Write $H_i=A_i\x P_i$, with $P_i$ parabolic and $A_i$ abelian, full and closed under taking roots. 

Every label-irreducible element in $H_i$ lies either in $A_i$ or in $P_i$. Note that $H_1$ and $H_2$ are full, so $H_1\cap H_2$ is full. Remark~\ref{full rmk} implies that $H_1\cap H_2$ is generated by the label-irreducibles that it contains. Hence $H_1\cap H_2$ is generated by the four subgroups $A_1\cap A_2$, $A_1\cap P_2$, $A_2\cap P_1$ and $P_1\cap P_2$. The first three subgroups generate a full abelian subgroup $A\leq H_1\cap H_2$ closed under taking roots. Since $H_1\cap H_2$ splits as $A\x(P_1\cap P_2)$, Lemma~\ref{almost parabolic subgroups} shows that $H_1\cap H_2$ is semi-parabolic.

There is no need to consider intersections of infinitely many semi-parabolic subgroups because of Remark~\ref{AP chains rmk} below.
\end{proof}

\begin{rmk}\label{AP chains rmk}
There is a uniform bound (depending only on $\G$) on the length of any chain of semi-parabolic subgroups of $\A_{\G}$. Indeed, let $H_1\lneq H_2\leq\A_{\G}$ be two semi-parabolic subgroups and write $H_i=A_i\x P_i$ so that the $P_i$ have trivial centre. Then $P_1\cap A_2=\{1\}$ and, since $P_1$ is full, we must have $P_1\leq P_2$. In conclusion, either $P_1\lneq P_2$, or $P_1=P_2$ and $A_1\lneq A_2$, since $A_1$ is full. In the latter case, we have $\rk A_1<\rk A_2$, since $A_1$ is closed under taking roots.
\end{rmk}

\begin{rmk}\label{AP within parabolic}
Consider a semi-parabolic subgroup $H\leq\A_{\G}$ and a parabolic subgroup $P\leq\A_{\G}$ that does not split as a product. If $H\leq P$ and $\G(H)=\G(P)$, then either $H$ is cyclic or $H=P$.
\end{rmk}

\begin{lem}\label{normalisers of subgroups of AP}
Let $H\leq\A_{\G}$ be semi-parabolic. Let $K\leq H$ be any subgroup with $\G(K)=\G(H)$. 
\begin{enumerate}
\item If $gKg^{-1}\leq H$ for some $g\in\A_{\G}$, then $gHg^{-1}=H$.
\item Suppose that some $g\in\A_{\G}$ commutes with $K$, but not with $H$. Then $H$ admits a splitting $A\x P_1\x P_2$, where $A$ is abelian, the $P_i$ are parabolics with trivial centre (possibly with $P_2=\{1\}$), and $K$ is contained in $A\x A'\x P_2$ for some abelian subgroup $A'\leq P_1$.
% either $A'$ is cyclic, or $P_1$ is a product and $A'$ has a basis element in each factor of $P_1$
\end{enumerate}
\end{lem}
\begin{proof}
We begin with part~(1). Consider a splitting $H=A\x P$ as in Lemma~\ref{almost parabolic subgroups}. Recall that $\G(A)\cap\G(P)=\emptyset$ and that every label-irreducible element of $H$ lies in $A\cup P$. If $g\in\A_{\G}$, the intersection $H\cap g^{-1}Hg$ is full, so Remark~\ref{full rmk} gives: 
\[H\cap g^{-1}Hg=(A\cap g^{-1}Ag)\x (P\cap g^{-1}Pg).\]
Now, if $gKg^{-1}\leq H$, we have $K\leq H\cap g^{-1}Hg$. Hence:
\[\G(H)=\G(K)\cu\G(H\cap g^{-1}Hg).\] 
This implies that $\G(A)\cu\G(A\cap g^{-1}Ag)$ and $\G(P)\cu\G(P\cap g^{-1}Pg)$, thus $g$ must normalise both $A$ and $P$. It follows that $gHg^{-1}=H$, proving part~(1).

We now prove part~(2). Let $g\in\A_{\G}$ be an element that commutes with $K$, but not with $H$. We can assume that the parabolic subgroup $P$, defined as above, has trivial centre.

By Lemma~\ref{increasing labels}(1), $K$ contains an element $k$ with $\G(k)=\G(K)=\G(H)$. We can write $k=ap$, where $a\in A$ and $p\in P$ satisfy $\G(a)=\G(A)$ and $\G(p)=\G(P)$. Let $p_1,\dots,p_n$ be the label-irreducible components of $p$. Since $g$ commutes with $k$, it must commute with $A$ and with all the $p_i$. The intersection of the centralisers of the $p_i$ is the subgroup $\langle p_1',\dots,p_n'\rangle \x Z_{\A_{\G}}(P)$, where $\langle p_i'\rangle$ is the maximal cyclic subgroup containing $\langle p_i\rangle$.

Since $g$ commutes with $A$, but not with $H$, it cannot commute with $P$, and so it must have powers of some of the $p_i'$ among its label-irreducible components. Up to reordering, these are powers of $p_1',\dots,p_m'$ for some $1\leq m\leq n$. We have a splitting $P=P_1\x P_2$ where the $P_i$ are parabolic and $\G(P_1)=\G(p_1)\sqcup\dots\sqcup\G(p_m)$. Since $P$ has trivial centre, so do the $P_i$. Finally, since $K$ commutes with $g$, it is contained in $A\x\langle p_1',\dots,p_m'\rangle\x P_2$, as required.
\end{proof}

In the rest of the subsection, we fix a convex-cocompact subgroup $G\leq\A_{\G}$. By analogy with Definition~\ref{G--parabolic defn}, we introduce the following.

\begin{defn}
A subgroup $Q\leq G$ is \emph{$G$--semi-parabolic} if $Q=G\cap H$ for a semi-parabolic subgroup $H\leq\A_{\G}$. In order to avoid confusion, the prefix $G$-- will never be omitted.
\end{defn}

Our interest in this notion is due to the following remark.

\begin{rmk}\label{AP centralisers rmk}
Centralisers in $G$ (in the sense of Definition~\ref{centralisers defn}) are $G$--semi-parabolic. This follows from Remark~\ref{LI properties}(5) and Corollary~\ref{AP intersections}.
\end{rmk}

\begin{lem}\label{G--AP structure}
If $Q\leq G$ is $G$--semi-parabolic, there exists a unique minimal semi-parabolic subgroup $H\leq\A_{\G}$ such that $Q=G\cap H$. We can write $H=\langle a_1,\dots,a_k\rangle\x P$, where:
\begin{enumerate}
\item the $a_i$ are pairwise-commuting label-irreducibles with $\langle a_i\rangle\cap Q\neq\{1\}$;
\item $P\leq\A_{\G}$ is parabolic and both $P$ and $G\cap P$ have trivial centre;
\item we have $\G(Q)=\G(H)$.
\end{enumerate}
\end{lem}
\begin{proof}
By Corollary~\ref{AP intersections}, the intersection $H$ of all semi-parabolic subgroups of $\A_{\G}$ containing $Q$ is the unique minimal semi-parabolic subgroup with $Q=G\cap H$.

We can write $H=\langle a_1,\dots,a_k\rangle\x P$, where the $a_i$ are pairwise-commuting label-irreducibles and $P\leq\A_{\G}$ is parabolic. Remark~\ref{LI properties}(6) shows that $G$ (and hence $Q$) contains a power of each $a_i$. We can assume that $P$ has trivial centre, as this can be incorporated in the $a_i$. 

Let us show that $\G(Q)=\G(H)$. By Remark~\ref{LI properties}(6), the sets $\G(a_i)$ are all contained in $\G(Q)$ and we have $\G(Q\cap P)=\G(Q)\cap\G(P)$. By Lemma~\ref{increasing labels}(1), there exists $g\in Q\cap P$ with $\G(g)=\G(Q\cap P)$. If this were a proper subset of $\G(P)$, we would be able to find a parabolic subgroup $P'$ with $g\in P'\lneq P$ and $\G(g)=\G(P')$. Lemma~\ref{increasing labels}(2) would then guarantee that $Q\cap P\leq P'$. Remark~\ref{LI properties}(6) and the fact that $P'$ is closed under taking roots would imply that $Q$ is contained in $\langle a_1,\dots,a_k\rangle\x P'$, violating minimality of $H$. We conclude that $\G(Q\cap P)=\G(P)$, which shows that $\G(Q)=\G(H)$. 

Finally, if $G\cap P$ contained a nontrivial element $g$ in its centre, we would have $G\cap P\leq Z_P(g)$. As above, the subgroup $Q=G\cap H$ would then be contained in $\langle a_1,\dots,a_k\rangle\x Z_P(g)$. Since $Z_P(g)$  has nontrivial centre, it is a proper semi-parabolic subgroup of $P$, which violates minimality of $H$.
\end{proof}

\begin{rmk}\label{filling subgroup rmk}
Consider a $G$--semi-parabolic subgroup $Q\leq G$ and a homomorphism $\rho\colon Q\ra\R$ with $\G(\ker\rho)=\G(Q)$. Then there exists a finitely generated subgroup $K\leq\ker\rho$ such that any $G$--semi-parabolic subgroup containing $K$ will contain $Q$.
% here it is important that semi-parabolics are root closed, but we only need this for centralisers anyway

Indeed, write $Q=G\cap H$ with $H=\langle a_1,\dots,a_k\rangle\x P$ as in Lemma~\ref{G--AP structure}. Write $P=P_1\x\dots\x P_m$, where each $P_i$ is a parabolic subgroup of $\A_{\G}$ that does not split as a product. By Remark~\ref{LI properties}(6), we have $\G(G\cap P_i)=\G(P_i)$. In addition, since $G\cap P$ has trivial centre, the intersection $G\cap P_i$ is non-abelian. Define $K$ so that $\G(K)=\G(Q)$ and so that it contains a non-abelian subgroup of each $G\cap P_i$. Remark~\ref{AP within parabolic} applied to each $P_i$ implies that $K$ satisfies the required property.
\end{rmk}

In the rest of the subsection, we prove a couple of results aimed at classifying kernels of homomorphisms $Q\ra\R$, where $Q\leq G$ is $G$--semi-parabolic.

\begin{lem}\label{kernel envelope}
Let $Q\leq G$ be $G$--semi-parabolic. Let $H=\langle a_1,\dots,a_k\rangle\x P$ be as in Lemma~\ref{G--AP structure}. If $\rho\colon Q\ra\R$ is a homomorphism, then $\ker\rho\cu G\cap H'$ for a subgroup $H'\leq H$ such that:
\begin{enumerate}
\item $H'=\langle a_1,\dots,a_s\rangle\x P$ for some $0\leq s\leq k$, up to reordering the $a_i$;
\item $\G(\ker\rho)=\G(H')$ and $Z_G(\ker\rho)=Z_G(H')$.
\end{enumerate}
\end{lem}
\begin{proof}
Define $H'=\langle a_1,\dots,a_s\rangle\x P$, where $s$ is the smallest integer such that $H'$ contains $\ker\rho$ (up to reordering the $a_i$). Minimality of $s$ implies that $\G(\ker\rho)$ contains $\G(a_1),\dots,\G(a_s)$. In order to conclude the proof, we only need to show that $\G(P)\cu\G(\ker\rho)$ and $Z_G(\ker\rho)=Z_G(H')$. 

Since $\G(G\cap H)=\G(H)$, Remark~\ref{LI properties}(6) implies that $\G(G\cap P)=\G(P)$. Thus, for every $v\in\G(P)$, the action $G\cap P\acts\T_v$ is not elliptic. If $G\cap P\acts\T_v$ had a fixed point at infinity, then Lemma~\ref{long min intersection} would show that all loxodromics have the same axis in $\T_v$ and, by \cite[Corollary~3.14]{Fio10a}, they would lie in the centre of $G\cap P$. However, $G\cap P$ has trivial centre by our choice of $H$.

We conclude that, for every $v\in\G(P)$, the action $G\cap P\acts\T_v$ is nonelementary. Now, we can use \cite[Theorem~1.4]{Maher-Tiozzo-Crelle} and the argument in the proof of \cite[Corollary~1.7(2)]{Sisto-Crelle} to conclude that, for every $v\in\G(P)$, the kernel of $\rho|_{G\cap P}$ contains an element acting loxodromically on $\T_v$ (random walks can be easily avoided when $\ker(\rho|_{G\cap P})$ is finitely generated). Hence $\G(P)\cu\G(\ker\rho)$.
% **IF** we knew that $\ker(\rho|_{G\cap P})$ is finitely generated, the following would be an easier proof...
% If $\ker(\rho|_{G\cap P})$ acted elliptically on $\T_v$, then $G\cap P$ would leave its fixed point set invariant. The action of $G\cap P$ on this set would factor through an action of the abelian group $\rho(G\cap P)$, hence it would leave a line invariant. By ..., this would imply that a finite-index subgroup of $G\cap P$ has a $\Z$--factor, contradicting our assumptions.
% Thus, since $\ker(\rho|_{G\cap P})$ is finitely generated and not elliptic, it contains an element acting loxodromically on $\T_v$

Finally, let us show that the inclusion $Z_G(H')\leq Z_G(\ker\rho)$ cannot be strict. If it were, Lem\-ma~\ref{normalisers of subgroups of AP}(2) would yield a splitting $P=P_1\x P_2$, where $P_1$ is a non-abelian parabolic and $\ker\rho|_{G\cap P_1}$ is abelian. Since $P$ is chosen as in Lemma~\ref{G--AP structure}, the intersection $G\cap P_1$ is a non-virtually-abelian special group (recall Lemma~\ref{cc intersection}). However, the above shows that $G\cap P_1$ is abelian--by--abelian, which contradicts the flat torus theorem (see e.g.\ \cite[Theorem~II.7.1(5)]{BH}).
\end{proof}

\begin{rmk}\label{kernel centre rmk}
Given a $G$--semi-parabolic subgroup $Q\leq G$ and a homomorphism $\rho\colon Q\ra\R$, the centre of $\ker\rho$ is always contained in the centre of $Q$.

Indeed, let the subgroups $H,H'$ and the integers $s,k$ be as in Lemma~\ref{kernel envelope}. The lemma shows that the centre of $\ker\rho$ commutes with $H'$. Since $\ker\rho\leq H'$, it is also clear that $\ker\rho$ commutes with $a_{s+1},\dots,a_k$. Together with $H'$, these elements generate $H$, hence the centre of $\ker\rho$ commutes with $H$, as required.
\end{rmk}

When the homomorphism $\rho$ is discrete, we have the following dichotomy.

\begin{prop}\label{kernel dichotomy}
Let $Q\leq G$ be $G$--semi-parabolic. Let $\rho\colon Q\ra\Z$ be a homomorphism. Then:
\begin{itemize}
\item either $\G(\ker\rho)=\G(Q)$, and $N_G(\ker\rho)$ is a finite-index subgroup of $N_G(Q)$;
\item or $\ker\rho$ is $G$--semi-parabolic, and a finite-index subgroup of $Q$ splits as $\Z\x\ker\rho$.
\end{itemize}
\end{prop}
\begin{proof}
Write $Q=G\cap H$ and $H=\langle a_1,\dots,a_k\rangle\x P$ as in Lemma~\ref{G--AP structure}. Let $H'=\langle a_1,\dots,a_s\rangle\x P$ be the subgroup with $\ker\rho\cu G\cap H'$ and $\G(\ker\rho)=\G(H')$ provided by Lemma~\ref{kernel envelope}.

Suppose first that $H=H'$. Then $\G(\ker\rho)=\G(H)=\G(Q)$ and Lem\-ma~\ref{normalisers of subgroups of AP}(1) implies that $N_{\A_{\G}}(\ker\rho)\leq N_{\A_{\G}}(H)$. In particular, $N_G(\ker\rho)\leq N_G(Q)$. In addition, $N_G(\ker\rho)$ contains the subgroup $\langle Q,Z_G(Q)\rangle$, which has finite index in $N_G(Q)$ by Lemma~\ref{special cc}(1). 

Suppose instead that $H'\lneq H$. Since $\rho$ takes values in $\Z$ and every homomorphism $\Z^2\ra\Z$ has nontrivial kernel, we must have $s=k-1$. Let $\pi\colon H\ra\Z$ be a homomorphism with $\ker\pi=H'$. Since $\G(Q)=\G(H)$, the restriction $\pi|_Q\colon Q\ra\Z$ is nontrivial and has kernel $G\cap H'$. Since $\ker\rho\cu G\cap H'=\ker\pi|_Q$, the homomorphism $\pi|_Q$ factors through $\rho$ and, since $\Z$ is Hopfian, we must have $\ker\rho=G\cap H'$. This shows that $\ker\rho$ is $G$--semi-parabolic. 

By Remark~\ref{LI properties}(6), the intersection $G\cap\langle a_k\rangle$ has finite index in $\langle a_k\rangle$. It follows that the subgroup $(G\cap H')\x(G\cap\langle a_k\rangle)=\ker\rho\x\Z$ has finite index in $G\cap H=Q$. This proves the proposition.
\end{proof}

We will also need the following.

\begin{lem}\label{UCC for centraliser kernels}
Let $\mc{K}$ be the collection of all subgroups of $G$ that are the kernel of a homomorphism $Q\ra\R$, where $Q$ varies among $G$--semi-parabolic subgroups. Then there exists a constant $N$, depending only on $G$, such that every chain of subgroups in $\mc{K}$ has length at most $N$.
\end{lem}
\begin{proof}
Choose $q\geq 1$ such that $G\leq\A_{\G}$ is $q$--convex-cocompact.

\smallskip
{\bf Claim~1:} \emph{there exists $N_1$ such that every chain of $G$--semi-parabolics has length at most $N_1$.}

\smallskip\noindent
\emph{Proof of Claim~1.}
Let $H_1,\dots,H_n$ be semi-parabolic subgroups of $\A_{\G}$ such that $G\cap H_1\lneq\dots\lneq G\cap H_n$. If the $H_i$ are chosen as in Lemma~\ref{G--AP structure}, then $H_1\lneq\dots\lneq H_n$. By Remark~\ref{AP chains rmk}, the latter chain has length bounded purely in terms of $\G$, proving the claim.
\hfill$\blacksquare$

\smallskip
{\bf Claim~2:} \emph{there exists $N_2$ such that every $G$--semi-parabolic subgroup has a generating set with at most $N_2$ elements.}

\smallskip\noindent
\emph{Proof of Claim~2.}
We begin by showing that, if $H=A\x P$ is a semi-parabolic subgroup of $\A_{\G}$, then the subgroup $(G\cap A)\x (G\cap P)$ has index at most $q$ in $G\cap H$.

Let $p_P\colon H\ra P$ be the factor projection. By Lemma~\ref{cc intersection}, $G\cap H$ is a convex-cocompact subgroup of $A\x P$, so $G\cap P$ has finite index in $p_P(G\cap H)$ (e.g.\ by Lemma~\ref{discrete factor}). By Lemma~\ref{special cc}(3) and Remark~\ref{fi overgroups rmk}, this index is at most $q$. Now, if $g,g'\in G\cap H$ are such that $p_P(g)$ and $p_P(g')$ are in the same coset of $G\cap P$, then $g$ and $g'$ are in the same coset of $(G\cap A)\x (G\cap P)$. It follows that $(G\cap A)\x (G\cap P)$ has index at most $q$ in $G\cap H$.

Now, by Lemma~\ref{special cc}(3) and \cite[Theorem~I.8.10]{BH}, there exists an integer $N_3$ such that every $G$--parabolic subgroup has a generating set with at most $N_3$ elements. Abelian subgroups of $G$ have rank at most $\dim\X_{\G}$. Along with the above observation, this shows that every $G$--semi-parabolic subgroup has a generating set with at most $q+\dim\X_{\G}+N_3$ elements.
\hfill$\blacksquare$

\smallskip
Now, consider a sequence of homomorphisms $\rho_i\colon Q_i\ra\R$, where each $Q_i$ is $G$--semi-parabolic and we have $\ker\rho_i\lneq\ker\rho_{i+1}$ for each $i$. By Lemma~\ref{AP intersections}, the group $Q_i\cap Q_{i+1}$ is again $G$--semi-parabolic and it contains $\ker\rho_i$. Thus, replacing $Q_i$ with $Q_i\cap Q_{i+1}$ and restricting $\rho_i$, we can assume that $Q_i\leq Q_{i+1}$. Repeating the procedure, we can ensure that the $Q_i$ form a chain without altering the kernels.

By Claim~1, there are at most $N_1$ distinct subgroups among the $Q_i$. So it suffices to consider the situation where all $Q_i$ are the same group $Q$. In this case, the $\rho_i$ descend to the abelianisation of $Q$, which has rank $\leq N_2$ by Claim~2. In conclusion, the chain of kernels has length at most $N_1(N_2+1)$.
\end{proof}

\subsection{$\om$--intersections of subgroups.}\label{om-intersection sect}

Let $\om$ be a non-principal ultrafilter on $\N$. Given a set $A$ and a sequence of subsets $A_i\cu A$, we denote their \emph{$\om$--intersection} by:
\[\bigcap_{\om}A_i=\{a\in A \mid a\in A_i \text{ for $\om$--all $i$}\}=\bigcup_{\om(J)=1}\ \bigcap_{i\in J} A_i.\]

\begin{rmk}\label{fg om-intersection}
Let $G$ be a group and let $H_i\leq G$ be a sequence of subgroups. If $\bigcap_{\om}H_i$ is finitely generated, then there exists $J\cu\N$ with $\om(J)=1$ such that $\bigcap_{\om}H_i=\bigcap_{i\in J}H_i$.

Indeed, suppose that $\bigcap_{\om}H_i$ is generated by elements $h_1,\dots,h_k$. There are subsets $J_s\cu\N$ with $\om(J_s)=1$ such that $h_s\in H_i$ for all $i\in J_s$. Thus it suffices to take $J:=J_1\cap\dots\cap J_k$. 
\end{rmk}

\begin{prop}\label{om-centralisers in G cor}
Let $G\leq\A_{\G}$ be convex-cocompact. Let $K_i\leq G$ be a sequence of subgroups. 
\begin{enumerate}
\item If all $K_i$ are $G$--semi-parabolic, then so is $\bigcap_{\om}K_i$.
\item If all $K_i$ are centralisers in $G$, then so is $\bigcap_{\om}K_i$.
\end{enumerate}
\end{prop}
\begin{proof}
We begin with part~(1). Let $H_i\leq\A_{\G}$ be semi-parabolic subgroups with $K_i=G\cap H_i$. Write $H_i=A_i\x P_i$ with $P_i$ parabolic and $A_i$ abelian. Since the $H_i$ are all full, $\bigcap_{\om}H_i$ is full, hence generated by the label-irreducibles that it contains. If $h\in\bigcap_{\om}H_i$ is label-irreducible, then either $h\in A_i$ for $\om$--all $i$, or $h\in P_i$ for $\om$--all $i$. This shows that $\bigcap_{\om}H_i$ is generated by $\bigcap_{\om}A_i$ and $\bigcap_{\om}P_i$. The former is clearly abelian, while the latter is parabolic by the characterisation in Proposition~\ref{parabolic prop}.

We conclude that $\bigcap_{\om}H_i$ is finitely generated. By Remark~\ref{fg om-intersection} and Lemma~\ref{AP intersections}, this is a semi-parabolic subgroup of $\A_{\G}$. Since $\bigcap_{\om}K_i=G\cap\bigcap_{\om}H_i$, this proves part~(1).

Regarding part~(2), recall that centralisers in $G$ are $G$--semi-parabolic by Remark~\ref{AP centralisers rmk}. If the $K_i$ are centralisers in $G$, part~(1) ensures that $\bigcap_{\om}K_i$ is finitely generated and so we can appeal again to Remark~\ref{fg om-intersection}. Intersections of centralisers are again centralisers, so this concludes the proof.
% note that the intersection between $G$ and a centraliser in $\A_{\G}$ needs not be a centraliser in $G$ !!!
\end{proof}

\section{Arc-stabilisers vs centralisers.}\label{perturbation sect}

Throughout this section, we fix the following setting.

\begin{ass}
Let $G\leq\A_{\G}$ be a $q$--convex-cocompact subgroup of a right-angled Artin group. We fix a $G$--invariant, $G$--essential convex subcomplex $Y\cu\X_{\G}$ on which $G$ acts with $q$ orbits of vertices $\mc{O}_1,\dots,\mc{O}_q$.
\end{ass}

Recall from the beginning of Subsection~\ref{special sect prelim} that $\X_{\G}$ admits various trees as restriction quotients $\pi_v\colon\X_{\G}\ra\T_v$, one for every vertex $v\in\G$. Note that $\pi_v(Y)\cu\T_v$ is either a single point fixed by $G$, or it is the unique $G$--minimal subtree of $\T_v$ (independently of the choice of $Y$).

\medskip
As discussed in the introduction, we are interested in understanding limits of sequences of $G$--trees consisting of $\T_v$ suitably rescaled and twisted by an automorphism of $G$. In order to identify arc-stabilisers of the limit $\R$--tree, it is necessary to gain a good understanding of arc-(almost-)stabilisers for each of the simplicial trees in the sequence. 

Arc-stabilisers of $G\acts\T_v$ are quite nice --- they are $G$--parabolic --- but this niceness will normally be lost when we twist $\T_v$ by an automorphism of $G$: the image of a $G$--parabolic subgroup under an automorphism of $G$ is not even convex-cocompact in general. By contrast, \emph{centralisers} (as in Definition~\ref{centralisers defn}) are much better behaved subgroups of $G$: we know that all automorphisms of $G$ take centralisers to centralisers, and that centralisers are always convex-cocompact.

Luckily, arcs of $\T_v$ can be perturbed so that their $G$--(almost-)stabiliser becomes a centraliser in $G$. The proof of this result is the main aim of this section. The precise statement is Corollary~\ref{arc-stabilisers are centralisers}, which we reproduce here as a theorem for the reader's convenience.

We emphasise that, without perturbing, it is still true that arc-stabilisers for $G\acts\T_v$ are the intersection between $G$ and the centraliser \emph{of a subset of $\A_{\G}$} (see Remark~\ref{perp rmk 2}). The point is that only centralisers \emph{of subsets of $G$} are well-behaved with respect to automorphisms of $G$.

\begin{thm}\label{main perturbation}
There exists a constant $L$, depending on $q$ and $\G$, with the following property. Every arc $\beta\cu\pi_v(Y)\cu\T_v$ with $\ell(\beta)>2L$ contains a sub-arc $\beta'\cu\beta$ with $\ell(\beta')\geq\ell(\beta)-2L$ such that:
\begin{enumerate}
\item either the stabiliser $G_{\beta'}$ is a centraliser in $G$, i.e.\ $Z_GZ_G(G_{\beta'})=G_{\beta'}$;
\item or $Z_GZ_G(G_{\beta'})=Z_G(g)$ for a label-irreducible element $g\in Z_G(G_{\beta'})$. 
\end{enumerate}
In the 2nd case, the element $g$ is loxodromic in $\T_v$ and its axis $\eta\cu\T_v$ satisfies $\ell(\eta\cap\beta')\geq\ell(\beta')-4q$. In addition, $\ell_Y(g)\leq q$ and $Z_G(g)$ contains $\langle g\rangle\x G_{\beta'}$ as a subgroup of index $\leq q$.
\end{thm}

\subsection{Decent pairs of hyperplanes.}

In this subsection, we introduce \emph{decent} pairs of hyperplanes of $Y$. Proposition~\ref{double centraliser prop} shows that stabilisers of decent pairs are (close to) centralisers in $G$. In the next subsections, we will see how to reduce general pairs of hyperplanes of $Y$ to decent ones.

For the following discussion, it is convenient to introduce the following notation.

\begin{defn}\label{hyperplane pair notation}
Given disjoint hyperplanes $\mf{u},\mf{w}\in\mscr{W}(\X_{\G})$, we write:
\begin{itemize}
\item $\mc{W}(\mf{u},\mf{w})=\mscr{W}(\mf{u}|\mf{w})\sqcup\{\mf{u},\mf{w}\}\cu\mscr{W}(\X_{\G})$;
\item $\Delta(\mf{u},\mf{w})=\g(\mc{W}(\mf{u},\mf{w}))\cu\G$;
\end{itemize}
\end{defn}

\begin{rmk}\label{perp rmk 2}
Let $\mf{u}$ and $\mf{w}$ be disjoint hyperplanes of $\X_{\G}$. If $\Delta=\Delta(\mf{u},\mf{w})$, then:
\begin{enumerate}
\item the subgroup of $\A_{\G}$ that stabilises $\mf{u}$ and $\mf{w}$ is conjugate to $\A_{\Delta^{\perp}}$;
\item $\Delta$ does not split as a nontrivial join.
\end{enumerate}
\end{rmk}

Recall that, if $\alpha\cu Y$ is a geodesic, $\mscr{W}(\alpha)\cu\mscr{W}(Y)$ is the set of hyperplanes that it crosses.

\begin{defn}\label{decent defn}
\begin{enumerate}
\item[]
\item A geodesic $\alpha\cu Y$ is \emph{decent} if, for every $v\in\g(\mscr{W}(\alpha))$, there exist an element $g_v\in G$ and a vertex $x_v\in\alpha$ such that $g_vx_v\in\alpha$ and $v\in\G(g_v)$. 
\item A pair of disjoint hyperplanes $\mf{u},\mf{w}\in\mscr{W}(Y)$ is \emph{decent} if there exists a decent geodesic $\alpha\cu Y$ with $\mscr{W}(\alpha)=\mc{W}(\mf{u},\mf{w})$. 
\end{enumerate}
\end{defn}

Given a hyperplane $\mf{w}\in\mscr{W}(\X_{\G})$, we denote its $G$--stabiliser by $G_{\mf{w}}$.

\begin{prop}\label{double centraliser prop}
Let $\mf{u},\mf{w}\in\mscr{W}(Y)$ be a decent pair of hyperplanes. Set $\Delta=\Delta(\mf{u},\mf{w})$. Then:
\begin{enumerate}
\item either $Z_GZ_G(G_{\mf{u}}\cap G_{\mf{w}})=G_{\mf{u}}\cap G_{\mf{w}}$;
\item or $Z_GZ_G(G_{\mf{u}}\cap G_{\mf{w}})=Z_G(g)$ for a label-irreducible element $g\in Z_G(G_{\mf{u}}\cap G_{\mf{w}})$. In this case, $\G(g)=\Delta$ and $g$ skewers all but at most $2q$ hyperplanes of $\mc{W}(\mf{u},\mf{w})$. In addition, $\ell_Y(g)\leq q$ and the subgroup $\langle g\rangle\x(G_{\mf{u}}\cap G_{\mf{w}})$ has index $\leq q$ in $Z_G(g)$.
\end{enumerate}
\end{prop}
\begin{proof}
Let $\alpha\cu Y$ be a decent geodesic with $\mscr{W}(\alpha)=\mc{W}(\mf{u},\mf{w})$. Replacing $\alpha,\mf{u},\mf{w}$ with their translates by an element of $\A_{\G}$ and conjugating $G\leq\A_{\G}$ accordingly, we can assume that the initial vertex of $\alpha$ is $1\in\A_{\G}$. In particular, $G_{\mf{u}}\cap G_{\mf{w}}=G\cap\A_{\Delta^{\perp}}$ (see Remark~\ref{perp rmk 2}).

For every $v\in\g(\mscr{W}(\alpha))=\Delta$, consider an element $g_v\in G$ and a point $x_v\in\alpha$ such that $g_vx_v\in\alpha$ and $v\in\G(g_v)$, as in Definition~\ref{decent defn}. 

Note that $\alpha\cu\A_{\Delta}\cu\X_{\G}$, so both $x_v$ and $g_vx_v$ lie in $\A_{\Delta}$. It follows that $g_v\in\A_{\Delta}$, and we can write $g_v=a_vh_va_v^{-1}$ as a reduced word with $h_v$ cyclically reduced and $a_v,h_v\in\A_{\Delta}$. We further separate $h_v=h_v'h_v''$, where $h_v'$ is the label-irreducible component of $h_v$ with $v\in\G(h_v')$, and $h_v''$ is the (possibly trivial) product of the remaining label-irreducible components of $h_v$. Let $C(h_v')$ be the maximal cyclic subgroup of $\A_{\G}$ containing $h_v'$.

Now, since $G_{\mf{u}}\cap G_{\mf{w}}$ fixes the set $\mscr{W}(\alpha)$ pointwise, Lemma~\ref{commutation criterion} implies that $g_v\in Z_G(G_{\mf{u}}\cap G_{\mf{w}})$ for every $v\in\Delta$. Thus:
\begin{align*}
Z_GZ_G(G_{\mf{u}}\cap G_{\mf{w}})&\leq\bigcap_{v\in\Delta}Z_{\A_{\G}}(g_v)=\bigcap_{v\in\Delta} a_vZ_{\A_{\G}}(h_v)a_v^{-1}\leq\bigcap_{v\in\Delta} a_vZ_{\A_{\G}}(h_v')a_v^{-1} \\
&=\bigcap_{v\in\Delta}a_v(C(h_v')\x\A_{\G(h_v')^{\perp}})a_v^{-1}=\bigcap_{v\in\Delta}a_vC(h_v')a_v^{-1}\x\bigcap_{v\in\Delta}a_v\A_{\G(h_v')^{\perp}}a_v^{-1} \\
&\leq\bigcap_{v\in\Delta}a_vC(h_v')a_v^{-1}\x\bigcap_{v\in\Delta}a_v\A_{\lk v}a_v^{-1}.
\end{align*}
Here, the second equality in the second line follows from Remark~\ref{full rmk}: indeed, Remark~\ref{perp rmk 2}(2) guarantees that the two sides contain exactly the same label-irreducibles.

Observe that $\bigcap_{v\in\Delta}a_v\A_{\lk v}a_v^{-1}=\A_{\Delta^{\perp}}$. Indeed, for every $v\in\Delta$, we have $\Delta^{\perp}\cu\lk v$. Since $a_v$ lies in $\A_{\Delta}$, it commutes with $\A_{\Delta^{\perp}}$. This shows that $\A_{\Delta^{\perp}}$ is contained in $P:=\bigcap_{v\in\Delta}a_v\A_{\lk v}a_v^{-1}$. Observing that $P$ is parabolic and $\G(P)\cu\bigcap_{v\in\Delta}\lk v=\Delta^{\perp}$, we conclude that $P=\A_{\Delta^{\perp}}$. 

Summing up, we have shown that:
\[G\cap\A_{\Delta^{\perp}}=G_{\mf{u}}\cap G_{\mf{w}}\leq Z_GZ_G(G_{\mf{u}}\cap G_{\mf{w}})\leq \left[\bigcap a_vC(h_v')a_v^{-1}\right]\x\A_{\Delta^{\perp}}.\]
If $Z_GZ_G(G_{\mf{u}}\cap G_{\mf{w}})$ is contained in $\A_{\Delta^{\perp}}$, then $G_{\mf{u}}\cap G_{\mf{w}}=Z_GZ_G(G_{\mf{u}}\cap G_{\mf{w}})$, and we are in the first case of the proposition.

Otherwise, $Z_GZ_G(G_{\mf{u}}\cap G_{\mf{w}})$ intersects $\bigcap a_vC(h_v')a_v^{-1}$ by Remark~\ref{LI properties}(6) (recall that centralisers are convex-cocompact). Let $h\in\A_{\G}$ be an element with $\langle h\rangle=\bigcap a_vC(h_v')a_v^{-1}$, and let $g$ be the smallest power of $h$ that lies in $G$.

It is clear that $g$ is label-irreducible and commutes with $G_{\mf{u}}\cap G_{\mf{w}}\leq\A_{\Delta^{\perp}}$. Since $v\in\G(h_v')\cu\Delta$ for every $v\in\Delta$, we must have $\G(g)=\Delta$. Remark~\ref{LI properties}(5) shows that $Z_G(g)=G\cap(\langle h\rangle\x\A_{\Delta^{\perp}})$. Since $Z_GZ_G(G_{\mf{u}}\cap G_{\mf{w}})$ is convex-cocompact and closed under taking roots in $G$, we conclude that:
\[Z_G(g)=Z_GZ_G(G_{\mf{u}}\cap G_{\mf{w}}).\]
We are left to prove the additional statements in the second case of the proposition.

Since $g$ lies in $Z_GZ_G(G_{\mf{u}}\cap G_{\mf{w}})$, it commutes with every element of the set:
\[A=\{k\in G \mid \exists x\in\alpha, \text{ s.t.}\ kx\in\alpha\}\cu Z_G(G_{\mf{u}}\cap G_{\mf{w}}).\]
In addition, for every $k\in A$, we have $\G(k)\cu\g(\mscr{W}(\alpha))=\Delta=\G(g)$. Thus, since $g$ is label-irreducible, Remark~\ref{LI properties}(4) applied to the label-irreducible components of $k$ shows that all $k\in A$ satisfy $\langle g,k\rangle\simeq\Z$. Since $g$ is the smallest power of $h$ that lies in $G$, we conclude that $A\cu\langle g\rangle$.

If $\mc{O}$ is a $G$--orbit with $\#(\mc{O}\cap\alpha)\geq 3$, then, since $A\cu\langle g\rangle$, there exists an axis of $g$ containing $\mc{O}\cap\alpha$. Let $\alpha_0\cu\alpha$ be the smallest subsegment that contains all intersections $\mc{O}_i\cap\alpha$, where $\mc{O}_i$ varies among $G$--orbits with $\#(\mc{O}_i\cap\alpha)\geq 3$. Since the union of all axes of $g$ forms a convex subcomplex $\Min(g)\cu\X_{\G}$, we have $\alpha_0\cu\Min(g)$. Since $\G(g)=\g(\mscr{W}(\alpha))$, the geodesic $\alpha_0$ cannot cross any hyperplanes separating distinct axes of $g$ (whose label would lie in $\G(g)^{\perp}$). Hence $\alpha_0$ is contained in the convex hull of a single axis of $g$, and every hyperplane crossed by $\alpha_0$ is skewered by $g$.

At most $2q$ vertices of $\alpha$ (and, therefore, at most $2q$ edges) can lie outside $\alpha_0$. It follows that $g$ skewers all but at most $2q$ hyperplanes in $\mscr{W}(\alpha)=\mc{W}(\mf{u},\mf{w})$.

Finally, note that $A$ contains an element $k$ with $\ell_Y(k)\leq q$ (for instance, consider $q+1$ consecutive vertices on $\alpha$). This implies that $\ell_Y(g)\leq q$. Recall that:
\[\langle g\rangle\x(G_{\mf{u}}\cap G_{\mf{w}})\leq Z_G(g) =G\cap\left[\langle h\rangle\x\A_{\Delta^{\perp}}\right].\]
Since $\ell_Y(g)\leq q$, we must have $g=h^n$ with $n\leq q$. Recalling that $G_{\mf{u}}\cap G_{\mf{w}}=G\cap\A_{\Delta^{\perp}}$, this shows that $\langle g\rangle\x(G_{\mf{u}}\cap G_{\mf{w}})$ has index $\leq q$ in $G\cap[\langle h\rangle\x\A_{\Delta^{\perp}}]$.

This concludes the proof of the proposition.
\end{proof}

\subsection{Decomposing geodesics in $Y$.}

In this subsection, we describe a procedure to decompose geodesics $\alpha\cu Y$ into a controlled number of better-behaved subsegments. The end result to keep in mind is Corollary~\ref{good decomposition}.

It is convenient to introduce the following (admittedly a bit heavy) terminology and notation. Luckily, this will not be required outside of this subsection.

\begin{defn}
Consider a geodesic $\alpha\cu Y$.
\begin{enumerate}
\item We denote by $0\leq o(\alpha)\leq q$ the number of orbits $\mc{O}_i$ with $\alpha\cap\mc{O}_i\neq\emptyset$.
\item For $v\in\G$ and $1\leq i\leq q$, look at the words (in the standard generators of $\A_{\G}$ and their inverses) spelled by the subsegments of $\alpha$ between consecutive points of $\alpha\cap\mc{O}_i$. We denote by $\rho_{i,v}(\alpha)\geq 0$ the number of such segments spelling words containing the letters $v^{\pm}$.
\item Define $n(\alpha):=\sum_i \#\{v\in\G \mid \rho_{i,v}(\alpha)\neq 0\}$.
\end{enumerate}
\end{defn}

\begin{defn}
Consider a geodesic $\alpha\cu Y$.
\begin{enumerate}
\item We say that $\alpha$ is \emph{almost $i$--excellent} if the endpoints of $\alpha$ lie in the same $\mc{O}_i$ and $\#(\alpha\cap\mc{O}_i)\geq 3$. The geodesic $\alpha$ is \emph{$i$--excellent} if, in addition, $\rho_{i,v}(\alpha)\neq 1$ for every $v\in\G$. We simply speak of \emph{(almost) excellent} geodesics when they are (almost) $i$--excellent for some $i$.
\item The geodesic $\alpha$ is \emph{almost good} if it is a union of almost excellent subsegments (possibly with large overlaps). 
%and, in addition, there exists no orbit $\mc{O}_i$ with $1\leq\#(\alpha\cap\mc{O}_i)\leq 2$. 
Similarly, $\alpha$ is \emph{good} if $\alpha$ is a union of excellent subsegments.
\end{enumerate}
\end{defn}

The following is the reason why we care about these properties.

\begin{lem}\label{good implies decent}
Good geodesics are decent.
\end{lem}
\begin{proof}
Since good geodesics are unions of excellent subsegments, it is enough to show that excellent geodesics are decent. So, consider an excellent geodesic $\alpha\cu Y$ and $v\in\g(\mscr{W}(\alpha))$. 

Let $\mc{O}$ be the $G$--orbit that contains the endpoints of $\alpha$. Then we can write the points of $\alpha\cap\mc{O}$, in the order in which they appear along $\alpha$, as:
\[x,\ g_1x,\ g_1g_2x,\ \ldots,\ g_1g_2\dots g_kx,\]
with all $g_i\in G$. Setting $a_i=x^{-1}g_ix\in\A_{\G}$, the points $1,a_1,a_1a_2,\ldots,a_1a_2\dots a_k$ lie on the geodesic $x^{-1}\alpha\cu\X_{\G}$. Note that $v\in\g(\mscr{W}(\alpha))=\g(\mscr{W}(x^{-1}\alpha))$, so $v\in\g(\mscr{W}(1|a_i))$ for some $i$.

Since $\alpha$ is excellent, there exists $j\neq i$ such that $v\in\g(\mscr{W}(1|a_j))$. Without loss of generality, we have $i<j\leq k$. Lemma~\ref{aligned2} guarantees that:
\[ v\in \G(a_1\dots a_i)\cup\G(a_1\dots a_j)\cup\G(a_{i+1}\dots a_j).\]
If $v\in \G(a_1\dots a_i)=\G(g_1\dots g_i)$, we can take $g_v=g_1\dots g_i$ and $x_v=x$. If instead $v\in \G(a_{i+1}\dots a_j)=\G(g_{i+1}\dots g_j)$, we set $x_v=g_1\dots g_ix$ and $g_v=(g_1\dots g_i)(g_{i+1}\dots g_j)(g_1\dots g_i)^{-1}$.
\end{proof}

In the rest of the subsection, we describe how to decompose general geodesics into good subsegments. To be precise, we say that $\alpha\cu Y$ is \emph{decomposed} into subsegments $\mu_1,\dots,\mu_r$ if $\alpha=\mu_1\cup\dots\cup\mu_r$ and $\mu_i\cap\mu_j$ is nonempty if and only if $|i-j|=1$, in which case $\mu_i\cap\mu_j$ is a single vertex.

\begin{lem}\label{dec not almost good}
If $\alpha\cu Y$ is not almost good, then $\alpha$ can be decomposed into at most $\max\{7,2o(\alpha)\}$ subsegments $\mu_j$ such that each satisfies one of the following:
\begin{itemize}
\item $\mu_j$ is a single edge;
\item $o(\mu_j)<o(\alpha)$.
\end{itemize}
\end{lem}
\begin{proof}
Set for simplicity $k=o(\alpha)$ and order the orbits so that $\mc{O}_1,\dots,\mc{O}_k$ are precisely those that intersect $\alpha$ nontrivially.

First, suppose that $\#(\alpha\cap\mc{O}_i)\leq 2$ for some $i\leq k$. Then we can decompose $\alpha$ into the $\leq 4$ edges that intersect $\alpha\cap\mc{O}_i$, plus the remaining $\leq 3$ subsegments of $\alpha$. Each of the latter intersects $\leq k-1$ orbits. In this case, we have decomposed $\alpha$ into $\leq 7$ subsegments with the required properties.

Thus, we can assume that $\#(\alpha\cap\mc{O}_i)\geq 3$ for all $i\leq k$. Let $\alpha_i\cu\alpha$ be the subsegment between the first and last points of $\alpha\cap\mc{O}_i$. Note that $\alpha_i$ is almost $i$--excellent.

Let $\alpha'\cu\alpha$ be the union of all $\alpha_i$. If $\alpha'$ were connected, then we would have $\alpha=\alpha'$ and $\alpha$ would be almost good. Thus, $\alpha'$ has between $2$ and $k$ connected components, with consecutive ones separated by a single open edge. Each component intersects $\leq k-1$ orbits. Then it suffices to decompose $\alpha$ into these components plus the remaining edges. These are $\leq 2k-1$ subsegments with the required properties.
\end{proof}

\begin{lem}\label{dec not good}
If $\alpha\cu Y$ is not good, then $\alpha$ can be decomposed into at most $\max\{7,2o(\alpha)\}$ subsegments $\mu_j$ such that each satisfies one of the following:
\begin{itemize}
\item $\mu_j$ is a single edge;
\item $o(\mu_j)<o(\alpha)$;
\item $o(\mu_j)=o(\alpha)$ and $n(\mu_j)<n(\alpha)$.
\end{itemize}
\end{lem}
\begin{proof}
Set again $k=o(\alpha)$ and let $\mc{O}_1,\dots,\mc{O}_k$ be the orbits that intersect $\alpha$ in at least $3$ vertices. Let $\alpha_i\cu\alpha$ be the subsegment between the first and last points of $\alpha\cap\mc{O}_i$. 

If $\alpha$ is not almost good, we simply apply Lemma~\ref{dec not almost good}. Suppose instead that $\alpha$ is almost good, so that $\alpha=\bigcup_i\alpha_i$. Since $\alpha$ is not good, one of the $\alpha_i$ is not excellent, hence there exist $1\leq j\leq q$ and $w\in\G$ such that $\rho_{j,w}(\alpha_j)=1$.

Let $I\cu\alpha_j$ be the only subsegment between consecutive points of $\alpha_j\cap\mc{O}_j$ in which $w$ appears. If $I$ is a single edge, we decompose $\alpha$ as the union of $I$ and two segments $\alpha^{\pm}$ . Otherwise, let $p$ be a vertex in the interior of $I$ and define $\alpha^{\pm}\cu\alpha$ as the two subsegments meeting at $p$. Observe that $\rho_{i,v}(\alpha^{\pm})\leq\rho_{i,v}(\alpha)$ for all $i$ and $v$, and $0=\rho_{j,w}(\alpha^{\pm})<\rho_{j,w}(\alpha)=1$. Thus $n(\alpha^{\pm})<n(\alpha)$.
\end{proof}

\begin{lem}\label{dec last lemma}
Every geodesic $\alpha\cu Y$ can be decomposed into at most $q\cdot(\max\{7,2o(\alpha)\})^{n(\alpha)}$ subsegments $\mu_j$ such that each satisfies one of the following:
\begin{itemize}
\item $\mu_j$ is a single edge;
\item $o(\mu_j)<o(\alpha)$;
\item $\mu_j$ is good.
\end{itemize}
\end{lem}
\begin{proof}
This follows from Lemma~\ref{dec not good} proceeding by induction on $n(\alpha)$. Note that a geodesic $\mu$ with $n(\mu)=0$ meets each $\mc{O}_i$ at most once and thus contains at most $q-1$ edges.
\end{proof}

\begin{cor}\label{good decomposition}
Setting $V:=\#\G^{(0)}$, every geodesic $\alpha\cu Y$ can be decomposed into at most $q^q\cdot(\max\{7,2q\})^{q^2V}$ subsegments $\mu_j$ such that each satisfies one of the following:
\begin{itemize}
\item $\mu_j$ is a single edge;
\item $\mu_j$ is decent (in fact, good).
\end{itemize}
\end{cor}
\begin{proof}
Note that $n(\alpha)\leq qV$. Thus, the number of subsegments in the decomposition provided by Lemma~\ref{dec last lemma} is at most $q\cdot(\max\{7,2q\})^{qV}$. Proceeding by induction on $o(\alpha)\leq q$, Lemmas~\ref{dec last lemma} and~\ref{good implies decent} yield the required conclusion.
\end{proof}

\subsection{Decomposing chains of hyperplanes.}

Let for simplicity $N_q=q^q\cdot(\max\{7,2q\})^{q^2V}$ be the constant in Corollary~\ref{good decomposition}. We say that hyperplanes $\mf{v}_1,\dots,\mf{v}_k$ form a \emph{chain} if, for each $i$, we can pick a halfspace $\mf{h}_i$ bounded by $\mf{v}_i$ so that $\mf{h}_1\subsetneq\dots\subsetneq\mf{h}_k$. 

Recall Definition~\ref{hyperplane pair notation}. It is convenient to introduce the following additional notation for a pair of disjoint hyperplanes $\mf{u},\mf{w}$ of $\X_{\G}$:
\begin{itemize}
\item $\delta(\mf{u},\mf{w})=\#\Delta(\mf{u},\mf{w})\in\N$;
\item $d_v(\mf{u},\mf{w})=\#(\g^{-1}(v)\cap\mscr{W}(\mf{u}|\mf{w}))\in\N$, where $v\in\G^{(0)}$.
\end{itemize}

Recall that, for every vertex $v\in\G$, we have a restriction quotient $\pi_v\colon\X_{\G}\ra\T_v$. The next lemma is saying that every geodesic in $\T_v$ can be decomposed into a bounded number of subpaths, which are alternately short and lower-complexity.

\begin{lem}\label{dec chains hyp}
Let $\mf{u},\mf{w}\in\mscr{W}(Y)$ be distinct hyperplanes with $\g(\mf{u})=\g(\mf{w})=v$. Suppose that $\mf{u},\mf{w}$ are not a decent pair. Then there exists a chain of hyperplanes $\mf{v}_0=\mf{u},\mf{v}_1,\dots,\mf{v}_{2s+1}=\mf{w}$, where:
\begin{itemize}
\item $\g(\mf{v}_i)=v$ for $0\leq i\leq 2s+1$;
\item $\delta(\mf{v}_{2j-1},\mf{v}_{2j})<\delta(\mf{u},\mf{w})$ for $1\leq j\leq s$;
\item $d_v(\mf{v}_{2j},\mf{v}_{2j+1})\leq N_q$ for $0\leq j\leq s$;
\item $s\leq N_q$.
\end{itemize}
\end{lem}
\begin{proof}
Let $\alpha\cu Y$ be a geodesic with $\mscr{W}(\alpha)=\mc{W}(\mf{u},\mf{w})$. Since $\mf{u},\mf{w}$ are not a decent pair, $\alpha$ is not decent. By Corollary~\ref{good decomposition}, we can decompose $\alpha$ into at most $N_q$ subsegments, each being either decent or a single edge. Among these, call $\mu_1,\dots,\mu_s$ the decent subsegments that cross at least $2$ hyperplanes labelled by $v$, in the order in which they appear along $\alpha$. Note that we must have $\g(\mscr{W}(\mu_j))\subsetneq\g(\mscr{W}(\alpha))$, otherwise $\alpha$ would be decent.

Define $\mf{v}_{2j-1}$ (resp.\ $\mf{v}_{2j}$) as the first (resp.\ last) hyperplane labelled by $v$ that is crossed by $\mu_j$. By the previous paragraph: 
 \[\delta(\mf{v}_{2j-1},\mf{v}_{2j})\leq\#\g(\mscr{W}(\mu_j))<\#\g(\mscr{W}(\alpha))=\delta(\mf{u},\mf{w}).\] 
 
Note that $\mu_j$ and $\mu_{j+1}$ are separated by $\leq N_q$ subsegments of $\alpha$, each crossing at most $1$ hyperplane labelled by $v$. This shows that $d_v(\mf{v}_{2j},\mf{v}_{2j+1})\leq N_q$, concluding the proof.
\end{proof}

\begin{cor}\label{dec chains hyp 2}
Let $\mf{u},\mf{w}\in\mscr{W}(Y)$ be distinct hyperplanes with $\g(\mf{u})=\g(\mf{w})=v$. Set $V=\#\G^{(0)}$. Then there exists a chain of hyperplanes $\mf{v}_0=\mf{u},\mf{v}_1,\dots,\mf{v}_{2s+1}=\mf{w}$, where:
\begin{itemize}
\item $\g(\mf{v}_i)=v$ for $0\leq i\leq 2s+1$;
\item $\mf{v}_{2j-1}$ and $\mf{v}_{2j}$ form a decent pair for $1\leq j\leq s$; % in fact, every geodesic joining them in their bridge is decent.
\item $d_v(\mf{v}_{2j},\mf{v}_{2j+1})\leq 2N_qV$ for $0\leq j\leq s$;
\item $s\leq N_q^V$.
\end{itemize}
\end{cor}
\begin{proof}
This follows from Lemma~\ref{dec chains hyp} by induction on $1\leq\delta(\mf{u},\mf{w})\leq V$. 
\end{proof}

Corollary~\ref{dec chains hyp 2} immediately implies:
% discarding the segments \nu_i that are too short

\begin{cor}\label{decent dec cor}
There exists a constant $L$, depending only on $q$ and $\G$, such that the following holds. Every arc $\beta\cu\pi_v(Y)\cu\T_v$ can be decomposed as a sequence of arcs $\mu_0\nu_1\mu_1\dots\nu_s\mu_s$ such that consecutive arcs share exactly one vertex and:
\begin{itemize}
\item the first and last edge of each arc $\nu_i$ correspond to a decent pair of hyperplanes of $Y$;
\item $\ell(\mu_i)\leq L$ and $\ell(\nu_i)>2q$ for every $i$; 
\item $s\leq L$.
\end{itemize} 
\end{cor}

Adding Proposition~\ref{double centraliser prop} to the above corollary, we obtain the desired result:

\begin{cor}\label{arc-stabilisers are centralisers}
Let $L$ be the constant in Corollary~\ref{decent dec cor}. Every arc $\beta\cu\pi_v(Y)\cu\T_v$ with $\ell(\beta)>2L$ contains a sub-arc $\beta'\cu\beta$ with $\ell(\beta')\geq\ell(\beta)-2L$ such that:
\begin{enumerate}
\item either $G_{\beta'}$ is a centraliser, i.e.\ $Z_GZ_G(G_{\beta'})=G_{\beta'}$;
\item or $Z_GZ_G(G_{\beta'})=Z_G(g)$ for a label-irreducible element $g\in Z_G(G_{\beta'})$. The element $g$ is loxodromic in $\T_v$ and its axis $\eta\cu\T_v$ satisfies $\ell(\eta\cap\beta')\geq\ell(\beta')-4q$. In addition, $\ell_Y(g)\leq q$, and the subgroup $\langle g\rangle\x G_{\beta'}$ has index $\leq q$ in $Z_G(g)$.
\end{enumerate}
\end{cor}
\begin{proof}
Decompose $\beta=\mu_0\nu_1\mu_1\dots\nu_s\mu_s$ as in Corollary~\ref{decent dec cor}. Define $\beta'$ as the sub-arc obtained by removing $\mu_0$ and $\mu_s$. It is clear that $\ell(\beta')\geq\ell(\beta)-2L$.

Proposition~\ref{double centraliser prop} shows that, for $i\in\{1,s\}$, one of the following two cases occurs:
\begin{enumerate}
\item either $Z_GZ_G(G_{\nu_i})=G_{\nu_i}$;
\item or $Z_GZ_G(G_{\nu_i})=Z_G(g_i)$ for a label-irreducible element $g_i\in Z_G(G_{\nu_i})$. The element $g_i$ is loxodromic in $\T_v$ with axis $\eta_i$ satisfying $\ell(\eta_i\cap\nu_i)\geq \ell(\nu_i)-2q>0$. In addition, $\ell_Y(g_i)\leq q$ and the subgroup $\langle g_i\rangle\x G_{\nu_i}$ has index $\leq q$ in $Z_G(g_i)$.
% g_i is loxodromic in \T_v because v\in\G(g)
\end{enumerate}
Based on this, there are three possibilities for $G_{\beta'}=G_{\nu_1}\cap G_{\nu_s}$. 
\begin{enumerate}
\item[(a)] Both $i=1$ and $i=s$ are of type~(1). Then $G_{\beta'}=Z_G\left(Z_G(G_{\nu_1})\cup Z_G(G_{\nu_s})\right)$ is a centraliser.
\item[(b)] Only one of them is of type~(1). Without loss of generality:
\begin{align*} 
Z_GZ_G(G_{\nu_1})&=Z_G(g_1)=G\cap\left[\langle h_1\rangle\x P_1\right], & Z_GZ_G(G_{\nu_s})&=G_{\nu_s}=G\cap P_s,
\end{align*}
where $P_1,P_s\leq\A_{\G}$ are parabolic, $\langle h_1\rangle$ is the maximal cyclic subgroup of $\A_{\G}$ containing $g_1$, and $G_{\nu_1}=G\cap P_1$. Since $g_1$ is loxodromic in $\T_v$, it does not lie in $G_{\nu_s}$, hence $h_1\not\in P_s$. It follows from Remark~\ref{full rmk} that $\left[\langle h_1\rangle\x P_1\right]\cap P_s=P_1\cap P_s$.

This shows that $Z_G(g_1)\cap Z_GZ_G(G_{\nu_s})=G\cap P_1\cap P_s=G_{\nu_1}\cap G_{\nu_s}$, so $G_{\beta'}$ is again a centraliser.
\item[(c)] Both $i=1$ and $i=s$ are of type~(2). Write again:
\begin{align*} 
Z_GZ_G(G_{\nu_1})=Z_G(g_1)&=G\cap\left[\langle h_1\rangle\x P_1\right], & Z_GZ_G(G_{\nu_s})=Z_G(g_s)&=G\cap\left[\langle h_s\rangle\x P_s\right].
\end{align*}
As in case~(b), we have $Z_G(g_1)\cap Z_G(g_s)=G\cap P_1\cap P_s=G_{\nu_1}\cap G_{\nu_s}$, except when $\langle h_1\rangle=\langle h_s\rangle$. 

In this case, we can assume that $g_1=g_s$ and simply call this element $g$. Note that $P_1=P_s$ and $G\cap P_i=G_{\nu_i}$. In particular, $G_{\beta'}=G_{\nu_1}=G_{\nu_s}$, hence $Z_GZ_G(G_{\beta'})=Z_G(g)$.

Finally, if $\eta\cu\T_v$ is the axis of $g$, we have seen that $\eta$ must intersect both $\nu_1$ and $\nu_s$ and that, in both cases, $\ell(\eta\cap\nu_i)\geq \ell(\nu_i)-2q$. This implies that $\ell(\eta\cap\beta')\geq\ell(\beta')-4q$.
\end{enumerate}
This concludes the proof.
\end{proof}

\subsection{Rotating actions.}

In this subsection, we record a consequence of Corollary~\ref{decent dec cor} that will be needed in the proof of Proposition~\ref{simplicial prop}.

\begin{defn}\label{rotating defn}
Consider a group $H$ and an action on a tree $H\acts T$.
\begin{enumerate}
\item We denote by $\mf{T}(H,T)\cu T$ the subtree $\Fix(H,T)$ if this is nonempty, and the $H$--minimal subtree otherwise.
\item We say that the action $H\acts T$ is \emph{$c$--rotating}, for some $c\geq 0$, if no element of $H\setminus\{1\}$ fixes an arc $\beta\cu T$ that is disjoint from $\mf{T}(H,T)$ and of length $>c$.
\end{enumerate}
\end{defn}

Recall that we are fixing a convex-cocompact subgroup $G\leq\A_{\G}$.

\begin{lem}\label{rotating lem}
There exists a constant $c=c(G)$ such that the following holds. Consider $v\in\G$ and a $G$--parabolic subgroup $P\leq G$ that is not elliptic in $\T_v$. Then, for every free factor $P_0\leq P$, the action $P_0\acts\mf{T}(P,\T_v)$ is $c$--rotating. 
\end{lem}
\begin{proof}
By Corollary~\ref{fin many conj parabolics}, there are only finitely many $G$--conjugacy classes of $G$--parabolic subgroups and they are all convex-cocompact in $\A_{\G}$. Thus, it suffices to prove the lemma with $P=G$. 

Let $L$ be the constant provided by Corollary~\ref{decent dec cor}. 
% Note that if $G$ is elliptic in $\T_v$, the corollary still holds: it's just that $\pi_v(Y)$ is a single point, so it contains no arcs
If $\beta\cu\mf{T}(G,\T_v)$ is an arc with $\ell(\beta)>L$, then it contains edges $e,e'$ corresponding to a decent pair of hyperplanes $\mf{w},\mf{w}'\in\mscr{W}(Y)$. This means that there exist an element $g_0\in G$ such that $v\in\G(g_0)$, and a point $x\in Y$ such that $\mscr{W}(x|g_0x)\cu\mc{W}(\mf{w},\mf{w}')$. By Lemma~\ref{commutation criterion}, this implies that $g_0$ commutes with the stabiliser $G_{\beta}\leq G$. Also note that $g_0$ is loxodromic in $\T_v$ and its axis shares at least one edge with $\beta$.

Now, consider a free factor $G_0\leq G$. If an element of $G_0\setminus\{1\}$ fixes $\beta$, then $g_0$ commutes with it and so we must have $g_0\in G_0$. Thus, the axis of $g_0$ is contained in $\mf{T}(G_0,\T_v)$ and $\beta$ must share a nontrivial arc with $\mf{T}(G_0,\T_v)$. This shows that the action $G_0\acts\mf{T}(G,\T_v)$ is $L$--rotating.
\end{proof}

\begin{cor}\label{rotating cor}
Let $c$ be as in Lemma~\ref{rotating lem}. Let $H\leq G$ be a convex-cocompact subgroup.
% previously, we required $H$ to be closed under taking roots, but it now seems unnecessary, since edge-stabilisers of $\T_v$ already are
Consider $v\in\G$ such that $H$ is elliptic in $\T_v$, but $N_G(H)$ is not. Then the action $N_G(H)\acts\mf{T}(N_G(H),\T_v)$ factors through an action of $N_G(H)/H$ and, for every free factor $N_0\leq N_G(H)/H$, the action $N_0\acts\mf{T}(N_G(H),\T_v)$ is $c$--rotating.
\end{cor}
\begin{proof}
Since $H$ is elliptic, $\Fix(H,\T_v)$ is nonempty and $N_G(H)$--invariant, hence it must contain the $N_G(H)$--minimal subtree. Thus, the action $N_G(H)\acts\mf{T}(N_G(H),\T_v)$ factors through $N_G(H)/H$. 

By Lemma~\ref{special cc}(1), $N_G(H)$ has a finite-index subgroup of the form $H\x P$, where $P$ is $G$--parabolic. Thus, $P$ projects injectively to a finite-index subgroup $\overline P\leq N_G(H)/H$. Note that:
\[\mf{T}(N_G(H),\T_v)=\mf{T}(H\x P,\T_v)=\mf{T}(P,\T_v).\] 

Let $N_0\leq N_G(H)/H$ be a free factor. Then $N_0\cap\overline P$ is a free factor of $\overline P$ and Lemma~\ref{rotating lem} shows that the action $N_0\cap\overline P\acts \mf{T}(N_G(H),\T_v)$ is $c$--rotating. 

Since $N_0\cap\overline P$ has finite index in $N_0$, we have $\mf{T}(N_0,\T_v)=\mf{T}(N_0\cap\overline P,\T_v)$. If $N_0$ is not elliptic, this is clear. If $N_0$ is elliptic, this is because edge-stabilisers of $\T_v$ are closed under taking roots.
% because hyperplane-stabilisers of $\X_{\G}$ are

Thus, if $\beta\cu\mf{T}(N_G(H),\T_v)$ is an arc of length $>c$ disjoint from $\mf{T}(N_0,\T_v)$, its $(N_0\cap\overline P)$--stabiliser is trivial, hence its $N_0$--stabiliser is finite. Again, since edge-stabilisers of $\T_v$ are closed under taking roots, this implies that the $N_0$--stabiliser of $\beta$ is trivial, showing that the action $N_0\acts\mf{T}(N_G(H),\T_v)$ is $c$--rotating.
\end{proof}

\section{Passing to the limit.}\label{limit sect}

This section is devoted to studying the limit $\R$--tree for a sequence $G\acts\T_v^{\phi_n}$, where $G\leq\A_{\G}$ is a convex-cocompact subgroup, $v\in\G$, and $\phi_n\in\out(G)$. This is carried out in Subsection~\ref{limit subsec}; in particular, see Propositions~\ref{arc-stab prop},~\ref{line and tripod prop} and~\ref{important addenda new}. 

Before that, in Subsections~\ref{almost stab subsec} and~\ref{tame sect}, we consider a more general setting: $G$ is an arbitrary group and we study limits of ``tame'' actions on simplicial trees (Definition~\ref{tame defn}).

\subsection{Almost-stabilisers.}\label{almost stab subsec}

Let $G$ be a group with an action $G\acts T$ on an $\R$--tree.

\begin{defn}
Consider an arc $\beta\cu T$ with endpoints $p,q$.
\begin{itemize}
\item For $0\leq s<\ell(\beta)/2$, we define: 
\[D(\beta,s)=\{g\in G\mid \max\{d(p,gp),d(q,gq)\}\leq s\}.\] 
We also consider the subgroup $\mf{D}(\beta,s):=\langle D(\beta,s)\rangle\leq G$. We write $D_G(\beta,s)$ and $\mf{D}_G(\beta,s)$ when it is necessary to specify the group under consideration.
%We refer to subgroups of this form as \emph{almost-stabilisers}.
\item For $0\leq t\leq\ell(\beta)$, we denote by $\beta[t]\cu\beta$ the middle sub-arc of length $t$. We also set $\beta^t:=\beta[\ell(\beta)-t]$; this is the closed sub-arc obtained by removing the initial and terminal segments of length $t/2$.
\end{itemize} 
\end{defn}

If $\beta\cu T$ is an arc, recall that $G_{\beta}\leq G$ denotes its stabiliser.

\begin{lem}\label{general almost-stabilisers}
Given an arc $\beta\cu T$ and $0\leq s<\ell(\beta)/2$:
% actually, the proof should give it for $s<\ell(\beta)$
\begin{enumerate}
\item for every $g\in D(\beta,s)$, we have $\beta^s\cu\Min(g,T)$;
\item either $D(\beta,s)$ contains a loxodromic, or $D(\beta,s)=G_{\beta^s}$.
\end{enumerate}
\end{lem}
\begin{proof}
We first prove part~(1). Let $x,y$ be the endpoints of $\beta$. For every $g\in G$, the midpoints of the arcs $[x,gx]$ and $[y,gy]$ lie in $\Min(g)$. If $s<\ell(\beta)/2$ and $g\in D(\beta,s)$, these two arcs are separated by the midpoint of $\beta$, and so are their midpoints. Since $\Min(g)$ is convex, it must contain the midpoint of $\beta$. Hence $\Min(g)\cap\beta$ is a sub-arc of $\beta$. Observing that $d(x,gx)\geq 2d(x,\Min(g))$, we deduce that $x$ and $y$ are at distance $\leq s/2$ from $\Min(g)$. Hence $\beta^s\cu\Min(g)$.

Regarding part~(2), suppose that every element of $D(\beta,s)$ is elliptic. Then part~(1) shows that $D(\beta,s)\leq G_{\beta^s}$. The reverse inclusion is clear.
\end{proof}

\begin{rmk}\label{convex metric rmk}
Consider points $x,y\in T$ and $g\in G$. Since the metric of $T$ is convex, we have $d(z,gz)\leq\max\{d(x,gx),d(y,gy)\}$ for every $z\in[x,y]$. 

In particular, given $\delta>0$, an arc $\beta\cu T$ and $0<t_1\leq t_2\leq\ell(\beta)$, we have $D(\beta[t_2],\delta)\cu D(\beta[t_1],\delta)$ and $\mf{D}(\beta[t_2],\delta)\leq\mf{D}(\beta[t_1],\delta)$.
\end{rmk}

\subsection{Tame actions.}\label{tame sect}

This subsection introduces the notion of \emph{tame} action on a tree. For sequences of tame actions, Proposition~\ref{tame sequences prop} allows us to understand arc-stabilisers of the limit in terms of those of the converging actions. In the next subsection, Corollary~\ref{special gives tame} shows that tameness is satisfied by convex-cocompact subgroups of right-angled Artin groups acting on the simplicial trees $\T_v$.

Let $G$ be any group.

\begin{defn}\label{tame defn}
An action on an $\R$--tree $G\acts T$ is \emph{$(\eps,N)$--tame}, for some $0<\eps<1/2$ and $N\geq 1$, if the following conditions are satisfied. Let $\mc{S}$ be the collection of subgroups of $G$ of the form $\mf{D}(\beta,\delta)$, where $\beta\cu T$ varies among nontrivial arcs and $\delta$ varies in the closed interval $[0,\eps\cdot\ell(\beta)]$.
\begin{enumerate} 
\item Every chain in $\mc{S}$ has length $\leq N$.
\item For every $0\leq\delta\leq\eps\cdot\ell(\beta)$ and every nontrivial arc $\beta\cu T$:
	\begin{itemize}
	\item either $\mf{D}(\beta,\delta)=D(\beta,\delta)=G_{\beta^{\delta}}$,
	\item or $D(\beta,\delta)$ contains a loxodromic whose axis is $\mf{D}(\beta,\delta)$--invariant.
	%and contains $\beta^{\delta}$ (now follows from Lemma~\ref{general almost-stabilisers})
	\end{itemize}
	We refer to the two cases as $\mf{D}(\beta,\delta)$ being \emph{elliptic} and \emph{non-elliptic}, respectively.
\item If $H_1,H_2\in\mc{S}$ and $H_1\lneq H_2$, then $H_1$ is elliptic. % if H is non-elliptic, it is the centraliser of its shortest label-irreducible loxodromic
\end{enumerate}
\end{defn}

\begin{defn}
Let $T$ be an $\R$--tree. We identify arcs in $T$ with points in $T\x T$ by taking endpoints.
\begin{enumerate}
\item A collection of arcs $\mc{P}\cu T\x T$ is \emph{$\delta$--dense} for some $\delta>0$ if, for every $x,y\in T$ satisfying $d(x,y)>2\delta$, there exists $(x',y')\in\mc{P}$ with $[x',y']\cu[x,y]$ and $\max\{d(x,x'),d(y,y')\}\leq\delta$.
\item Let $T_n$ be a sequence of $\R$--trees. A sequence of collections of arcs $\mc{P}_n\cu T_n\x T_n$ is \emph{eventually dense} if there exist $\delta_n>0$ such that each $\mc{P}_n$ is $\delta_n$--dense and $\delta_n\ra 0$.
\end{enumerate}
\end{defn}

Fix a non-principal ultrafilter $\om$ on $\N$ and recall the terminology from Subsection~\ref{ultrafilter sect}.

\begin{prop}\label{tame sequences prop}
Let $G\acts T_n$ be a sequence of $(\eps,N)$--tame actions $\om$--converging to $G\acts T_{\om}$.
\begin{enumerate}
\item Let $\beta\cu T_{\om}$ be an arc. Then we can choose a sequence of arcs $\beta_n\cu T_n$ converging to $\beta$ so that the following dichotomy holds.
	\begin{enumerate}
	\smallskip
	\item[(a)] Either there exists $0<\delta<\eps\cdot\ell(\beta)$ such that $\mf{D}(\beta_n,\delta)=G_{\beta_n}$ for $\om$--all $n$, and we have $G_{\beta}=\bigcap_{\om}G_{\beta_n}$.
	\item[(b)] Or, for every $0<\delta<\eps\cdot\ell(\beta)$, the subgroup $\mf{D}(\beta_n,\delta)$ is non-elliptic for $\om$--all $n$. 
	
	In this case, $G_{\beta}$ leaves invariant a line $\alpha\cu T_{\om}$ containing $\beta$. The $G$--stabiliser of $\alpha$ equals $\bigcap_{\om}\mf{D}(\beta_n,\delta)$ (independently of the choice of $\delta$), 
	%  note that $\bigcap_{\om}D(\beta_n,\delta)$, by contrast, must fix $\alpha$ pointwise (D and \mf{D} are different!!)
	and $G_{\beta}$ is the kernel of the (possibly trivial) homomorphism $\bigcap_{\om}\mf{D}(\beta_n,\delta)\ra\R$ given by translation lengths along $\alpha$.
	\end{enumerate}
\smallskip
\item Let $\g\cu T_{\om}$ be a line. Then we can choose a sequence of arcs $\beta_n\cu T_n$ converging to $\g$ so that, for every sufficiently large $\delta>0$, the $G$--stabiliser of $\g$ equals $\bigcap_{\om}\mf{D}(\beta_n,\delta)$.
\end{enumerate}
Moreover, if $\mc{P}_n\cu T_n\x T_n$ is an eventually dense sequence of collections of arcs, then the approximations $\beta_n$ in~(1) and~(2) can be chosen within $\mc{P}_n$.
\end{prop}
\begin{proof}
We will initially deal with both parts of the proposition simultaneously: set $\eta:=\beta$ in part~(1) and $\eta:=\g$ in part~(2). Let $\eta_n\cu T_n$ be a sequence of arcs converging to $\eta$. Recall that $\eta_n[s]$ is the middle segment of $\eta_n$ of length $s$. 

Consider $0<\delta<\eps\cdot\ell(\eta)$. % both strict inequalities are important
If there exists $\tfrac{\delta}{\eps}\leq s\leq\ell(\eta_n)$ such that $\mf{D}(\eta_n[s],\delta)$ is non-elliptic, let $\tfrac{\delta}{\eps}\leq t_{n,\delta}\leq\ell(\eta_n)$ be the largest such $s$ (the maximum exists e.g.\ by~(3) in Definition~\ref{tame defn} and Remark~\ref{convex metric rmk}). Otherwise, set $t_{n,\delta}=0$. Let $t_{\delta}$ be the $\om$--limit of $t_{n,\delta}$. 

If $\delta_1<\delta_2$, we have $t_{n,\delta_1}\leq t_{n,\delta_2}$ for every $n$, hence $t_{\delta_1}\leq t_{\delta_2}$. We will need the following observation.

\smallskip
{\bf Claim:} \emph{Suppose that, for some $\delta$ and $n$, we have $t_{n,\delta}\neq 0$. Then there exists a subgroup $H_n\leq G$ such that $H_n=\mf{D}(\eta_n[s],\delta')$ for all $\delta\leq\delta'\leq\eps t_{n,\delta}$ and $\tfrac{\delta'}{\eps}\leq s\leq t_{n,\delta}$. There is a unique $H_n$--invariant line $\alpha_n\cu T_n$ and it contains $\eta_n[t_{n,\delta}-\delta]$.}

\smallskip\noindent
\emph{Proof of claim.}
Fix for a moment $\delta'\geq\delta$ and recall that $t_{n,\delta'}\geq t_{n,\delta}$. By~(3) in Definition~\ref{tame defn} and Remark~\ref{convex metric rmk}, the subgroup $\mf{D}(\eta_n[s],\delta')$ is non-elliptic and constant as $s$ varies in $[\tfrac{\delta'}{\eps},t_{n,\delta}]$. Let us call it $H_{n,\delta'}$ for short. By~(2) in Definition~\ref{tame defn}, $H_{n,\delta'}$ leaves invariant a line $\alpha_{n,\delta'}\cu T_n$, which must contain the arc $\eta_n[t_{n,\delta}-\delta']$ by Lemma~\ref{general almost-stabilisers}. 

Taking $s=t_{n,\delta}$, the fact that $\delta\leq\delta'$ implies that $H_{n,\delta}\leq H_{n,\delta'}$. Applying again (3) in Definition~\ref{tame defn}, we deduce that $H_{n,\delta}=H_{n,\delta'}$, so this subgroup is independent of the specific value of $\delta'$ and we can call it $H_n$. The lines $\alpha_{n,\delta'}$ are also independent of $\delta'$, since they are the axis of all loxodromics in $H_{n,\delta'}=H_n$. This proves the claim.
\hfill$\blacksquare$

\smallskip
Now, we distinguish three cases.

\smallskip
{\bf Case~A:} \emph{$\eta=\g$ and there exists $\delta_0>0$ with $t_{\delta_0}=+\infty$.}

\smallskip
We must have $t_{n,\delta_0}>0$ for $\om$--all $n$, so the claim provides subgroups $H_n\leq G$ and lines $\alpha_n\cu T_n$. Since $\alpha_n$ contains $\eta_n[t_{n,\delta_0}-\delta_0]$ and $t_{n,\delta_0}$ diverges, the lines $\alpha_n$ converge to $\g$. Since $\alpha_n$ is $H_n$--invariant, it is clear that the subgroup $H:=\bigcap_{\om}H_n$ leaves $\g$ invariant.

Let us show that $H$ coincides with the $G$--stabiliser of $\g$. Consider $g\in G$ with $g\g=\g$. Let $\alpha_n[s]$ denote the sub-arc of $\alpha_n$ of length $s$ that has the same midpoint as $\eta_n$. Choosing $s$ so that $\ell_{T_{\om}}(g)<\eps s$, we have $g\in\bigcap_{\om}D(\alpha_n[s],\eps s)$. If $s$ is large enough, we have $\delta_0<\eps s$ and, since $D(\eta_n[t_{n,\delta_0}],\delta_0)$ leaves $\alpha_n$ invariant, it follows that $D(\alpha_n[s],\eps s)\supseteq D(\eta_n[t_{n,\delta_0}],\delta_0)$. Thus, by~(3) in Definition~\ref{tame defn}, we have $\mf{D}(\alpha_n[s],\eps s)=\mf{D}(\eta_n[t_{n,\delta_0}],\delta_0)=H_n$. This shows that $g\in\bigcap_{\om}H_n=H$, as required.

Taking for instance $\beta_n=\eta_n[t_{n,\delta_0}]$ (or slightly smaller arcs lying in $\mc{P}_n$), this proves part~(2) of the proposition in this case. Note that the other possibility for part~(2) is easier. If $t_{\delta}<+\infty$ for every $\delta>0$, then the stabiliser of $\g$ actually fixes $\g$ pointwise, and it is easy to see that it coincides with $\bigcap_{\om}\mf{D}(\beta_n,\delta)=\bigcap_{\om}G_{\beta_n}$ for every $\delta>0$. 

We continue with part~(1) of the proposition.

\smallskip
{\bf Case~B:} \emph{$\eta=\beta$ and $t_{\delta}=\ell(\beta)$ for every $0<\delta<\eps\cdot\ell(\beta)$.}

\smallskip
As in the previous case, the claim provides subgroups $H_n\leq G$ and $H_n$--invariant lines $\alpha_n\cu T_n$. Set $H:=\bigcap_{\om}H_n$. For every $\delta>0$, the line $\alpha_n$ contains the arc $\eta_n[t_{n,\delta}-\delta]$ for $\om$--all $n$. It follows that the $\alpha_n$ converge to an $H$--invariant line $\alpha\cu T_{\om}$ containing $\beta$. Exactly as in the previous case, one shows that $H$ is actually the entire $G$--stabiliser of $\alpha$.

We can define $\beta_n:=\alpha_n[\ell(\beta)]$ (or take a slightly smaller arc lying in $\mc{P}_n$). It is clear that $\beta_n$ converges to $\beta$ and, for each choice of $\delta$, we have $\mf{D}(\beta_n,\delta)=H_n$ for $\om$--all $n$. 

Since the $\beta_n$ approximate $\beta$ and $\delta>0$, we have $G_{\beta}\cu \bigcap_{\om}D(\beta_n,\delta) \cu H$.
% no finite generation issues here, it would seem
Let $\tau_n\colon H_n\ra\R$ be the homomorphism given by translation lengths along $\alpha_n$. The limit $\tau_{\om}\colon H=\bigcap_{\om}H_n\ra\R$ gives translation lengths along $\alpha\cu T_{\om}$. Since $G_{\beta}\leq H$, it is clear that $G_{\beta}$ is the kernel of $\tau_{\om}$. 

Thus, this case corresponds to case~(b) of part~(1) of the proposition. We are left to consider one last situation.

\smallskip
{\bf Case~C:} \emph{$\eta=\beta$ and there exists $\delta_0>0$ such that $t_{\delta_0}<\ell(\beta)$.}

\smallskip
By Definition~\ref{tame defn}, for each $n$ there exist $k\leq N$ and values $\ell(\eta_n)=s_{0,n}>s_{1,n}>\dots>s_{k,n}>0$ such that the $G$--stabiliser of $\eta_n[s]$ is constant as $s$ varies in each interval $s\in(s_{i+1,n},s_{i,n}]$. As $n$ varies, $k$ is $\om$--constant and the $s_{i,n}$ converge to a sequence $\ell(\beta)=s_0\geq s_1\geq\dots\geq s_k\geq 0$.

Let $j$ be the largest index with $s_j=\ell(\beta)$. Up to shrinking the approximation $\eta_n$, we can assume that $j=0$ (and that $\eta_n$ lies within $\mc{P}_n$). Then, for every $s_1< s\leq\ell(\beta)$, the $G$--stabilisers of $\eta_n[s]$ and $\eta_n$ coincide for $\om$--all $n$.

Consider $\delta>0$ and $g\in G$ with $g\beta=\beta$. If $s\leq\ell(\beta)$, we have $g\in\bigcap_{\om}D(\eta_n[s],\delta)\cu\bigcap_{\om}\mf{D}(\eta_n[s],\delta)$. If $\delta<\delta_0$ we have $t_{\delta}\leq t_{\delta_0}$. Thus, if $s>t_{\delta_0}$, the subgroup $\mf{D}(\eta_n[s],\delta)$ is elliptic for $\om$--all $n$, and so it coincides with the $G$--stabiliser of $\eta_n[s-\delta]$ by~(2) in Definition~\ref{tame defn}. If $s>\delta+s_1$, the latter equals $G_{\eta_n}$. In conclusion, when $\delta$ is small enough and $s$ is large enough, we have shown that the $G$--stabiliser of $\eta$ is contained in $\bigcap_{\om}G_{\eta_n}$, hence it coincides with it. 

Taking $\beta_n:=\eta_n$, this corresponds to case~(a) of part~(1). This concludes the proof of the proposition.
\end{proof}

\subsection{Almost-stabilisers in special groups.}

Consider a right-angled Artin group $\A_{\G}$ and set $r:=\dim\X_{\G}$. Recall that, for every $v\in\G$, we have an action $\A_{\G}\acts\T_v$ coming from a restriction quotient of $\X_{\G}$. As usual, the stabiliser of an arc $\beta\cu\T_v$ is denoted $(\A_{\G})_{\beta}$.

\begin{lem}\label{almost-stabiliser new}
Consider an arc $\beta\cu\T_v$ and $0\leq\delta\leq\tfrac{\ell(\beta)}{4r+2}$. Then:
\begin{enumerate}
\item either $D_{\A_{\G}}(\beta,\delta)=(\A_{\G})_{\beta^{\delta}}$;
\item or $(\A_{\G})_{\beta^{\delta}}\subsetneq D_{\A_{\G}}(\beta,\delta)\cu\langle h\rangle\x (\A_{\G})_{\beta^{\delta}}=Z_{\A_{\G}}(h)$ for a label-irreducible element $h\in\A_{\G}$ that is not a proper power. Moreover, $0<\ell_{\T_v}(h)\leq\delta$ and the axis of $h$ in $\T_v$ contains $\beta^{\delta}$.
\end{enumerate}
\end{lem}
\begin{proof}
Let $p,q$ be the endpoints of $\beta$ and let $p',q'$ be those of $\beta^{\delta}$. Set $D=D_{\A_{\G}}(\beta,\delta)$ for simplicity. Let $D_0\cu D$ be the subset of elliptic elements. By Lemma~\ref{general almost-stabilisers}, we have $D_0=(\A_{\G})_{\beta^{\delta}}$. Let $D_1\cu D$ be the subset of loxodromic elements that, in addition, are label-irreducible. 

\smallskip
{\bf Claim~1:} \emph{every $g\in D$ can be decomposed as a product $h_1h_0$ with $h_1\in D_1\sqcup\{1\}$ and $h_0\in D_0$, where $h_0$ commutes with $h_1$. If $h_1\neq 1$, it has the same axis and translation length in $\T_v$ as $g$.}

\smallskip\noindent
\emph{Proof of Claim~1.}
Consider $g\in D$. If $g$ is elliptic, we can take $h_1=1$ and $h_0=g$. Suppose instead that $g$ is loxodromic, and let $g=g_1\cdot\ldots\cdot g_k$ be its decomposition into label-irreducibles. 

Since the $g_i$ commute pairwise, at least one of them must be loxodromic in $\T_v$, or $g$ would be elliptic. Since the sets $\G(g_i)$ are pairwise disjoint, at most one $g_i$ is loxodromic in $\T_v$. Say $g_1$ is the loxodromic component. Then its axis is fixed pointwise by $g_2,\dots,g_k$, so $g$ has the same axis and the same translation length in $\T_v$ as $g_1$. We can then set $h_1=g_1$ and $h_2=g_2\cdot\ldots\cdot g_k$.

We are only left to check that $h_1,h_0$ lie in $D$. Since $h_1$ and $g$ have the same axis and translation length, they displace all points of $\T_v$ by the same amount. Hence $h_1\in D$. By Lemma~\ref{general almost-stabilisers}, the axis of $g$ contains $\beta^{\delta}$. This coincides with the axis of $h_1$, which is fixed pointwise by $h_0$. So $h_0$ fixes $\beta^{\delta}$ pointwise, hence $h_0\in D$.  
\hfill$\blacksquare$

\smallskip
If $D_1=\emptyset$, then we are in the first case of the lemma and we are done.

\smallskip
{\bf Claim~2.} \emph{if $D_1\neq\emptyset$, there exists $h\in D_1$ such that $h$ is not a proper power in $\A_{\G}$ and $D_1\cu\langle h\rangle$.}

\smallskip\noindent
\emph{Proof of Claim~2.}
Recall that $d(p',q')=\ell(\beta)-\delta$. If $h\in D_1$, the points $p'$ and $q'$ lie on the axis of $h$ by Lemma~\ref{general almost-stabilisers}. Hence $d(p',hp')\leq d(p,hp)\leq\delta$. By our choice of $\delta$, we have:
\[d(p',h^{4r+1}p')\leq (4r+1)\delta\leq \ell(\beta)-\delta=d(p',q'),\] 
so the point $h^{4r+1}p'$ lies between $p'$ and $q'$ (up to replacing $h$ with $h^{-1}$). In conjunction with \cite[Lemma~3.13]{Fio10a}, this shows that the subgroup of $\A_{\G}$ generated by any two elements of $D_1$ is cyclic. Finally, it is clear that $D_1$ is closed under taking roots. The claim follows.
\hfill$\blacksquare$

\smallskip
In conclusion, we have shown that $D_{\A_{\G}}(\beta,\delta)\cu\langle h\rangle\x (\A_{\G})_{\beta^{\delta}}\cu Z_{\A_{\G}}(h)$. Since $h\in D$, it is clear that $\ell_{\T_v}(h)\leq\delta$ and the axis of $h$ in $\T_v$ contains $\beta^{\delta}$. We are only left to prove that $Z_{\A_{\G}}(h)\cu\langle h\rangle\x (\A_{\G})_{\beta^{\delta}}$.

Since $h$ is label-irreducible, Remark~\ref{LI properties}(5) shows that $Z_{\A_{\G}}(h)=\langle h\rangle\x P$ for some parabolic subgroup $P$. For every $g\in P$, we have $\G(g)\cu\G(h)^{\perp}$, so $v\not\in\G(g)$. It follows that $P$ is elliptic in $\T_v$. Since $P$ commutes with $h$, it must fix the axis of $h$, which, in turn, contains $\beta^{\delta}$.  Hence $P$ is contained in $(\A_{\G})_{\beta^{\delta}}$, which completes the proof.
\end{proof}

\begin{cor}\label{D for G}
Let $\beta$ and $\delta$ be as in Lemma~\ref{almost-stabiliser new}. Let $G\leq\A_{\G}$ be $q$--convex-cocompact. Then:
\begin{enumerate}
\item either $\mf{D}_G(\beta,\delta)=G_{\beta^{\delta}}$; 
\item or $\mf{D}_G(\beta,\delta)=Z_G(g)$, for a label-irreducible element $g\in G$. In this case, $0<\ell_{\T_v}(g)\leq\delta q$ and the axis of $g$ in $\T_v$ contains $\beta^{\delta}$. 
%Moreover, $Z_G(g)$ has a subgroup of index $\leq q$ of the form $\langle g\rangle\x G_{\beta^{\delta}}$.
\end{enumerate}
\end{cor}
\begin{proof}
If we are in case~(1) of Lemma~\ref{almost-stabiliser new}, it is clear that $G\cap D_{\A_{\G}}(\beta,\delta)\cu G_{\beta^{\delta}}$, hence $\mf{D}_G(\beta,\delta)=G_{\beta^{\delta}}$. So, let us suppose that we are in case~(2) of Lemma~\ref{almost-stabiliser new} and $D_{\A_{\G}}(\beta,\delta)\cu\langle h\rangle\x (\A_{\G})_{\beta^{\delta}}=Z_{\A_{\G}}(h)$ for a label-irreducible element $h\in\A_{\G}$. If $G\cap D_{\A_{\G}}(\beta,\delta)$ is contained in $\{1\}\x(\A_{\G})_{\beta^{\delta}}$, we again obtain $G\cap D_{\A_{\G}}(\beta,\delta)\cu G_{\beta^{\delta}}$ and $\mf{D}_G(\beta,\delta)=G_{\beta^{\delta}}$.

Otherwise, an element of $G\cap D_{\A_{\G}}(\beta,\delta)$ has a label-irreducible component that is a power of $h$. Remark~\ref{LI properties}(6) shows that there exists $1\leq k\leq q$ such that $h^k\in G$. Let $g$ be the smallest such power of $h$. Note that $\ell_{\T_v}(g)\leq q\ell_{\T_v}(h)\leq \delta q$. The axis of $g$ in $\T_v$ coincides with that of $h$, so it contains $\beta^{\delta}$. It is clear that:
\[\mf{D}_G(\beta,\delta)=\langle G\cap D_{\A_{\G}}(\beta,\delta)\rangle \leq G\cap Z_{\A_{\G}}(h)= Z_G(g).\]
Finally, if this inclusion were strict, then $Z_G(g)\setminus\mf{D}_G(\beta,\delta)$ would contain an element of $\langle h\rangle\x (\A_{\G})_{\beta^{\delta}}$ with the same axis as $g$ and strictly smaller translation length, a contradiction.
\end{proof}

\begin{cor}\label{special gives tame}
If $G\leq\A_{\G}$ is convex-cocompact, then $G\acts\T_v$ is $(\tfrac{1}{4r+2},N)$--tame for some $N\geq 1$.
\end{cor}
\begin{proof}
We verify the three conditions in Definition~\ref{tame defn}. 
Condition~(2) is immediate from Corollary~\ref{D for G} (note that the loxodromic required by the condition might not be the element $g$ from Corollary~\ref{D for G}, but rather any shortest loxodromic in $G\cap D_{\A_{\G}}(\beta,\delta)$). Condition~(1) follows from (a special case of) Lemma~\ref{UCC for centraliser kernels}, since stabilisers of arcs of $\T_v$ are $G$--parabolic and centralisers are $G$--semi-parabolic. Finally, let $g_1,g_2\in G$ be label-irreducibles with $Z_G(g_1)\lneq Z_G(g_2)$. Then $\langle g_1\rangle\neq\langle g_2\rangle$, hence $\G(g_1)\cu\G(g_2)^{\perp}$ by Remark~\ref{LI properties}(4). This shows that $g_1,g_2$ cannot both be loxodromic in $\T_v$, which implies Condition~(3). 
\end{proof}

\subsection{Arc-stabilisers in the limit.}\label{limit subsec}

Let $G\leq\A_{\G}$ be convex-cocompact and let $Y\cu\X_{\G}$ be a convex subcomplex on which $G$ acts essentially and cocompactly. Let $S\cu G$ be a finite generating set. For every vertex $v\in\G$, consider the restriction quotient $\pi_v\colon\X_{\G}\ra\T_v$. 

Let $\varphi_n\in\aut(G)$ be a sequence of automorphisms projecting to an infinite sequence in $\out(G)$. The standard Bestvina--Paulin argument \cite{Bestvina-Duke,Paulin-IM} (see in particular \cite[Section~2, p.~338, Case~1]{Paulin-arboreal}) guarantees that the quantity:
\[\tau_n:=\inf_{x\in\X_{\G}}\max_{s\in S}d(x,\varphi_n(s)x)\]
diverges for $n\ra+\infty$. Let $o_n\in Y\cu\X_{\G}$ be points realising these infima. For any $G$--metric space $Z$, we let $Z^{\varphi_n}$ represent $Z$ with the twisted $G$--action $g\cdot x:=\varphi_n(g)x$.
% we don't need to change specific representatives of $\varphi_n$ in their outer class, because we don't mind the $o_n$ not being constant

Fix a non-principal ultrafilter $\om$. Let $(G\acts\X_{\om},o)$ be the $\om$--limit of the sequence $(G\acts\tfrac{1}{\tau_n}\X_{\G}^{\varphi_n},o_n)$. Note that the action $G\acts \X_{\om}$ does not have a global fixed point: because of our choice of $\tau_n$, every point of $\X_{\om}$ is displaced by at least one element of $S$. It follows that the action $G\acts\X_{\om}$ has unbounded orbits, for instance because $\X_{\om}$ has a bi-Lipschitz equivalent $G$--invariant $\CAT$ metric (the limit of the $\CAT$ metric on $\X_{\G}$). 

Since $\X_{\G}$ embeds isometrically and $\A_{\G}$--equivariantly in the finite product $\prod_{v\in\G}\T_v$, the limit $\X_{\om}$ embeds isometrically and $G$--equivariantly in $\prod_{v\in\G}\T_v^{\om}$, where $T_v^{\om}$ is the $\om$--limit of the sequence $(G\acts \tfrac{1}{\tau_n}T_v^{\varphi_n},\pi_v(o_n))$. Since $G\acts\X_{\om}$ has unbounded orbits, there exists a vertex $v\in\G$ such that the action $G\acts\T_v^{\om}$ is not elliptic.

\begin{rmk}
It is not hard to show that there exists a vertex $v\in\G$ such that
\[ \inf_{x\in \T_v}\max_{s\in S}d(x,\varphi_n(s)x)\geq c(\G)\cdot\tau_n,\]
where $c(\G)$ is a constant depending only on $\G$. Without this inequality, it might happen (a priori) that the non-elliptic limit tree $\T_v^{\om}$ has a $G$--fixed point at infinity. Indeed, $\pi_v(o_n)$ might be far from the point realising $\inf_{x\in \T_v}\max_{s\in S}d(x,\varphi_n(s)x)$, as $Y$ is not convex in the product $\prod_{v\in\G}\pi_v(Y)$. In any case, there is no need to rule out this possibility, as it is irrelevant to the following discussion.
\end{rmk}

In the rest of the subsection, we consider the following setting.

\begin{ass}\label{trees notation}
Fix a vertex $v\in\G$ such that $G\acts\T_v^{\om}$ is not elliptic. For simplicity, we set $T_G:=\pi_v(Y)$, which is the $G$--minimal subtree of $\T_v$. Denote by $G\acts T_n$ the action on $T_G^{\varphi_n}$ with its metric rescaled by $\tau_n$. We also set $T_{\om}:=\T_v^{\om}$, which is the $\om$--limit of $(T_n,\pi_v(o_n))$. 
\end{ass}

We emphasise that the $G$--action on $T_{\om}$ will not be minimal in general ($T_{\om}$ is the universal $\R$--tree as soon as it is not a line).

The following characterises arc-stabilisers for the action $G\acts T_{\om}$. Recall Definition~\ref{centralisers defn}.

\begin{prop}\label{arc-stab prop}
For every arc $\beta\cu T_{\om}$, at least one of the following two options occurs:
\begin{enumerate}
\item $G_{\beta}$ is a centraliser in $G$ and $G_{\beta}$ is elliptic in $\om$--all $T_n$;
\item $G_{\beta}$ is the kernel of a (possibly trivial) homomorphism $\rho\colon Z\ra\R$, where $Z\leq G$ is a centraliser. In addition, $Z$ is the $G$--stabiliser of a line $\alpha\cu T_{\om}$ containing $\beta$, and $\rho$ gives translation lengths along $\alpha$. The $Z$--minimal subtree of $T_n$ is a line for $\om$--all $n$ (hence $Z$ is non-elliptic in $T_n$) and we have $N_G(Z)=Z$.
\end{enumerate}
\end{prop}
\begin{proof}
Recall that the trees $T_n$ and $T_G=\pi_v(Y)$ coincide up to rescaling, but are endowed with different $G$--actions. It is therefore convenient to adopt the following convention: if $\eta$ is an arc in $T_n$, we denote by $\wt\eta$ the corresponding arc of $T_G$. We will always write either $\mf{D}_{T_n}(\cdot,\cdot)$ or $\mf{D}_{T_G}(\cdot,\cdot)$ in order to emphasise the $G$--action under consideration.

Since twisting and rescaling preserve tameness (and its parameters), Corollary~\ref{special gives tame} shows that the actions $G\acts T_n$ are all $(\eps,N)$--tame for $\eps=\tfrac{1}{4r+2}$ and some fixed $N$. It follows that we can approximate $\beta$ by a sequence of arcs $\beta_n\cu T_n$ as in Proposition~\ref{tame sequences prop}(1). 

In addition, we can assume that the arcs $\wt\beta_n\cu T_G$ satisfy the dichotomy in Corollary~\ref{arc-stabilisers are centralisers}. Indeed, since $L/\tau_n \ra 0$, Corollary~\ref{arc-stabilisers are centralisers} shows that such arcs form an eventually dense family.

Note that, for every $0\leq\delta<\ell(\beta)$, we have:
\begin{align*}
\mf{D}_{T_n}(\beta_n,\delta)&=\varphi_n^{-1}\big(\mf{D}_{T_G}(\wt\beta_n,\delta\tau_n)\big), & G_{\beta_n}&=\varphi_n^{-1}\big(G_{\wt\beta_n}\big).
\end{align*}
We distinguish two cases, corresponding to Cases~(1a) and~(1b) of Proposition~\ref{tame sequences prop}.

\smallskip
{\bf Case~A:} \emph{there exists $0<\delta<\eps\cdot\ell(\beta)$ such that $\mf{D}_{T_n}(\beta_n,\delta)=G_{\beta_n}$ for $\om$--all $n$.}

\smallskip
In this case, $G_{\beta}=\bigcap_{\om}G_{\beta_n}$. Thus, it suffices to show that $G_{\beta_n}$ is a centraliser for $\om$--all $n$. Then Proposition~\ref{om-centralisers in G cor}(2) guarantees that $G_{\beta}$ is itself a centraliser. In particular $G_{\beta}$ is finitely generated, so $G_{\beta}\leq G_{\beta_n}$ for $\om$--all $n$ and $G_{\beta}$ is elliptic in $\om$--all $T_n$. This is Case~(1) of our proposition.

Since automorphisms of $G$ take centralisers to centralisers, it suffices to show that the subgroups $G_{\wt\beta_n}$ are centralisers. These must fall into one of the two cases of Corollary~\ref{arc-stabilisers are centralisers}. In the first case, it is clear that $G_{\wt\beta_n}$ is a centraliser. The other case can be ruled out as follows.

There would exist elements $g_n\in G$ that are loxodromic in $T_G$ with $0<\ell_{T_G}(g_n)\leq\ell_Y(g_n)\leq q$, and whose axes $\eta_n\cu T_G$ satisfy $\ell(\eta_n\cap\wt\beta_n)\geq\ell(\wt\beta_n)-4q$. In particular, $g_n\in\mf{D}_{T_G}(\wt\beta_n,9q)$ and, for large $n$, we have:
\[\varphi_n^{-1}(g_n)\in\varphi_n^{-1}\big(\mf{D}_{T_G}(\wt\beta_n,9q)\big)=\mf{D}_{T_n}(\beta_n,9q/\tau_n)\cu\mf{D}_{T_n}(\beta_n,\delta)=G_{\beta_n}.\]
This contradicts the fact that the elements $\varphi_n^{-1}(g_n)$ are loxodromic in $T_n$.

\smallskip
{\bf Case~B:} \emph{for each $0<\delta<\eps\cdot\ell(\beta)$, the subgroup $\mf{D}_{T_n}(\beta_n,\delta)$ is non-elliptic for $\om$--all $n$.}

\smallskip
In this case, $G_{\beta}$ leaves invariant a line $\alpha\cu T_{\om}$ containing $\beta$. The stabiliser of $\alpha$ is $\bigcap_{\om}\mf{D}_{T_n}(\beta_n,\delta)$ for some choice of $\delta$, and $G_{\beta}$ is the kernel of the homomorphism giving translation lengths along $\alpha$.

Since $\mf{D}_{T_n}(\beta_n,\delta)$ is non-elliptic, Corollary~\ref{D for G} shows that $\mf{D}_{T_G}(\wt\beta_n,\delta\tau_n)=Z_G(g_n)$ for a label-irreducible element $g_n\in G$ that is loxodromic in $T_G$. Again by Proposition~\ref{om-centralisers in G cor}, the subgroup 
\[Z:=\bigcap_{\om}\mf{D}_{T_n}(\beta_n,\delta)=\bigcap_{\om}Z_G(\varphi_n^{-1}(g_n))\] 
is a centraliser. Summing up, $Z$ is the entire $G$--stabiliser of the line $\alpha$ and $G_{\beta}=\ker\rho$, where $\rho\colon Z\ra\R$ gives translation lengths along $\alpha$.

Note that $Z_G(g_n)$ leaves invariant the axis $\wt\alpha_n\cu T_G$ of $g_n$, and translation lengths along it are given by a homomorphism $\eta_n\colon Z_G(g_n)\ra\R$. The lines $\wt\alpha_n\cu T_G$ correspond to lines $\alpha_n\cu T_n$ converging to $\alpha$, and the homomorphism $\rho\colon Z\ra\R$ is the $\om$--limit of the restrictions to $Z$ of the compositions $\rho_n:=\eta_n\o\varphi_n$ rescaled by $\tau_n$. 

If $\rho_n$ vanishes on $Z$ for $\om$--all $n$, then $\rho$ vanishes on $Z$, so $Z=G_{\beta}$ and $G_{\beta}$ is elliptic in $\om$--all $T_n$. In this case, we fall again in Case~(1) of our proposition.

Otherwise $\rho_n$ is nonzero on $Z$ for $\om$--all $n$, hence the $Z$--minimal subtree of $T_n$ is the line $\alpha_n$. If $g\in G$ normalises $Z$, then we have $g\alpha_n=\alpha_n$ for $\om$--all $n$. Since $\alpha_n$ converge to $\alpha$, we then have $g\alpha=\alpha$ and, since $Z$ is the entire $G$--stabiliser of $\alpha$, we have $g\in Z$. In conclusion $N_G(Z)=Z$. 

This yields Case~(2) of our proposition and concludes the proof.
\end{proof}

An \emph{infinite tripod} in an $\R$--tree is the union of three rays pairwise intersecting at a single point.

\begin{prop}\label{line and tripod prop}
All line- and (infinite tripod)-stabilisers for $G\acts T_{\om}$ are centralisers.
\end{prop}
\begin{proof}
The stabiliser of an infinite tripod is the intersection of two line-stabilisers, so it suffices to show that line-stabilisers are centralisers. The latter follows from part~(2) of Proposition~\ref{tame sequences prop}, retracing the proof of Proposition~\ref{arc-stab prop}.
\end{proof}

\begin{rmk}
Proposition~\ref{arc-stab prop} shows in particular that arc-stabilisers are closed under taking roots in $G$. Thus, the stabiliser of a line in $T_{\om}$ can never swap its two ends.
\end{rmk}

The following expands on Case~(2) of Proposition~\ref{arc-stab prop}. 

\begin{prop}\label{important addenda new}
Let $\alpha\cu T_{\om}$ be a line acted upon by its stabiliser $G_{\alpha}$ via a (possibly trivial) homomorphism $\rho\colon G_{\alpha}\ra\R$. Suppose that $G_{\alpha}$ is non-elliptic in $T_n$ for $\om$--all $n$.
\begin{enumerate}
\item[(a)] There exists $x\in G$ such that $G_{\alpha}=Z_G(x)$ and we have $N_G(G_{\alpha})=G_{\alpha}$.
\smallskip
\item[(b)] Assuming that $\rho$ is discrete and $\ker\rho$ is finitely generated, the following hold.
	\begin{itemize}
	\item If $\ker\rho$ is elliptic in $\om$--all $T_n$, then $\ker\rho$ is $G$--semi-parabolic and $G_{\alpha}\lneq N_G(\ker\rho)$.
	\item If $\ker\rho$ is non-elliptic in $\om$--all $T_n$, then the centre of $\ker\rho$ contains an element $h$ that is loxodromic in $\om$--all $T_n$. 
	% with translation length going to zero
	If $\ker\rho=G_{\alpha}$, then $h$ can be chosen to be label-irreducible.
	\end{itemize}
\smallskip
\item[(c)] Suppose that:
	\begin{itemize}
	\item[(c1)] either the automorphisms $\varphi_n$ are coarse-median preserving;
	\item[(c2)] or there does not exist an element $x\in G$ such that $\varphi_n(Z_G(x))$ lies in a single $G$--conjugacy class of subgroups for $\om$--all $n$.
	\end{itemize}
Then $\rho$ is discrete, $\ker\rho$ is a centraliser, and $G_{\alpha}$ is a proper subgroup of $G$. In addition, if $\ker\rho\neq G_{\alpha}$, then $\ker\rho$ is elliptic in $\om$--all $T_n$. 
\smallskip
\item[(d)] Suppose that either $\rho$ is not discrete, or $\ker\rho$ is not $G$--semi-parabolic. Then either the centre of $Z$ has rank $\geq 2$, or it is infinite cyclic and contained in $\ker\rho$.
\smallskip
\item[(e)] If $\rho$ is not discrete, then, for every arc $\beta\cu\alpha$, we have $G_{\beta}\leq G_{\alpha}$.
\end{enumerate}
\end{prop}
\begin{proof}
We begin with some preliminary remarks.

By Proposition~\ref{line and tripod prop}, we know that $G_{\alpha}$ is a centraliser, so let us write $Z:=G_{\alpha}$ for short. Since $Z$ is finitely generated, and we are assuming that it is non-elliptic in $\om$--all $T_n$, there are loxodromics for its action on $T_n$ for $\om$--all $n$. Thus, retracing one last time the proof of Proposition~\ref{arc-stab prop} using part~(2) of Proposition~\ref{tame sequences prop}, we are necessarily in Case~B, and we obtain a sequence of label-irreducible elements $g_n\in G$ (without loss of generality, not proper powers of elements of $G$) such that $Z=\bigcap_{\om}Z_G(\varphi_n^{-1}(g_n))$. In addition, the $Z$--minimal subtree of $T_n$ is a line for $\om$--all $n$, and we have $N_G(Z)=Z$.

The homomorphism $\rho\colon Z\ra\R$ is obtained as the limit of the homomorphisms $\rho_n\colon Z\ra\R$ giving translation lengths along the $Z$--invariant line in $T_n$. Each $\rho_n$ is the restriction to $Z$ of the composition $\eta_n\o\varphi_n$, where $\eta_n\colon Z_G(g_n)\ra\R$ is the homomorphism giving translation lengths in $\T_v$, rescaled by $\tau_n$. Note that $\eta_n$ has the same kernel as the straight projection $\pi_{g_n}\colon Z_G(g_n)\ra\Z$ introduced in Remark~\ref{cmp and straight proj rmk}(2). Recall that $\ker\pi_{g_n}$ is $G$--parabolic.

\smallskip
{\bf Proof of part~(a).} We have already observed that $N_G(Z)=Z$ and $Z=\bigcap_{\om}Z_G(\varphi_n^{-1}(g_n))$. Let us show that we actually have $Z=Z_G(\varphi_n^{-1}(g_n))$ for $\om$--all $n$.

Since $Z$ is finitely generated, we have $Z\leq Z_G(\varphi_n^{-1}(g_n))$ for $\om$--all $n$, so $Z$ commutes with $\varphi_n^{-1}(g_n)$. Since $N_G(Z)=Z$, the elements $\varphi_n^{-1}(g_n)$ lie in the centre of $Z$ for $\om$--all $n$. Thus, the $\varphi_n^{-1}(g_n)$ pairwise commute, hence the maximal cyclic groups containing their label-irreducible components are $\om$--constant. This shows that $Z_G(\varphi_n^{-1}(g_n))$ is $\om$--constant, hence it coincides with $Z$, as required.

\smallskip
{\bf Proof of part~(b).} First, suppose that $\ker\rho$ is elliptic in $\om$--all $T_n$. We will show that $Z$ is a proper subgroup of $N_G(\ker\rho)$. Since $Z$ is $G$--semi-parabolic and $N_G(Z)=Z$, we can then use Proposition~\ref{kernel dichotomy} to deduce that $\ker\rho$ is $G$--semi-parabolic, as required.

Since $\ker\rho$ is finitely generated and elliptic in $T_n$, we have $\ker\rho\leq\ker\rho_n$ for $\om$--all $n$. Since $\rho$ is discrete, we have $Z/\ker\rho\simeq\Z$ (recall that $Z$ is non-elliptic in $T_n$, so $\ker\rho$ is a proper subgroup). 
%Recall that $Z$ is non-elliptic in $T_n$, but $\ker\rho$ is, so they cannot coincide.
Since we also have $Z/\ker\rho_n\simeq\Z$, this implies that $\ker\rho_n=\ker\rho$ for $\om$--all $n$ ($\Z$ is Hopfian). Recalling from part~(a) that $Z=\varphi_n^{-1}(Z_G(g_n))$, it follows that:
\[Z\leq N_G(\ker\rho)=N_G(\ker\rho_n)=\varphi_n^{-1}N_G(\ker\eta_n)=\varphi_n^{-1}N_G(\ker\pi_{g_n}).\] 

Suppose for the sake of contradiction that we have $Z=N_G(\ker\rho)$. Then $N_G(\ker\pi_{g_n})=\varphi_n(Z)=Z_G(g_n)$. Since $\ker\pi_{g_n}$ is $G$--parabolic, Corollary~\ref{fin many conj parabolics} shows that the centralisers $Z_G(g_n)$ are chosen from finitely many $G$--conjugacy classes of subgroups. Since the $g_n$ are all label-irreducible, it follows that there is a conjugacy class $\mc{C}\cu G$ such that $g_n\in\mc{C}$ for $\om$--all $n$. This implies that the translation length of $g_n$ in $\T_v$ is uniformly bounded, so $\inf(\rho_n(Z)\cap\R_{>0})$ converges to zero. Since $\rho_n$ converges to $\rho$ and $\ker\rho_n=\ker\rho$ for $\om$--all $n$, we conclude that $\rho$ is trivial and $Z=\ker\rho=\ker\rho_n$. This contradicts our assumption that $Z$ be non-elliptic in $\om$--all $T_n$.

In order to complete the proof of part~(b), suppose now that $\ker\rho$ is non-elliptic for $\om$--all $T_n$. Let $h_1,\dots,h_k$ be a finite generating set for $\ker\rho$. The elements $h_i$ and $h_ih_j$ cannot all be elliptic in $T_n$ or $\ker\rho$ would also be elliptic in $T_n$ by Serre's lemma.
%Helly's lemma. 
This shows that there exists an element $h'\in\ker\rho$ such that $\rho_n(h')>0$ for $\om$--all $n$. Hence $\inf(\rho_n(Z)\cap\R_{>0})\ra 0$. 

Now, let $Z_0$ denote the centre of $Z$. Recall that, if $G$ is, say, $q$--convex-cocompact, then $Z_G(g_n)$ has a subgroup of index $\leq q$ that splits as $\langle g_n\rangle\x\ker\pi_{g_n}$ (e.g.\ by Remark~\ref{LI properties}(6)). Thus, since we have seen in the proof of part~(a) that $\varphi_n^{-1}(g_n)\in Z_0$ for $\om$--all $n$, we have: 
\[\inf(\rho_n(Z_0)\cap\R_{>0})\leq\eta_n(g_n)\leq q\cdot \inf(\eta_n(Z_G(g_n))\cap\R_{>0})=q\cdot \inf(\rho_n(Z)\cap\R_{>0})\ra 0.\]

The subgroup $Z_0$ is free abelian, nontrivial and convex-cocompact (note that $Z_0$ is itself a centraliser). Thus, $Z_0$ admits a basis of label-irreducible elements $x_1,\dots,x_m$ with $m\geq 1$. Since $Z_0$ contains $\varphi_n^{-1}(g_n)$, on which $\rho_n$ does not vanish, we can assume that $\rho_n(x_1)>0$ for $\om$--all $n$. If $x_1\in\ker\rho$, we can take $h=x_1$ and we are done. Note that this is always the case if $\ker\rho=Z=G_{\alpha}$.

Otherwise $\rho(x_1)\neq 0$, hence $\rho(Z_0)\neq\{0\}$. Modifying the basis of $Z_0$ if necessary (which can kill label-irreducibility of its elements), we can assume that $\rho(x_1)$ generates $\rho(Z_0)$, and $x_i\in\ker\rho$ for all $i\geq 2$. If $\rho_n$ vanished on all $x_i\neq x_1$ for $\om$--all $n$, we would have $\inf(\rho_n(Z_0)\cap\R_{>0})=\rho_n(x_1)\ra \rho(x_1)\neq 0$, contradicting the fact that $\inf(\rho_n(Z_0)\cap\R_{>0})\ra 0$. Thus, there exists $i\geq 2$ such that $\rho_n(x_i)>0$ for $\om$--all $n$, and $x_i$ lies in the centre of $\ker\rho$ as required. This proves part~(b).

\smallskip
{\bf Proof of part~(c).} If $\ker\rho=Z$, then all statements are immediate. Indeed, $\rho$ is trivial (hence discrete), its kernel is the centraliser $Z$, and we cannot have $G=Z$, as $G$ would act elliptically on $T_{\om}$. Thus, we assume in the rest of the proof that $\ker\rho\neq Z$.

Recall that there exist label-irreducible elements $g_n\in G$ such that $Z=\varphi_n^{-1}(Z_G(g_n))$ for $\om$--all $n$. Also recall the notion of straight projection from Remark~\ref{cmp and straight proj rmk}(2). 

\smallskip
{\bf Claim:} \emph{the subgroup $\varphi_n(Z)=Z_G(g_n)$ does not lie in a single $G$--conjugacy class for $\om$--all $n$.}

\smallskip\noindent
\emph{Proof of claim.}
Under Assumption~(c2), this is immediate from part~(a). Suppose instead that the automorphisms $\varphi_n$ are coarse-median preserving. Then Remark~\ref{cmp and straight proj rmk} shows that we have $g=\varphi_n^{-1}(g_n)$ for a fixed label-irreducible $g\in G$ and $\om$--all $n$. In addition, $\ker\rho_n=\varphi_n^{-1}(\ker\pi_{g_n})=\ker\pi_g$, so $\ker\pi_g$ is contained in $\ker\rho$. Since $Z/\ker\pi_g\simeq\Z$ and $\ker\rho\neq Z$, this implies that $\ker\rho=\ker\pi_g$.

If the $\varphi_n(Z)=Z_G(g_n)$ lied in a single conjugacy class of subgroups, then the elements $\varphi_n(g)=g_n$ would lie in a single conjugacy class, since they are all label-irreducible. As in the proof of part~(b), this would imply that $g$ is elliptic in $T_{\om}$. In this case $g\in\ker\rho\setminus\ker\pi_g$, contradicting the fact that $\ker\rho=\ker\pi_g$.
\hfill$\blacksquare$

\smallskip
The claim immediately implies that $G_{\alpha}=Z$ is a proper subgroup of $G$. We now prove the remaining statements.

Since $g_n\in Z_G(\ker\pi_{g_n})$, we have $\ker\pi_{g_n}\leq Z_GZ_G(\ker\pi_{g_n})\leq Z_G(g_n)$. Being a centraliser, $Z_GZ_G(\ker\pi_{g_n})$ is convex-cocompact and closed under taking roots. Thus, since $Z_G(g_n)$ virtually splits as $\langle g_n\rangle\x\ker\pi_{g_n}$, we must have either $Z_GZ_G(\ker\pi_{g_n})=\ker\pi_{g_n}$ or $Z_GZ_G(\ker\pi_{g_n})=Z_G(g_n)$.

Suppose first that $Z_GZ_G(\ker\pi_{g_n})=Z_G(g_n)$ for $\om$--all $n$. Recall from Remark~\ref{cmp and straight proj rmk}(2) that $\ker\pi_{g_n}$ is $G$--parabolic. Thus, Corollary~\ref{fin many conj parabolics} implies that the subgroups $Z_G(g_n)$ lie in finitely many $G$--conjugacy classes. It follows that $\varphi_n(Z)=Z_G(g_n)$ lies in a single $G$--conjugacy class for $\om$--all $n$. This is ruled out by the claim.

Thus, we must have $Z_GZ_G(\ker\pi_{g_n})=\ker\pi_{g_n}$ for $\om$--all $n$. In this case, $\ker\pi_{g_n}$ is a centraliser, so $\ker\rho_n=\varphi_n^{-1}(\ker\pi_{g_n})$ is a centraliser, and it is a convex-cocompact. Since $Z$ virtually splits as $\ker\rho_n\x\Z$, we conclude that the subgroup $\ker\rho_n$ is $\om$--constant. It follows that $\ker\rho_n\leq\ker\rho$ for $\om$--all $n$. Since $Z/\ker\rho_n\simeq\Z$ and we assumed that $\ker\rho\neq Z$, it follows that $\ker\rho_n=\ker\rho$ for $\om$--all $n$. This shows that $\rho$ is discrete and $\ker\rho$ is a centraliser elliptic in $\om$--all $T_n$.

\smallskip
{\bf Proof of part~(d).} Recall from the proof of part~(a) that the elements $\varphi_n^{-1}(g_n)$ all lie in the centre of $Z$. Thus, the latter always has rank at least $1$. Suppose that the centre of $Z$ has rank $1$ and intersects $\ker\rho$ trivially. We will show that $\rho$ is discrete with $G$--semi-parabolic kernel. 

Since the centre of $Z$ is cyclic, there exists an element $g\in G$ (now not necessarily label-irreducible) such that $g=\varphi_n^{-1}(g_n)$ for $\om$--all $n$. Thus $\rho_n(g)=\eta_n(g_n)\in\R$ generates the image of $\rho_n$. If we had $\rho_n(g)\ra 0$, then the centre of $Z$ would be elliptic in $T_{\om}$, hence contained in $\ker\rho$. Thus, $\rho_n(g)$ must stay bounded away from zero, and the image of $\rho$ is discrete and generated by $\rho(g)$.

Now, $Z$ virtually splits as $\langle g\rangle\x P$ for some subgroup $P\leq Z$, which is necessarily finitely generated. Since $\rho(g)$ is nontrivial and generates the image of $\rho$, we deduce that $\ker\rho$ intersects $\langle g\rangle\x P$ in a subgroup that projects isomorphically onto $P$. In particular, $\ker\rho$ is finitely generated. Since the images of the $\rho_n$ stay bounded away from zero and $\ker\rho$ is finitely generated, we conclude that $\ker\rho$ is elliptic in $\om$--all $T_n$. Finally, by part~(b), $\ker\rho$ is $G$--semi-parabolic.

\smallskip
{\bf Proof of part~(e).} Since $\rho$ is non-discrete, there exists a loxodromic element $g\in Z$ with translation length small enough that there exists a point $x\in\beta$ such that $g^Nx\in\beta$, where $N=4\dim\X_{\G}+1$. If $h\in G_{\beta}$, Lemma~\ref{long min intersection} shows that $h$ preserves the axis of $g$ in $T_n$ for $\om$--all $n$. Thus, $h$ preserves the axis of $g$ in $T_{\om}$, hence $h\in G_{\alpha}$ as required.
% Alternative proof, avoiding Lemma~\ref{long min intersection}, but requiring $\beta$ to be stable with convex-cocompact stabiliser:
% If $g\in Z$ is loxodromic with sufficiently small translation length, then $g$ takes an arc of $\beta$ to an arc of $\beta$. Since $\beta$ is stable, this means that $g$ normalises $G_{\beta}$. Since $G_{\beta}$ is convex-cocompact, a power of $g$ commutes with $G_{\beta}$. Thus, $G_{\beta}$ fixes pointwise the axis of $g$, i.e. the line $\alpha$.
\end{proof}

Parts~(c1) and~(c2) of Proposition~\ref{important addenda new} are the key to Theorem~\ref{main cmp} and Corollary~\ref{infinite out finite cmp}, respectively. The differences in the assumptions of Proposition~\ref{important addenda new}(c2) and Corollary~\ref{infinite out finite cmp} can be reconciled using the following consequence of a providential result of B.\ H.\ Neumann.

%Given a subgroup $H\leq G$, the \emph{conjugacy class} of $H$ is the set of subgroups of $G$ of the form $gHg^{-1}$ with $g\in G$.
\begin{lem}\label{Neumann's lemma}
Let $G$ be a group with a countable collection $\mscr{H}$ of $G$--conjugacy classes of subgroups. Suppose that, for every class $\mc{H}\in\mscr{H}$, the $\out(G)$--orbit of $\mc{H}$ is infinite. Then there exists a sequence $\phi_n\in\out(G)$ such that, for every $\mc{H}\in\mscr{H}$, the sequence $\phi_n(\mc{H})$ eventually consists of pairwise distinct classes. 
\end{lem}
\begin{proof}
Let $\dots\cu\mscr{H}_n\cu\mscr{H}_{n+1}\cu\dots$ be an exhaustion of $\mscr{H}$ by finite subsets. We define $\phi_n$ inductively, starting with an arbitrary automorphism $\phi_1$. Suppose that $\phi_1,\dots,\phi_n$ have been defined. We would like to choose $\phi_{n+1}$ so that, for every $\mc{H}\in\mscr{H}_n$, we have $\phi_{n+1}(\mc{H})\not\in\{\phi_1(\mc{H}),\dots,\phi_n(\mc{H})\}$. If this is possible for every $n$, then we obtain the required sequence of automorphisms.

Suppose instead that for some $n$, we cannot choose $\phi_{n+1}$ with this property. Then, for every $\phi\in\out(G)$, there exists $\mc{H}\in\mscr{H}_n$ such that $\phi(\mc{H})\in\{\phi_1(\mc{H}),\dots,\phi_n(\mc{H})\}$. Denoting by $\out_{\mc{H}}(G)$ the stabiliser of $\mc{H}$ within $\out(G)$, this means that $\out(G)$ is covered by finitely many cosets of the infinite-index subgroups $\out_{\mc{H}}(G)$ with $\mc{H}\in\mscr{H}_n$. By \cite[Lemma~4.1]{Neumann-cosets}, this is impossible.
\end{proof}

Finally, we record here the following result, which will be repeatedly needed in Subsection~\ref{main argument sect}.

\begin{lem}\label{almost dense orbits}
Let $\beta\cu T_{\om}$ be an arc such that $G_{\beta}$ is finitely generated and elliptic in $\om$--all $T_n$. If $\beta$ falls in Case~(2) of Proposition~\ref{arc-stab prop}, assume in addition that $\rho$ has discrete image. Then there exists a sequence $\eps_n\ra 0$ such that $Z_G(G_{\beta})$ acts on $\Fix(G_{\beta},T_n)$ with $\eps_n$--dense orbits.
\end{lem}
\begin{proof}
Observe first that $\varphi_n(G_{\beta})$ is $G$--semi-parabolic for every $n$. Indeed, this is clear if $\beta$ falls in Case~(1) of Proposition~\ref{arc-stab prop}, since then $G_{\beta}$ is a centraliser. Otherwise, Proposition~\ref{important addenda new}(b) shows that $G_{\beta}$ is the kernel of a homomorphism $Z\ra\Z$, where $Z$ is a centraliser and $Z$ is a proper subgroup of $N_G(G_{\beta})$. In this case, the group $\varphi_n(G_{\beta})$ is of the same form for every $n$, and Proposition~\ref{kernel dichotomy} ensures that $\varphi_n(G_{\beta})$ is $G$--semi-parabolic.

Now, a consequence is that $\varphi_n(G_{\beta})$ is convex-cocompact in $G$ for all $n$. Recall that $T_G\cu\T_v$ is the $G$--minimal subtree. Thus, Lemma~\ref{special cc}(2) shows that $\varphi_n(Z_G(G_{\beta}))=Z_G(\varphi_n(G_{\beta}))$ acts on $\Fix(\varphi_n(G_{\beta}),T_G)$ with $\leq c$ orbits of edges, where $c$ only depends on $G$ and its embedding in $\A_{\G}$.

Since $T_n$ is a copy of $T_G$, twisted by $\varphi_n$ and rescaled by $\tau_n\ra+\infty$, it follows that $Z_G(G_{\beta})$ acts on $\Fix(G_{\beta},T_n)$ with $\eps_n$--dense orbits, where $\eps_n:=c/\tau_n\ra 0$.
\end{proof}

\section{From $\R$--trees to DLS automorphisms.}\label{Rips sect}

In this section, we use the description of the limit tree $T_{\om}$ obtained in Subsection~\ref{limit subsec} to prove Theorems~\ref{main cmp} and~\ref{main general} and Corollary~\ref{infinite out finite cmp}.

First, in Subsection~\ref{Guirardel sect}, we review various standard results originating from ideas of Rips, Sela, Bestvina, Feighn and Guirardel. Then, in Subsection~\ref{outer dls sect}, we briefly discuss how to ensure that DLS automorphisms are not inner. Finally, Subsection~\ref{main argument sect} contains the core argument.

\subsection{Actions on $\R$--trees.}\label{Guirardel sect}

In this subsection, we review a few classical facts on actions on $\R$--trees. 

A subtree of an $\R$--tree is \emph{non-degenerate} if it is not a single point. Arcs are always assumed to be non-degenerate. A \emph{finite subtree} is the convex hull of a finite set of points.

If $G$ is a group, we refer to $\R$--trees equipped with an isometric $G$--action simply as \emph{$G$--trees}. A $G$--tree $T$ is \emph{minimal} if it does not contain any proper $G$--invariant subtrees. 
% this is probably fine as-is, but we speak of minimal trees earlier in the paper
A non-degenerate subtree $U\cu T$ is \emph{stable} if all its arcs have the same $G$--stabiliser. We say that $T$ is \emph{BF--stable} (after \cite{BF-stable}) if every arc of $T$ contains a stable sub-arc.

If $T_1$ and $T_2$ are $G$--trees, a \emph{morphism} is a $G$--equivariant map $f\colon T_1\ra T_2$ with the property that every arc of $T_1$ can be covered by finitely many arcs on which $f$ is isometric.

A $G$--tree is said to be \emph{geometric} if it originates from a finite foliated $2$--complex $X$ with fundamental group $G$. The precise definition will not be relevant to us and is omitted. We instead refer the reader to \cite{LP} or \cite[Subsection~1.7]{Gui-Fou} for additional details. 

\begin{rmk}\label{geometric modulo kernel}
Let $T$ be a geometric $G$--tree. If $N$ is the kernel of the $G$--action, then $T$ is geometric also as a $G/N$--tree.
% for instance, this is clear from the construction of geometric actions in \cite[Section~2]{Gui-CMH} (just take a different cover of the 2--complex with free fundamental group) (every geometric action is of this form because they are not nontrivial strong limits)
\end{rmk}

The following can be deduced for instance from \cite[Theorem~2.2]{LP} or \cite[Section~2]{Gui-CMH}.

\begin{prop}\label{geometric approximations}
Let $T$ be a minimal $G$--tree with $G$ finitely presented. To every finite subtree $K\cu T$, we can associate a geometric $G$--tree $\mc{G}_K$ and a morphism $f_K\colon\mc{G}_K\ra T$ so that: 
\begin{itemize}
\item there exists a finite subtree $\wt K\cu\mc{G}_K$ such that $f_K$ is isometric on $\wt K$ and $f_K(\wt K)=K$;
\item if $K\cu K'$, then there exists a morphism $f_K^{K'}\colon \mc{G}_K\ra\mc{G}_{K'}$ such that $f_K=f_{K'}\o f_K^{K'}$ and $f_K^{K'}(\wt K)\cu\wt K'$;
\item if $H\leq G$ is finitely generated and fixes $K$, there exists $K'\supseteq K$ such that $H$ fixes $f_K^{K'}(\wt K)$.
\end{itemize}
\end{prop}
% the construction only works when $K$ is large enough, but we can take $\mc{G}_K$ to be the one for larger $K$ when $K$ is too small

We will refer to the morphisms $f\colon\mc{G}_K\ra T$ provided by the previous proposition as \emph{geometric approximations} of $T$.

\begin{defn}[\cite{Gui-Fou}, Definition~1.4]\label{transverse covering defn}
A \emph{transverse covering} of a $G$--tree $T$ is a $G$--invariant family $\mc{U}=\{U_i\}_{i\in I}$ of closed subtrees of $T$ that cover $T$ and satisfy the following:
\begin{itemize}
\item if $i\neq j$, the intersection $U_i\cap U_j$ is empty or a singleton;
\item every arc of $T$ can be covered by finitely many elements of $\mc{U}$.
\end{itemize}
\end{defn}

A transverse covering $\mc{U}$ of a $G$--tree $T$ always gives rise to a splitting of $G$. Indeed, we can construct an action without inversions on a simplicial tree $G\acts S_{\mc{U}}$ as follows \cite[Lemma~4.7]{Gui-GT}.

The vertex set of $S_{\mc{U}}$ is a disjoint union $V_0(S_{\mc{U}})\sqcup V_1(S_{\mc{U}})$, where $V_1(S_{\mc{U}})$ is identified with $\mc{U}$ and $V_0(\mc{S}_{\mc{U}})$ is the set of points appearing as the intersection of two elements of $\mc{U}$. The tree $S_{\mc{U}}$ is bipartite, with edges joining each point of $V_0(\mc{S}_{\mc{U}})$ to all elements of $\mc{U}=V_1(S_{\mc{U}})$ that contain it.

Note that, if $T$ is $G$--minimal, then so is $S_{\mc{U}}$. 
% can take ``preimages''

\begin{defn}[\cite{Gui-Fou}, Definition~1.17]
Consider a $G$--tree $T$ and a subgroup $H\leq G$. A non-degenerate subtree $U\cu T$ is \emph{$H$--indecomposable} if, for any two arcs $\beta,\beta'\cu U$, there exist elements $h_1,\dots,h_n\in H$ such that $\beta'\cu h_1\beta\cup\dots\cup h_n\beta$ and $h_i\beta\cap h_{i+1}\beta$ is an arc for each $1\leq i<n$.
 
 Note that $U$ is not required to be $H$--invariant and the arcs $h_i\beta\cap h_{i+1}\beta$ can be disjoint from $U$.
\end{defn}

The terminology is motivated by the fact that, if $\mc{U}$ is a transverse covering of $T$, then every $G$--indecomposable subtree of $T$ must be contained in one of the elements of $\mc{U}$ \cite[Lemma~1.18]{Gui-Fou}.

We also record here part of Lemmas~1.19 and~1.20 in \cite{Gui-Fou}.

\begin{lem}\label{indec lem}
Let $T$ be a $G$--tree with a $G$--indecomposable subtree $U\cu T$.
\begin{enumerate}
\item If $f\colon T\ra T'$ is morphism, then $f(U)$ is $G$--indecomposable.
\item If $G\acts T$ is BF--stable, then $U$ is a stable subtree.
\item If $T$ is itself $G$--indecomposable, then it is $G$--minimal.
\end{enumerate}
\end{lem}

The following is a version of Imanishi's theorem \cite{Imanishi} due to Guirardel.

\begin{prop}\label{Imanishi}
Let $T$ be a geometric $G$--tree with $G$ finitely presented. Then $T$ admits a unique transverse covering $\mc{U}=\{U_i\}_{i\in I}$ where, for each $i$:
\begin{itemize}
\item either $U_i$ is a non-degenerate arc containing no branch points of $T$ in its interior;
\item or $G_i\acts U_i$ is indecomposable and geometric, where $G_i\leq G$ the stabiliser of $U_i$.
\end{itemize}
\end{prop}
\begin{proof}
Existence follows from Proposition~1.25 and Remark~1.29 in \cite{Gui-Fou}. The additional hypothesis required for Guirardel's result is always satisfied for geometric actions of \emph{finitely presented} groups (for instance, combining \cite[Remark~2.3]{LP} with \cite[Theorem~0.2(2)]{LP}). Uniqueness is due to \cite[Lemma~1.18]{Gui-Fou}. 
\end{proof}

\begin{defn}
We refer to the elements of the transverse covering $\mc{U}$ provided by Proposition~\ref{Imanishi} as the \emph{components} of $T$ (this is justified by uniqueness of $\mc{U}$).
\end{defn}

The following classification result is due to Rips and Bestvina--Feighn \cite{BF-stable}. This formulation is taken from \cite[Proposition~A.6]{Gui-Fou}.

\begin{prop}\label{3 types}
Let $T$ be a geometric $G$--tree with $G$ finitely presented and torsion-free. Suppose that $T$ has trivial arc-stabilisers and is $G$--indecomposable. Then $T$ is of one of the following types.
\begin{itemize}
\item {\em Axial:} $T$ is a line and $G$ is a free abelian group acting on $T$ with dense orbits.
\item {\em Surface:} $G$ is the fundamental group of a compact surface with boundary supporting an arational measured foliation that gives rise to $T$.
\item {\em Exotic:} none of the above.
\end{itemize}
\end{prop}

We acknowledge that, with no requirement on exotic components, the previous proposition seems trivially satisfied. What we will actually need about these three types of $G$--trees is that they can be approximated by simplicial $G$--trees in a controlled way, as shown in \cite{Gui-CMH}. We will refer the reader to precise statements when these will become necessary later in this section.

The following observation will allow us to apply Proposition~\ref{3 types} even though our actions normally have large arc-stabilisers.

\begin{lem}\label{stable components}
Let $T$ be a BF--stable $G$--tree, with $G$ finitely presented. Suppose that a geometric approximation $f\colon\mc{G}\ra T$ admits an indecomposable component $U\cu\mc{G}$. Let $\beta\cu f(U)$ be an arc with finitely generated stabiliser $G_{\beta}$. Then there exists a geometric approximation $\mc{G}'\ra T$ with an indecomposable component $U'\cu\mc{G}'$ such that:
\begin{enumerate}
\item $U'$ is invariant under the $G$--stabiliser of $U$;
\item the $G$--stabiliser of every arc of $U'$ coincides with $G_{\beta}$.
\end{enumerate}
\end{lem}
\begin{proof}
By Proposition~\ref{geometric approximations}, there exists a geometric approximation $f'\colon\mc{G}'\ra T$ such that $G_{\beta}$ is elliptic in $\mc{G}'$ and such that $f$ factors as the composition of $f'$ and a morphism $p\colon\mc{G}\ra\mc{G}'$. Let $G_U$ be the $G$--stabiliser of $U\cu\mc{G}$. By Lemma~\ref{indec lem}(1), the image $p(U)\cu\mc{G}'$ is $G_U$--indecomposable, hence contained in an indecomposable component $U'\cu\mc{G}'$.
% as mentioned above, by \cite[Lemma~1.18]{Gui-Fou}.
Since distinct indecomposable components share at most one point, $U'$ must be $G_U$--invariant.

By Lemma~\ref{indec lem}, the image $f'(U')$ is $G$--indecomposable, hence a stable subtree of $T$. Since $f'(U')$ contains $f(U)$, which in turn contains $\beta$, we see that the stabiliser of every arc of $f'(U')$ is equal to $G_{\beta}$. In particular, since $f'(U')$ is $G_{U'}$--invariant and not a single point,
% since f is a morphism
the subgroup $G_{\beta}$ is normalised by $G_{U'}$.

Now, the subtree $\Fix(G_{\beta},\mc{G}')$ is nonempty and $G_{U'}$--invariant, thus it contains the $G_{U'}$--minimal subtree of $\mc{G}'$. By Lemma~\ref{indec lem}(3), the latter is $U'$. This shows that $G_{\beta}$ is contained in the stabiliser of every arc of $U'$. On the other hand, the $G$--stabiliser of an arc of $U'$ is contained in the $G$--stabiliser of an arc of $f'(U')$, since $f'$ is a morphism, hence it is contained in $G_{\beta}$. This shows that the $G$--stabiliser of every arc of $U'$ is exactly $G_{\beta}$, as required.
\end{proof}

Finally, we record the following standard fact on refining simplicial splittings. We say that $G$ \emph{splits over} a subgroup $C$ if it is an amalgamated product $G=A\ast_C B$ or an HNN extension $G=A\ast_C$.
% use this terminology earlier?

\begin{lem}\label{refining splittings}
Let $G\acts T$ be a minimal action without inversions on a simplicial tree. Suppose that, for a vertex $v\in T$, the stabiliser $G_v$ splits over a subgroup $C\leq G_v$, with Bass--Serre tree $G_v\acts T'$. Suppose in addition that, for every edge $e\cu T$ incident to $v$, the stabiliser $G_e$ is elliptic in $T'$. Then $G$ splits over $C$. In addition, if the splitting of $G_v$ is HNN, then so is the one of $G$.
\end{lem}

\subsection{Outer DLS automorphisms.}\label{outer dls sect}

As mentioned in the Introduction, there are various situations in which a DLS automorphism turns out to be an inner automorphism in a non-obvious way. In this subsection, we provide two simple criteria to ensure that this does not happen. 

\begin{lem}\label{non-inner from HNN}
Consider a group $G$ with an HNN splitting $G=A\ast_C$. Suppose that $Z_C(C)$ commutes with the chosen stable letter $t\in G$. If the twist $\tau\in\aut(G)$ induced by an element $c\in Z_C(C)\setminus\{1\}$ is an inner automorphism of $G$, then $c$ has finite order in $G$ and $C=A$.
% so $G=A\rtimes\langle t\rangle$
\end{lem}
\begin{proof}
Consider $c\in Z_C(C)\setminus\{1\}$ and the twist $\tau\in\aut(G)$ with $\tau(t)=ct$ and $\tau(a)=a$ for every $a\in A$. Let $G\acts T$ be the Bass--Serre tree of the HNN extension. Let $\alpha$ and $\alpha'$ be the axes of $t$ and $ct$, respectively. There exists a point $x_0\in\alpha$ such that the stabiliser of $x_0$ is $A$ and the stabiliser of the edge $[x_0,tx_0]$ is $C$. Note that $[x_0,tx_0]$ is contained in the intersection $\alpha\cap\alpha'$.

Suppose $\tau$ is inner. Thus, there exists $g\in Z_G(A)$ such that $gtg^{-1}=ct$. In particular $g\alpha=\alpha'$, preserving the orientation induced by $t$ and $ct$. We distinguish three cases.

If $x_0$ is the only point of $T$ that is fixed by $A$, then $g$ must fix $x_0$. Since $g\alpha=\alpha'$, the edge $[x_0,tx_0]$ is also fixed by $g$, hence $g\in C$. Since $g\in Z_G(A)$, we have $g\in Z_C(C)$. By our assumptions, this implies that $g$ commutes with $t$, contradicting the fact that $gtg^{-1}=ct$.

If $A$ has more than one fixed point in $T$, then either $tAt^{-1}\geq A$ or $tAt^{-1}\leq A$. 
% Indeed, if $A$ has more than one fixed point in $T$ then it fixes an edge adjacent to $x_0$ and, up to the $A$--action, we can assume that it is either $[x_0,tx_0]$ or $[t^{-1}x_0,x_0]$
Suppose first that one of these inclusions is strict. Since $g\alpha=\alpha'$, there exists $n\in\Z$ such that $gt^n$ fixes the edge $[x_0,tx_0]$, hence $g=c't^{-n}$ for some $c'\in C$. This implies that either $gAg^{-1}\gneq A$ or $gAg^{-1}\lneq A$, contradicting the fact that $g\in Z_G(A)$.

Finally, suppose that $tAt^{-1}=A$. In this case, $C=A$ and $G=A\rtimes_{\psi}\langle t\rangle$ for some $\psi\in\aut(A)$. Since $g\in Z_G(A)$, a standard computation shows that $g=xt^n$ for some $x\in A$ and $n\in\Z$ such that $\psi^n(a)=x^{-1}ax$ for all $a\in A$. It follows that $g$ commutes with $t^n$. Since $C=A$, the centre of $A$ commutes with $t$, so $g\not\in Z_A(A)$, hence $n\neq 0$. Observing that $\tau(t^n)=c^nt^n=gt^ng^{-1}$, we conclude that $c^n=1$.
\end{proof}

\begin{lem}\label{non-inner from cc}
Let a special group $G$ act on a simplicial tree $\mc{S}$ with a single orbit of edges and no inversions. Suppose that the stabiliser $C$ of an edge of $\mc{S}$ satisfies the following:
\begin{itemize}
\item $C$ is convex-cocompact and closed under taking roots in $G$;
\item $N_G(C)$ is not elliptic in $\mc{S}$;
\item $N_G(C)/C$ is not cyclic, nor a free product of two virtually abelian groups elliptic in $\Fix(C,\mc{S})$.
\end{itemize} 
Then this splitting of $G$ gives rise to a partial conjugation or a fold that has infinite order in $\out(G)$.
\end{lem}
\begin{proof}
Lemma~\ref{cc conjugates} ensures that $C$ does not properly contain any of its conjugates, so the $G$--stabiliser of every edge of $\Fix(C,\mc{S})$ is equal to $C$. Since $G$ acts edge-transitively on $\mc{S}$, it follows that $N_G(C)$ acts edge-transitively on $\Fix(C,\mc{S})$. Thus, the induced action $N_G(C)/C\acts\Fix(C,\mc{S})$ gives a $1$--edge free splitting of $N_G(C)/C$. 

Let $A$ and $B$ be $G$--stabilisers of adjacent vertices of $\Fix(C,\mc{S})$. Set $\overline A:=N_A(C)/C$ and $\overline B:=N_B(C)/C$. Depending on the number of orbits of vertices, we have $N_G(C)/C=\overline A\ast\overline B$ or $N_G(C)/C=\overline A\ast\Z$, where $\overline A$ and $\overline B$ are elliptic in $\Fix(C,\mc{S})$, while the $\Z$--factor is loxodromic. By our third assumption, $\overline A$ is not virtually abelian in the former case (up to swapping $A$ and $B$), and $\overline A$ is nontrivial in the latter. 

By Corollary~\ref{cc normalisers}, $N_G(C)/C$ is virtually special, hence so is $\overline A$. In particular, the centre of $\overline A$ virtually splits as a direct factor. In addition, since $C$ is closed under taking roots, $N_G(C)/C$ is torsion-free. Thus, if $N_G(C)/C=\overline A\ast\overline B$, there exists $\overline a\in \overline A$ such that $\overline a$ projects to an infinite order element of $\overline A$ modulo its centre. If instead $N_G(C)/C=\overline A\ast\Z$, there exists $\overline a\in\overline A$ simply of infinite order in $\overline A$.

Let $\overline\varphi$ be the DLS automorphism of $N_G(C)/C$ induced by the above free splitting and the element $\overline a$. Recall that $\overline\varphi$ is the identity on $\overline A$ and the conjugation by $\overline a$ on $\overline B$. Since the splitting of $N_G(C)/C$ is free, and because of our choice of $\overline a$, it is straightforward to see that $\overline\varphi$ has infinite order in the outer automorphism group of $N_G(C)/C$.

Possibly replacing $\overline a$ with a power, Corollary~\ref{cc normalisers} shows that there exists an element $a\in Z_A(C)$ projecting to $\overline a$. Let $\varphi$ be the DLS automorphism of $G$ induced by its splitting and the element $a$. Note that part~(2) of Lemma~\ref{pf lemma} ensures that we can take $\langle a\rangle\perp C$, so that, in the HNN case, $\varphi$ is indeed a fold.

Since $\varphi|_C=\id_C$, the automorphism $\varphi$ leaves $N_G(C)$ invariant and projects to $\overline\varphi$ on $N_G(C)/C$. Also note that $\varphi$ is the identity on $A$. Thus, if a power of $\varphi$ were an inner automorphism of $G$, it would be the conjugation by an element of $Z_G(A)\leq N_G(C)$. 
% we don't need $A$ to fix a unique vertex of $T$ for this
This would contradict the fact that $\overline\varphi$ is an infinite-order outer automorphism of $N_G(C)/C$. This proves the lemma.
\end{proof}

\subsection{Proof of Theorems~\ref{main cmp} and~\ref{main general}}\label{main argument sect}

As discussed at the beginning of Subsection~\ref{limit subsec}, infinite sequences in $\out(G)$ give rise to non-elliptic $G$--actions on $\R$--trees $G\acts T_{\om}$. In this subsection, we show how to use such actions to obtain the required simplicial splittings of $G$, along with DLS automorphisms with infinite order in $\out(G)$.

Throughout, we consider the setup of Subsection~\ref{limit subsec}. In particular, we use the notation $\varphi_n$, $T_n$ and $T_{\om}$ with the same meaning as Assumption~\ref{trees notation}. 

Having introduced BF--stability in Subsection~\ref{Guirardel sect}, we can now record the following immediate consequence of Proposition~\ref{arc-stab prop} and Lemma~\ref{UCC for centraliser kernels}.

\begin{cor}\label{stable cor}
The action $G\acts T_{\om}$ is BF--stable.
\end{cor}

Recall that the action $G\acts T_{\om}$ is non-elliptic by construction, so the following makes sense.

\begin{defn}
We denote by $T\cu T_{\om}$ the $G$--minimal subtree.
%Maybe better to write $T_{\rm min}$ or $T_{\om}^{\rm min}$ or $T_{\om}^m$ ?
\end{defn}

We now prove three long, technical propositions that make up the core of the proof of Theorems~\ref{main cmp} and~\ref{main general}. Each of them exploits certain features of the action $G\acts T$ to construct a splitting of $G$ and (possibly) a DLS automorphism. It is useful to recall Proposition~\ref{arc-stab prop} and Proposition~\ref{important addenda new} to better understand their relevance.

\begin{prop}\label{Z vertex group prop}
Let $\alpha\cu T$ be a line acted upon by its stabiliser $Z$ via a nontrivial homomorphism $\rho\colon Z\ra\R$. Then one of the following happens:
\begin{enumerate}
\item $\rho$ is discrete with $G$--semi-parabolic kernel; 
\item $G$ admits a splitting where $Z$ is a vertex group and each incident edge group is contained in $\ker\rho$ (we include here also the ``trivial'' case when $G=Z$ and $T=\alpha$);
\item there exists a geometric approximation $\mc{G}\ra T$ and an indecomposable component $U\cu\mc{G}$ such that $U\not\simeq\R$ and the stabiliser of every arc of $U$ coincides with $\ker\rho$, which is a centraliser.
\end{enumerate}
\end{prop}
\begin{proof}
Assume throughout the proof that $\rho$ is not discrete with $G$--semi-parabolic kernel. We will either construct a splitting as in Option~(2) or find a component of a geometric approximation as in Option~(3). We remind the reader that $\ker\rho$ can be infinitely generated.

Observe that, by Lemma~\ref{kernel envelope}, either $\G(\ker\rho)=\G(Z)$, or an element of the centre of $Z$ lies outside $\ker\rho$. In addition, if $\G(\ker\rho)\neq\G(Z)$ then $\rho$ cannot be discrete, otherwise Proposition~\ref{kernel dichotomy} would imply that $\ker\rho$ is $G$--semi-parabolic, violating our initial assumption.

Thus, it suffices to consider the following two cases, which we will treat by rather different arguments. Fix a finite generating set $Z_0\cu Z$ with $Z_0=Z_0^{-1}$ and $1\in Z_0$.

\smallskip
{\bf Case~(a):} \emph{$\rho$ is not discrete and there exists an element $\overline z\in Z_Z(Z)\setminus\ker\rho$.}

\smallskip\noindent
Since morphisms are $1$--Lipschitz, we have $\ell_{\mc{G}}(g)\geq\ell_T(g)$ for every geometric approximation $\mc{G}\ra T$ and every element $g\in G$. Proposition~\ref{geometric approximations} ensures that we can choose a geometric approximation $f\colon\mc{G}\ra T$ such that each element of $Z_0\cup\{\overline z\}$ has the same translation length in $\mc{G}$ and $T$. 

Now, since $\overline z$ is loxodromic in $\mc{G}$ and commutes with $Z$, its axis is $Z$--invariant. This shows that the $Z$--minimal subtree of $\mc{G}$ is a line $\wt\alpha$ with $f(\wt\alpha)=\alpha$. 
% since $Z$ leaves invariant only one line in $T$
In addition, $Z$ translates along $\wt\alpha\cu\mc{G}$ and $\alpha\cu T$ according to the same homomorphism $\rho\colon Z\ra\R$, since the elements of the generating set $Z_0$ have the same translation length in $\mc{G}$ and $T$. 

Let $\mc{U}$ be the transverse covering of $\mc{G}$ provided by Proposition~\ref{Imanishi}. Let $U\in\mc{U}$ be a component that shares an arc with $\wt\alpha$. 

If $U$ is not indecomposable, then $U$ is an arc containing no branch points of $\mc{G}$ in its interior. In this case, we have $U\cu\wt\alpha$. Since $\rho$ is not discrete, there exist elements of $Z$ that translate arbitrarily little along $\wt\alpha$. It follows that $\wt\alpha$ contains no branch points of $\mc{G}$, hence $\mc{G}=\wt\alpha$ and $T=\alpha$. Thus $G=Z$ and we are in the ``trivial'' case of Option~(2).

Suppose instead that $U$ is indecomposable and let $G_U\leq G$ be its stabiliser. Since $\rho$ is not discrete, for every $\eps>0$, the group $Z$ is generated by its elements with translation length $<\eps$. Thus, $Z$ is generated by elements $g\in Z$ such that $gU\cap U$ contains an arc, since $U$ and $\wt\alpha$ share an arc. Since $U$ is part of a transverse covering, these generators preserve $U$, hence $Z\leq G_U$. In particular, we have $\wt\alpha\cu U$.

If the image $f(U)$ is a line, then $f(U)=\alpha$, hence $G_U=Z$. Recall from the discussion after Definition~\ref{transverse covering defn} that $G_U$ is a vertex group in the splitting of $G$ given by the action $G\acts S_{\mc{U}}$. All stabilisers of incident edges are subgroups of $G_U=Z$ that are elliptic in $\mc{G}$, hence in $T$. This shows that they are contained in $\ker\rho$. Thus, this splitting is as required in Option~(2) of the proposition.

Finally, suppose that $f(U)$ is not a line. By Lemma~\ref{indec lem}, the action $G_U\acts f(U)$ is minimal. Since $\alpha\cu f(U)$, it follows that $f(U)$ contains an infinite tripod $\tau$ containing $\alpha$. Lemma~\ref{indec lem} also shows that $f(U)$ is stable, so $G_{\tau}$ coincides with the $G$--stabiliser of any arc of $\alpha$. By Proposition~\ref{important addenda new}(e), the latter is exactly $\ker\rho$, so $G_{\tau}=\ker\rho$. Proposition~\ref{line and tripod prop} now implies that $\ker\rho$ is a centraliser. In particular, $\ker\rho$ is finitely generated and, up to changing geometric approximation and indecomposable component, Lemma~\ref{stable components} allows us to assume that $\ker\rho$ is the stabiliser of every arc of $U$. This is the situation described in Option~(3) of the statement, so this concludes the discussion of Case~(a).

\smallskip
{\bf Case~(b):} \emph{we have $\G(\ker\rho)=\G(Z)$.}

\smallskip\noindent
In this case, we will show that we always fall in Option~(2) of the proposition. Let $K\leq\ker\rho$ be a finitely generated subgroup such that any centraliser containing $K$ contains $Z$, as provided by Remark~\ref{filling subgroup rmk}.

\smallskip
{\bf Claim~1:} \emph{we have $\Fix(K,T)=\alpha$ and the stabiliser of every arc of $\alpha$ is exactly $\ker\rho$.}

\smallskip\noindent
\emph{Proof of Claim~1.}
It is clear that $\alpha\cu\Fix(K,T)$. Consider an arc $\eta\cu\Fix(K,T)$.

Observe that $\eta$ cannot fall in Case~(1) of Proposition~\ref{arc-stab prop}. Indeed, $G_{\eta}$ would be a centraliser and, since $K$ is provided by Remark~\ref{filling subgroup rmk}, it would follow that $Z\leq G_{\eta}$. However, this would contradict the fact that $Z$ is not elliptic in $T$ (since $\rho$ is nontrivial).

Thus, $\eta$ must fall in Case~(2) of Proposition~\ref{arc-stab prop}. In particular, $G_{\eta}$ is the kernel of a homomorphism $\rho'\colon Z'\ra\R$, where $Z'$ is a centraliser stabilising a line $\alpha'\cu T$ containing $\eta$. Again, we must have $Z\leq Z'$, so $Z$ stabilises $\alpha'$. Since $Z$ translates nontrivially along $\alpha$, we must have $\alpha=\alpha'$, hence $Z=Z'$. This shows that $\eta\cu\alpha$ and $G_{\eta}=\ker\rho$, thus proving the claim.
\hfill$\blacksquare$

\smallskip
Fix a point $p\in\alpha$ and let $\beta\cu\alpha$ be the convex hull of the set $Z_0\cdot p$. Proposition~\ref{geometric approximations} allows us to choose a geometric approximation $f\colon\mc{G}\ra T$ so that $\beta$ lifts isometrically to a $K$--fixed arc $\wt\beta\cu\mc{G}$. 

Since $\rho$ is nontrivial, $Z$ is not elliptic in $T$, nor can it be elliptic in $\mc{G}$. Let $\mc{S}_Z\cu\mc{G}$ be the $Z$--minimal subtree. 
% this is the only point where we really need $\rho$ to not be trivial, rather than just $Z$ being non-elliptic in $\om$--all $T_n$
Let $\mc{U}$ be the transverse covering of $\mc{G}$ provided by Proposition~\ref{Imanishi}. 

\smallskip
{\bf Claim~2:} \emph{if $g\mc{S}_Z\cap\mc{S}_Z$ contains an arc, for some $g\in G$, then $g\in Z$. In addition, if some $U\in\mc{U}$ shares an arc with $\mc{S}_Z$, then $U\cu\mc{S}_Z$.}

\smallskip\noindent
\emph{Proof of Claim~2.}
If $\wt p\in\wt\beta$ is the lift of $p$, the arc $\wt\beta$ contains $Z_0\cdot\wt p$. Recalling that $Z_0$ generates $Z$ and that $1\in Z_0=Z_0^{-1}$, we deduce that the $Z$--minimal subtree $\mc{S}_Z$ is contained in $Z\cdot\wt\beta$. In particular, every arc of $\mc{S}_Z$ contains a sub-arc that is fixed by a $Z$--conjugate of $K$.

Suppose that $g\mc{S}_Z\cap\mc{S}_Z$ contains a nontrivial arc $\eta$ for some $g\in G$. By the previous paragraph, we can choose $\eta$ so that it is simultaneously fixed by $z_1Kz_1^{-1}$ and $(gz_2)K(gz_2)^{-1}$, for some $z_1,z_2\in Z$. Up to shrinking $\eta$, the morphism $f$ is isometric on it, and $f(\eta)$ is an arc of $T$ fixed by $z_1Kz_1^{-1}$ and $(gz_2)K(gz_2)^{-1}$. 

The first half of Claim~1 implies that $f(\eta)\cu \alpha\cap g\alpha$. In particular, $\alpha$ and $g\alpha$ share an arc, so the second half of Claim~1 implies that $g(\ker\rho)g^{-1}=\ker\rho$. Since $\G(\ker\rho)=\G(Z)$, Lemma~\ref{normalisers of subgroups of AP}(1) implies that $ N_G(\ker\rho)\leq N_G(Z)=Z$. In conclusion, $g\in Z$ as required. 

Now, suppose that a component $U\in\mc{U}$ shares an arc with $\mc{S}_Z$. If $U$ is not indecomposable, then $U$ is an arc containing no branch points of $\mc{G}$ in its interior, so it is clear that $U\cu\mc{S}_Z$. 

Suppose instead that $U$ is indecomposable. As above, $f(U\cap\mc{S}_Z)$ contains an arc fixed by a $Z$--conjugate of $K$. Lemma~\ref{indec lem} shows that $f(U)$ is a stable subtree of $T$, so $f(U)$ is fixed pointwise by a $Z$--conjugate of $K$. Claim~1 implies that $f(U)\cu\alpha$. Since $f$ is a morphism, $f(U)$ is not a single point and, by Lemma~\ref{indec lem}, it is $G_U$--minimal. We conclude that $f(U)=\alpha$, hence $G_U\leq Z$. Since $U$ is the $G_U$--minimal subtree of $\mc{G}$, it follows that $U\cu\mc{S}_Z$.
\hfill$\blacksquare$

\smallskip
Note that $\mc{S}_Z$ is closed in $\mc{G}$. Indeed, every point $x\in\overline{\mc{S}}_Z$ is the missing endpoint of a half-open arc $\s\cu\mc{S}_Z$. By Definition~\ref{transverse covering defn}, $\overline\s$ is covered by finitely many elements of $\mc{U}$, so one of these must intersect $\overline\s$ in an arc containing $x$. By Claim~2, this element of $\mc{U}$ is contained in $\mc{S}_Z$, hence $x\in\mc{S}_Z$.

Now, consider the covering $\mc{V}$ of $\mc{G}$ whose elements are either $G$--translates of $\mc{S}_Z$, or elements of $\mc{U}$ that are not contained in any $G$--translate of $\mc{S}_Z$. By Claim~2, $\mc{V}$ is a transverse covering of $\mc{G}$. 
% it is also useful to notice that $\mc{U}$ is countable
Let $G\acts S_{\mc{V}}$ be the minimal action on a bipartite simplicial tree constructed as described after Definition~\ref{transverse covering defn}.

By Claim~2, $Z$ is the $G$--stabiliser of $\mc{S}_Z$, which corresponds to a vertex of $S_{\mc{V}}$. Stabilisers of incident edges are $Z$--stabilisers of points of $\mc{S}_Z$. In particular, they are elliptic in $\mc{G}$, hence in $T$, so they must be contained in $\ker\rho$. In conclusion, we have realised the situation in Option~(2) of the proposition. This concludes the proof.
\end{proof}

In Option~(2) of Proposition~\ref{Z vertex group prop} we will be able to obtain an HNN splitting of $G$ by applying Lemma~\ref{refining splittings} to the natural HNN splittings of $Z$ induced by $\rho$. The next result shows how to handle Option~(3) instead.

\begin{prop}\label{exotic/surface prop}
Consider a geometric approximation $f\colon\mc{G}\ra T$. Let $U\cu\mc{G}$ be an indecomposable component with $U\not\simeq\R$ such that every arc of $U$ has the same stabiliser $H\leq G$. Also suppose that $H$ is convex-cocompact and closed under taking roots in $G$. Then one of the following happens:
\begin{itemize}
\item[(a)] $G$ splits over $H$, giving rise to a fold or partial conjugation with infinite order in $\out(G)$.
\item[(b)] $G$ splits over a centraliser $Z_G(k)$, where $k\in G$ is label-irreducible and $H\lhd Z_G(k)$ with $Z_G(k)/H\simeq\Z$. In addition, the twist $\psi\in\aut(G)$ determined by $k$ and this splitting has infinite order in $\out(G)$.
\end{itemize}
\end{prop}
\begin{proof}
Let $G_U\leq G$ be the stabiliser of $U$. Clearly, $H$ is the kernel of the action $G_U\acts U$ and the induced action $G_U/H\acts U$ has trivial arc-stabilisers. The latter action is still indecomposable, and geometric by Remark~\ref{geometric modulo kernel}. Note that $G_U/H$ is finitely presented, since $H$ is finitely generated, and torsion-free, since $H$ is closed under taking roots. 
% I'm not entirely sure that $G_U$ will be finitely presented, however note that by \cite[Remark~1.29]{Gui08} we can still apply \cite[Proposition~A.8]{Gui08}, see \cite[Definition~A.3]{Gui08}
Thus, we can invoke Proposition~\ref{3 types}. The ``axial'' case does not occur since $U$ is not a line. The two cases of the current proposition will correspond, respectively, to the ``exotic'' and ``surface'' cases.

Let $G\acts S_{\mc{U}}$ be the simplicial tree provided by Proposition~\ref{Imanishi} and the discussion after Definition~\ref{transverse covering defn}. The subgroup $G_U$ is the stabiliser of a vertex $u\in S_{\mc{U}}$. Let $\mscr{E}$ be the collection of stabilisers of edges of $S_{\mc{U}}$ incident to $u$. Note that $\mscr{E}$ is a union of finitely many $G_U$--conjugacy classes of subgroups of $G_U$, and each element of $\mscr{E}$ is the $G_U$--stabiliser of a point of $U$.

In view of Lemma~\ref{refining splittings}, our goal is to construct a $1$--edge splitting of $G_U$ in which all elements of $\mscr{E}$ are elliptic. For this purpose, it suffices to construct a $1$--edge splitting of $G_U/H$ in which all elements of the collection $\overline{\mscr{E}}$ of projections of elements of $\mscr{E}$ are elliptic. We treat the exotic and surface cases separately.

\smallskip
{\bf Case~(a):} \emph{the action $G_U/H\acts U$ is of exotic type.}

\smallskip\noindent
By Proposition~7.2 and Theorem~6.2 in \cite{Gui-CMH}, the action $G_U/H\acts U$ is a limit (in the length function topology) of actions on simplicial trees $G_U/H\acts\mc{S}_n$ where all edge stabilisers are trivial and all elements of $\overline{\mscr{E}}$ are elliptic. 

Picking any $\mc{S}_n$ and collapsing all orbits of edges but one, we obtain an action on a simplicial tree $G_U/H\acts\mc{S}$ with a single orbit of edges. This corresponds to a splitting of $G_U/H$ as $A\ast B$ or $A\ast\Z$, where every element of $\overline{\mscr{E}}$ is conjugate into either $A$ or $B$, and the possible $\Z$--factor is loxodromic in $\mc{S}$. Via Lemma~\ref{refining splittings}, this induces a $1$--edge splitting of $G$ over $H$. 

Since $N_G(H)/H$ contains $G_U/H$, it is clear that $N_G(H)/H$ is not cyclic and that $N_G(H)$ is not elliptic in the Bass--Serre tree of the splitting of $G$. We would like to obtain a fold or partial conjugation with infinite order in $\out(G)$ by applying Lemma~\ref{non-inner from cc}. For this, we are left to ensure that $N_G(H)/H$ is not a free product of two virtually abelian groups elliptic in the Bass--Serre tree. In fact, it suffices to choose $\mc{S}$ so that $G_U/H$ is not a free product of two virtually abelian groups elliptic in $\mc{S}$.
% we are not going through Kurosh's theorem here (which yields a possible annoying loxodromic free factor); this really just because of how we refine splittings and because we only have two factors

Suppose that $G_U/H=V_1\ast V_2$, where the $V_i$ are nontrivial virtually abelian groups (otherwise any choice of $\mc{S}$ will do). If neither $V_1$ nor $V_2$ is isomorphic to $\Z$, then they are both elliptic in all the $\mc{S}_n$, hence they are elliptic in $U$. Since $G_U/H$ acts on $U$ with trivial arc-stabilisers, we obtain a $G_U$--invariant simplicial subtree of $U$, contradicting the fact that $U$ is indecomposable.

Thus, suppose that $V_2\simeq\Z$. If again $V_1\not\simeq\Z$, then $V_1$ is elliptic in all $\mc{S}_n$ and the same argument shows that $V_2$ must be loxodromic for large $n$. In this case, every element of $\overline{\mscr{E}}$ is conjugate into $V_1$, so we can simply take $\mc{S}=\mc{S}_n$.

Finally, suppose that $G_U/H\simeq F_2$. Note that, for each $n$, every subgroup in $\overline{\mscr{E}}$ is contained in a free factor of $F_2$ that is elliptic in $\mc{S}_n$. Since distinct free factors of $F_2$ intersect trivially, if a free factor contains a nontrivial element of $\overline{\mscr{E}}$, then it must be elliptic in all $\mc{S}_n$, hence also in $U$. We conclude that there is at most one conjugacy class of free factors of $F_2$ that contains nontrivial elements of $\overline{\mscr{E}}$. If $\langle x\rangle$ is one such free factor, it suffices to take $\mc{S}$ to be the HNN splitting $F_2=\langle x\rangle\ast_{\{1\}}$. 

\smallskip
{\bf Case~(b):} \emph{the action $G_U/H\acts U$ is of surface type.}

\smallskip\noindent
In this case, we have $G_U/H=\pi_1\Sigma$ for a compact surface with boundary $\Sigma$. The action $\pi_1\Sigma\acts U$ is dual to an arational measured foliation on $\Sigma$. Since the subgroups $\overline{\mscr{E}}$ are elliptic in $U$, they are contained in the fundamental groups of the boundary components of $\Sigma$.

Let $\g$ be an essential simple closed curve on $\Sigma$ representing a nonzero homology class in $H_1(\Sigma,\Z)$. In particular, $\g$ is $2$--sided in $\Sigma$, and $\langle\g\rangle$ is a maximal cyclic subgroup of $\pi_1\Sigma$. Dual to $\g$, we have a simplicial $\pi_1\Sigma$--tree with edge-stabilisers conjugate to $\langle\g\rangle$, in which all elements of $\overline{\mscr{E}}$ are elliptic. 

Let $g\in G_U$ be a lift of $\g$. Note that $g$ is loxodromic in $U$, since the foliation on $\Sigma$ is arational. Lemma~\ref{refining splittings} gives a $1$--edge splitting of $G$ over the subgroup $C=H\rtimes\langle g\rangle$. 
%In this case, we have $G_U/H=\pi_1\Sigma$ for a compact surface with boundary $\Sigma$. The action $\pi_1\Sigma\acts U$ is dual to an arational measured foliation on $\Sigma$. By \cite[Theorem~8.1]{Gui-CMH}, we can approximate this action by simplicial $\pi_1\Sigma$--trees with infinite cyclic edge-stabilisers while keeping all elements of $\overline{\mscr{E}}$ elliptic. Note that these cyclic edge-stabilisers are loxodromic in $U$, since they correspond to essential simple closed curves on $\Sigma$ and the foliation dual to $U$ is arational.
%
%Now, Lemma~\ref{refining splittings} gives a $1$--edge splitting of $G$ over a subgroup of the form $C=H\rtimes\langle g\rangle$, where $g\in G_U$ is loxodromic in $U$. 

\smallskip
{\bf Claim:} \emph{there exists a label--irreducible element $k\in C$ such that $C=Z_G(k)$.}

\smallskip\noindent
\emph{Proof of claim.}
Recall that $N_G(H)$ virtually splits as $H\x K$ with $K$ convex-cocompact in $G$. Thus, for some $n\geq 1$, we can write $g^n=hk$ with $h\in H$ and $k\in K$. Since $H$ and $g$ commute with $k$, 
% $k$ is a label-irreducible component of $g^n$
we have $C\leq Z_G(k)$. Conversely, note that $g$ and $k$ are loxodromic in $\mc{G}$ with the same axis, which is contained in $U$. This axis is preserved by $Z_G(k)$, so $Z_G(k)\leq G_U$. Since $G_U/H$ is hyperbolic and $\langle g\rangle$ projects to a maximal cyclic subgroup of $G_U/H$, which also contains the projection of $k$, we conclude that $Z_G(k)\leq H\rtimes\langle g\rangle=C$. This shows that $C=Z_G(k)$. 

In particular, $C$ is convex-cocompact in $G$ and it has a finite-index subgroup of the form $H\x\langle k\rangle$. Recalling that $K$ is convex-cocompact in $G$ and $K\cap H=\{1\}$, this shows that $\langle k\rangle=C\cap K$ is convex-cocompact. Hence $k$ is label-irreducible, proving the claim. 
\hfill$\blacksquare$

\smallskip
%Finally, let $\psi\in\aut(G)$ be the twist determined by $k$ and our splitting of $G$. Observe that $\psi|_H=\id_H$ and that $\psi(G_U)=G_U$, with the restriction to $G_U/H=\pi_1\Sigma$ being a (conventional) Dehn twist in the mapping class group of $\Sigma$. In particular, $\psi$ restricts to an automorphism of $G_U/H$ with infinite order in $\out(G_U/H)$.
Finally, let $\psi\in\aut(G)$ be the twist determined by $k$ and our splitting of $G$. Observe that $\psi|_H=\id_H$ and that $\psi(G_U)=G_U$, with the restriction to $G_U/H=\pi_1\Sigma$ being the (conventional) Dehn twist around $\g$ in the mapping class group of $\Sigma$. In particular, $\psi$ restricts to an automorphism of $G_U/H$ with infinite order in $\out(G_U/H)$.

Also note that $\psi$ is the identity on $Z_G(k)$. Thus, if $\psi$ were an inner automorphism of $G$, then it would have to be the conjugation by an element of $Z_GZ_G(k)\leq Z_G(k)\leq G_U$. Hence $\psi$ would restrict to an inner automorphism of $G_U/H$, contradicting the previous paragraph. The same argument applies to powers of $\psi$, so $\psi$ has infinite order in $\out(G)$, as required.
\end{proof}

Finally, the next result covers the situation where every line in $T$ falls in Option~(1) of Proposition~\ref{Z vertex group prop} and we are also unable to apply Proposition~\ref{exotic/surface prop}.

\begin{prop}\label{simplicial prop}
Suppose that the following hold:
\begin{itemize}
\item every geometric approximation $\mc{G}\ra T$ is simplicial;
\item every line of $T$ is acted upon discretely by its $G$--stabiliser;
\item the $G$--stabiliser of every stable arc of $T$ is $G$--semi-parabolic;
% only asking this for **stable** arcs
\item $T$ is not a line.
\end{itemize}
Then one of the following happens:
\begin{enumerate}
\item $G$ splits over the stabiliser of a stable arc of $T$, giving rise to a fold or partial conjugation with infinite order in $\out(G)$;
\item $G$ splits over some $Z_G(g)$, where $g\in G$ is label-irreducible and determines a twist with infinite order in $\out(G)$;
\item $T$ contains a line $\alpha$ falling in Option~(2) of Proposition~\ref{Z vertex group prop}.
\end{enumerate}
\end{prop}
\begin{proof}
Let $\beta\cu T$ be an arc such that $G_{\beta}$ is maximal among all stabilisers of arcs of $T$ (such an arc exists by Lemma~\ref{UCC for centraliser kernels}). Note that $\beta$ is a stable arc and set $H:=G_{\beta}$ for simplicity. 

By Proposition~\ref{geometric approximations}, we can choose a geometric approximation $f\colon\mc{G}\ra T$ with an $H$--fixed edge $e\cu\mc{G}$ such that $f$ is isometric on $e$ and $\beta\cu f(e)$ (up to shrinking $\beta$). 

Let $G\acts\mc{S}$ be the $1$--edge splitting obtained by collapsing all edges of $\mc{G}$ outside the orbit $G\cdot e$. Let $\overline e\cu\mc{S}$ be the projection of $e$, and let $A$ and $B$ be the $G$--stabilisers of its two vertices. Note that the stabilisers of $e$ and $\overline e$ coincide with $H$.

We divide the proof into three cases, depending on the behaviour of $H$ and its normaliser. 
 
\smallskip
{\bf Case~(a):} \emph{$H$ is non-elliptic in $\om$--all $T_n$.}

\smallskip\noindent
Then we are in Case~(2) of Proposition~\ref{arc-stab prop}, so $H=\ker\rho$ for a homomorphism $\rho\colon Z\ra\R$, where $Z$ is the stabiliser of a line $\alpha\cu T_{\om}$ containing $\beta$.
% $\alpha$ might not be contained in $T$ if $\rho$ is trivial

\smallskip
{\bf Claim~1:} \emph{we have $N_G(H)=Z$.}

\smallskip\noindent
\emph{Proof of Claim~1.}
Recall from Proposition~\ref{arc-stab prop} that the $Z$--minimal subtree of $T_n$ is a line $\alpha_n$ and that the lines $\alpha_n$ converge to $\alpha$. Since $H$ is non-elliptic in $\om$--all $T_n$, its minimal subtree coincides with $\alpha_n$. Thus, if $g\in N_G(H)$, we must have $g\alpha_n=\alpha_n$ for $\om$--all $T_n$, hence $g\alpha=\alpha$. This shows that $N_G(H)\leq Z$, while the other inclusion is immediate.
\hfill$\blacksquare$

\smallskip
Suppose first that $\rho$ is trivial, so that $H=Z$. In this case, Proposition~\ref{important addenda new}(b) yields a label-irreducible element $h\in Z_H(H)$ that is loxodromic in $\om$--all $T_n$ with $\ell_{T_n}(h)\ra 0$. If $g\in G$ commutes with $h$, the argument in the proof of Claim~1 shows that $g\in Z$. In particular, we have $Z_G(h)=H$.

Let $\psi\in\aut(G)$ be the twist or partial conjugation determined by $h$ and the splitting $G\acts\mc{S}$. In order to show that we fall in Options~(1) or~(2) of the proposition, we only need to prove that $\psi$ has infinite order in $\out(G)$.

Note that the standard argument from \cite[Section~6]{Rips-Sela} applies in this case, yielding a sequence $k_n\ra+\infty$ such that, for every finite generating set $F\cu G$, we have for $\om$--all $n$:
\begin{align*}
\lim_{n\ra+\infty}\inf_{x\in T_n}\max_{f\in F}d(x,\psi^{k_n}(f)x)&<\inf_{x\in T}\max_{f\in F}d(x,fx), & \inf_{x\in T}\max_{f\in F}d(x,\psi(f)x)&=\inf_{x\in T}\max_{f\in F}d(x,fx).
\end{align*}
This shows that no power of $\psi$ can be an inner automorphism of $G$, as required.
%Here is an alternative, self-contained proof.
% (1) We need to ensure that $\Min(a,T_n)$ and $\Min(b,T_n)$ have bounded (or empty) intersection with the axis of $h$ in $T_n$, since then we can apply the argument from Case~(b).
% (2) If $\Min(a,T_n)$ has unbounded intersection with the axis $\alpha_n$ of $h$ in $T_n$, then $a\alpha_n=\alpha_n$. (use that we can pass to label-irreducible components without altering axis and translation length in $\T_v$, plus the lemmas from \cite{Fio10a})
% (3) The lines $\alpha_n$ converge to the $Z$--minimal line of $T$. This cannot be fixed by $a$ since $a\not\in H$. The only problem is $a\in Z$ and $a$ is loxodromic.
% (4) The idea is conjugating $a$ by some element of $A$ to avoid $a\in Z$, which automatically avoids $a\in H$. Note that, since $H$ is closed under taking roots, it has infinite index in $A$, and $A$ is also closed under taking roots. If $A$ is not contained in $Z$, we win: there exists an element that does not stabilise the $Z$--minimal line, and conjugating $a$ by this element changes its axis.
% (5) Otherwise $A$ is contained in $Z$. Actually, since $A$ contains $H$ and is closed under taking roots, we must have $A=Z$. This can be avoided by choosing carefully $\mc{G}$ (so that a large portion of the $Z$--minimal line lifts.

Suppose now instead that $H=\ker\rho$ is a proper subgroup of $Z$. Recall that we are assuming that $\rho$ has discrete image and that $H=\ker\rho$ is $G$--semi-parabolic, hence convex-cocompact. 

Thus, we can write $Z=H\rtimes\langle z\rangle$ for some $z\in Z$. Corollary~\ref{cc normalisers}(1) guarantees that $hz^k$ commutes with $H$ for some $h\in H$ and $k\geq 1$. This element commutes with a finite-index subgroup of $Z$, hence with the entire $Z$, because $G$ is a subgroup of $\A_{\G}$. This shows that the centre of $Z$ contains an element outside $\ker\rho$. 

Proceeding as in Case~(a) of the proof of Proposition~\ref{Z vertex group prop}, we can ensure that the chosen geometric approximation $\mc{G}$ contains a $Z$--invariant line $\wt\alpha$ on which $Z$ acts via the homomorphism $\rho$. Note that $f(\wt\alpha)=\alpha$, so the $G$--stabiliser of $\wt\alpha$ must coincide with $Z$. 

Note that distinct $G$--translates of $\wt\alpha$ can share at most one point. Indeed, if $g\wt\alpha$ and $\wt\alpha$ share an edge, then $gHg^{-1}$ fixes an arc of $\wt\alpha$, hence an arc of $\alpha$. Since we chose $H$ so that it is maximal among stabilisers of arcs of $T$, the stabiliser of every arc of $\alpha$ is equal to $H$. In conclusion, we have $gHg^{-1}\leq H$, and the symmetric argument yields $gHg^{-1}=H$. By Claim~1, we obtain $g\in Z$, hence $g\wt\alpha=\wt\alpha$.

Since $\mc{G}$ is simplicial, we obtain a transverse covering of $\mc{G}$ made up of the $G$--translates of $\wt\alpha$ and all edges of $\mc{G}$ that are not contained in any $G$--translate of $\wt\alpha$. Proceeding as at the end of Case~(b) of Proposition~\ref{Z vertex group prop}, we end up in the situation of Option~(2) of Proposition~\ref{Z vertex group prop} (which is Option~(3) of the current proposition).

This completes the discussion of Case~(a). In the remaining two cases, we will always construct folds or partial conjugations arising from $G\acts\mc{S}$, thus ending up in Option~(1) of the proposition.
 
Before we proceed, recall that $A$ and $B$ are the stabilisers of the two vertices of $\overline e\cu\mc{S}$. The stabiliser of $\overline e$ is $H$. We make the following observations.

\smallskip
{\bf Claim~2:} \emph{the sets $A\setminus H$ and $B\setminus H$ are both nonempty.}

\smallskip\noindent
\emph{Proof of Claim~2.}
This is clear if $\mc{S}$ gives an amalgamated product splitting of $G$. If it gives an HNN splitting with stable letter $t$, we are also fine unless $A=H$ and either $tHt^{-1}\geq H$ or $tHt^{-1}\leq H$. Since $H$ is convex-cocompact, this can only occur if $t\in N_G(H)$, because of Lemma~\ref{cc conjugates}. But then Corollary~\ref{cc normalisers}(1) implies that $ht^k$ commutes with $H$ for some $h\in H$ and $k\geq 1$, so the axis of $ht^k$ in $T$ is $\langle H,t^k\rangle$--invariant, hence $G$--invariant. This implies that $T$ is a line, contradicting our assumptions.
\hfill$\blacksquare$

\smallskip
{\bf Claim~3:} \emph{if $H$ is elliptic in $\om$--all $T_n$, then $N_G(H)$ is non-elliptic in $\om$--all $T_n$ and the $N_G(H)$--minimal subtree of $\om$--all $T_n$ is not a line.}

\smallskip\noindent
\emph{Proof of Claim~3.}
By Lemma~\ref{almost dense orbits}, there exists a sequence $\eps_n\ra 0$ such that $N_G(H)$ acts with $\eps_n$--dense orbits on $\Fix(H,T_n)$. Since $\beta$ can be approximated by a sequence of arcs $\beta_n\cu\Fix(H,T_n)$, which have length bounded away from zero, this shows that $N_G(H)$ is non-elliptic in $\om$--all $T_n$.

Since $\Fix(H,T_n)$ is $N_G(H)$--invariant, it contains the $N_G(H)$--minimal subtree as an $\eps_n$--dense subset. If the latter is a line $\alpha_n\cu T_n$ for $\om$--all $n$, then these lines converge to an $N_G(H)$--invariant line $\alpha\cu T_{\om}$. Again, since $\beta$ is approximated by arcs in $\Fix(H,T_n)$, we have $\beta\cu\alpha$.

Since $N_G(H)$ is not elliptic in $T_n$, we have $N_G(H)\neq H$. Thus, since $H=G_{\beta}$, the normaliser $N_G(H)$ must translate nontrivially along $\alpha$. This shows that $\alpha$ is contained in the $G$--minimal subtree $T\cu T_{\om}$, so our assumptions guarantee that $N_G(H)$ acts discretely on $\alpha$. Since the kernel of the action on $\alpha$ is exactly $H$, it follows that $N_G(H)/H\simeq\Z$.

Now, let $g\in N_G(H)$ be an element generating this quotient. By the above discussion, we must have $\ell_{T_n}(g)\leq\eps_n\ra 0$, contradicting the fact that $g$ is not elliptic in $T$.
\hfill$\blacksquare$

\smallskip
{\bf Case~(b):} \emph{$H$ is elliptic in $\om$--all $T_n$ and $N_G(H)$ is elliptic in $T$.}

\smallskip\noindent
Since $N_G(H)$ is finitely generated (e.g.\ by Corollary~\ref{cc normalisers}(3)), we can choose the geometric approximation $f\colon\mc{G}\ra T$ so that $N_G(H)$ is elliptic in $\mc{G}$, hence in $\mc{S}$. Up to replacing $A,B,H$ with $G$--conjugates and swapping $A$ and $B$, we can assume that $N_G(H)\leq A$. In particular, $Z_G(H)=Z_A(H)$.

By Claim~2, we can pick elements $a\in A\setminus H$ and $b\in B\setminus H$.

\smallskip
{\bf Claim~4:} \emph{there exist $n$ and $z\in Z_G(H)$ such that $z$ is loxodromic in $T_n$ with axis that has bounded (or empty) intersection with both $\Min(a,T_n)$ and $\Min(b,T_n)$.}

\smallskip\noindent
\emph{Proof of Claim~4.} 
Recall that $\langle H,Z_G(H)\rangle$ has finite index in $N_G(H)$ by Corollary~\ref{cc normalisers}(1). Thus, $Z_G(H)$ and $N_G(H)$ have the same minimal subtree in $\om$--all $T_n$, since $H$ is elliptic. By Claim~3, this minimal subtree is well-defined and its boundary is a Cantor set. On the other hand, if $a$ is loxodromic in $T_n$, then the boundary of its axis consists of only two points.

Approximate $\beta$ by a sequence of arcs $\beta_n\cu\Fix(H,T_n)$. If $a$ is elliptic in $T_n$, then the length of $\Fix(a,T_n)\cap\beta_n$ must go to zero, since $a$ does not fix any portion of the stable arc $\beta$. The same holds for $b$. Note that, for $\om$--all $n$, the arc $\beta_n$ contains several branch points of $\Fix(H,T_n)$ because of Lemma~\ref{almost dense orbits}. 

We conclude that $\Fix(H,T_n)\setminus (\Min(a,T_n)\cup\Min(b,T_n))$ contains at least two disjoint rays, and the same holds for the $Z_G(H)$--minimal subtree. This yields the required element $z\in Z_G(H)$.
\hfill$\blacksquare$

\smallskip

By Corollary~\ref{cc normalisers}(1) and Lemma~\ref{pf lemma}(2), $Z_G(H)$ virtually splits as $Z_H(H)\x K$ with $K\perp H$. Since $H$ is elliptic in $T_n$, we can assume that the element $z$ provided by Claim~4 lies in $K$ (possibly replacing $z$ with a proper power and projecting it to $K$, which does not alter its axis in $T_n$). Also recall that $Z_G(H)=Z_A(H)$. Thus, $z$ and $\mc{S}$ determine a DLS automorphism $\psi\in\aut(G)$, which is necessarily a fold or partial conjugation.

Note that $\psi^k(a)=a$, while $\psi^k(b)=z^kbz^{-k}$ for all $k\geq 1$. Since $\Min(a,T_n)$ and $\Min(b,T_n)$ have bounded projection to the axis of $z$ in $T_n$, the distance between $\Min(\psi^k(a),T_n)$ and $\Min(\psi^k(b),T_n)$ diverges for $k\ra+\infty$. It follows that $\ell_{T_n}(\psi^k(ab))$ diverges for $k\ra+\infty$, showing that $\psi$ has infinite order in $\out(G)$, as required.

\smallskip
{\bf Case~(c):} \emph{$H$ is elliptic in $\om$--all $T_n$ and $N_G(H)$ is non-elliptic in $T$.}

\smallskip\noindent
Making sure that the chosen stable arc $\beta$ is contained in the axis of an element of $N_G(H)$, we can ensure that $N_G(H)$ remains non-elliptic in $\mc{S}$. Claim~3 guarantees that $N_G(H)/H$ is not cyclic.

Given Lemma~\ref{non-inner from cc}, we are only left to consider the case when $N_G(H)/H$ is a free product of virtually abelian groups $V_1\ast V_2$. Recall that $\beta\cu T$ has been chosen so that its stabiliser is maximal among stabilisers of arcs of $T$ (at the beginning of the proof). This guarantees that the action $N_G(H)/H\acts\Fix(H,\mc{G})$ gives a free splitting of $N_G(H)/H$. 
%Thus, either $N_G(H)/H$ is the free group $F_2$ acting freely on $\Fix(H,\mc{G})$, or the action $N_G(H)/H\acts\Fix(H,\mc{G})$ has exactly one orbit of edges, in which case $\mc{G}=\mc{S}$. 
The only situation where Lemma~\ref{non-inner from cc} cannot be applied is if both $V_1$ and $V_2$ are elliptic in $\mc{G}$ (and hence in $T$).

Let us show that $V_1$ and $V_2$ cannot both be elliptic in $T$. Suppose for the sake of contradiction that they are. Recall that $\mf{T}(V_i,T_n)$ denotes $\Fix(V_i,T_n)$ if this is nonempty, and the $V_i$--minimal subtree of $T_n$ otherwise, which is necessarily a line. Since $N_G(H)$ is not elliptic in $T$, the fixed sets of $V_1$ and $V_2$ in $T$ have positive distance, say $D>0$. Thus, the subtrees $\mf{T}(V_1,T_n)$ and $\mf{T}(V_2,T_n)$ are at distance $\geq D/2$ for $\om$--all $n$.
% for this it is very important that these subtrees are **lines** when they are not fixed

Recall that Lemma~\ref{almost dense orbits} yields a sequence $\eps_n\ra 0$ such that $V_1\ast V_2$ acts with $\eps_n$--dense orbits on $\Fix(H,T_n)$. In particular, note that $\Fix(H,T_n)$ and the $N_G(H)$--minimal subtree of $T_n$ are at Hausdorff distance $\leq\eps_n$. Thus, Corollary~\ref{rotating cor} shows that there exists a sequence $\eps_n'\ra 0$ such that the actions of $V_1$ and $V_2$ on $\Fix(H,T_n)$ are both $\eps_n'$--rotating, in the sense of Definition~\ref{rotating defn}.

Now, a straightforward ping-pong argument implies that an $N_G(H)$--orbit misses the ball of radius $D/4$ centred at the midpoint of the arc joining $\mf{T}(V_1,T_n)$ and $\mf{T}(V_2,T_n)$. For large $n$, we have $\eps_n<D/4$, so this is the required contradiction.
\end{proof}

\begin{rmk}
In Case~(a) of the proof of Proposition~\ref{simplicial prop}, we have constructed a \emph{shortening automorphism} in the sense of \cite{Rips-Sela}. However, in Case~(b), we have made the rather unusual choice of constructing a ``lengthening automorphism'', and in Case~(c) we have not described the resulting automorphism at all.

We followed this path in order to give a more direct proof of Proposition~\ref{simplicial prop}. Nevertheless, we want to emphasise that a shortening automorphism can indeed be constructed in each of the three cases of the proof of Proposition~\ref{simplicial prop}. This requires some more work, as one cannot simply ``contract one edge'' as in \cite[Section~6]{Rips-Sela}, but rather needs to perform a folding procedure.
\end{rmk}

We are only left to record the following simple observation before we can begin with the proof of the main theorems.

\begin{lem}\label{line centre lem}
Suppose that $T\simeq\R$ and that $G\acts T$ has discrete orbits. If the kernel of the $G$--action is $G$--semi-parabolic, then it has nontrivial centre.
\end{lem}
\begin{proof}
Let $H$ be the kernel of the $G$--action. Since $G$ acts discretely, there exists a loxodromic element $g\in G$ such that $G=H\rtimes\langle g\rangle$. Since $H$ is $G$--semi-parabolic, and in particular convex-cocompact, Corollary~\ref{cc normalisers}(1) shows that $hg^k$ commutes with $H$ for some $h\in H$ and $k\geq 1$. 

If the centre of $H$ were trivial, then the centre of $G$ would be isomorphic to $\Z$ and it would contain the element $hg^k$. In particular, $\om$--all automorphisms $\varphi_n$ would fix $hg^k$, hence $\ell_{T_n}(hg^k)\ra 0$. This would contradict the fact that $hg^k$ is loxodromic in $T$ and $T_{\om}$.
\end{proof}

We are finally ready to prove Theorems~\ref{main cmp} and~\ref{main general} and Corollary~\ref{infinite out finite cmp}.

\begin{proof}[Proof of Theorem~\ref{main cmp}]
The fact that the DLS automorphisms appearing in the statement of the theorem are coarse-median preserving follows from Theorem~\ref{CMP GDT}, which will be proved in Section~\ref{CMP GDT sect}. Here we only show that such automorphisms exist and have infinite order in $\out(G)$.

Let $G$ be a special group with $\out_{\rm cmp}(G)$ infinite. Choose a sequence $\varphi_n\in\aut_{\rm cmp}(G)$ projecting to an infinite sequence in $\out_{\rm cmp}(G)$. We can apply the construction at the beginning of Subsection~\ref{limit subsec} to obtain an action on an $\R$--tree $G\acts T_{\om}$. Let $T\cu T_{\om}$ be the $G$--minimal subtree.

By Proposition~\ref{arc-stab prop} and Proposition~\ref{important addenda new}(c1), the $G$--stabiliser of every arc of $T$ is a centraliser, and every line of $T$ is acted upon discretely by its $G$--stabiliser. In addition, $T$ is not itself a line.

Suppose first that no geometric approximation $\mc{G}\ra T$ admits indecomposable components in the transverse covering provided by Proposition~\ref{Imanishi}, i.e.\ that all geometric approximations of $T$ are simplicial. Then we can apply Proposition~\ref{simplicial prop}. Note that Option~(3) never occurs: in the notation of the proof of Proposition~\ref{simplicial prop}, it corresponds to Case~(a), when $\ker\rho$ is a proper subgroup of $Z$ and is non-elliptic in $\om$--all $T_n$. This is ruled out by Proposition~\ref{important addenda new}(c1).

In conclusion, we are in Options~(1) or~(2) of Proposition~\ref{simplicial prop}, so $G$ splits over a centraliser, giving rise to a DLS automorphism that conforms to the requirements in the statement of Theorem~\ref{main cmp}.

To conclude the proof, we are left to consider the case when some geometric approximation $f\colon\mc{G}\ra T$ admits an indecomposable component $U$. Up to replacing $\mc{G}$ and $U$, Lemma~\ref{stable components} allows us to assume that all arcs of $U$ have the same stabiliser $H$, which is also the stabiliser of a stable arc of $T$. In addition, $U$ is not a line, otherwise $f(U)\cu T$ would be a line with a non-discrete action by its stabiliser. Thus, we can apply Proposition~\ref{exotic/surface prop}, which shows that $G$ splits over a centraliser and admits a fold, partial conjugation or twist with infinite order in $\out(G)$. Twists only occur in the ``surface case'' and they satisfy the requirements of Theorem~\ref{main cmp}.

This concludes the proof.
\end{proof}

\begin{proof}[Proof of Theorem~\ref{main general}]
Let $G$ be a special group with $\out(G)$ infinite. Any infinite sequence in $\out(G)$ yields a $G$--tree $G\acts T_{\om}$ as in Subsection~\ref{limit subsec}. Let $T\cu T_{\om}$ be the $G$--minimal subtree.

Suppose first that one of the following happens:
\begin{enumerate}
\item[(i)] a line $\alpha\cu T$ is acted upon non-discretely by its $G$--stabiliser;
\item[(ii)] the $G$--stabiliser of an arc $\beta\cu T$ is not $G$--semi-parabolic;
\item[(iii)] $T$ is a line.
\end{enumerate} 
In Case~(ii), $\beta$ necessarily falls into Option~(2) of Proposition~\ref{arc-stab prop} and we denote by $\alpha$ the line that it provides. In Case~(iii), we simply set $\alpha:=T$. In each of the three cases, we obtain a line $\alpha$ that is acted upon nontrivially by its $G$--stabiliser.

Now, we apply Proposition~\ref{Z vertex group prop} to the line $\alpha$. Observe that, we can assume that we are in Option~(2) of Proposition~\ref{Z vertex group prop}. Indeed, this is clear in Case~(iii), since $\alpha=T$. Regarding instead Cases~(i) and~(ii), we clearly cannot fall into Option~(1) of Proposition~\ref{Z vertex group prop}, whereas Option~(3) can be handled using Proposition~\ref{exotic/surface prop}, resulting in a DLS automorphism as in Theorem~\ref{main cmp}.

In conclusion, suppose that we have a line $\alpha\cu T$ falling into Option~(2) of Proposition~\ref{Z vertex group prop}. Let $Z$ be the stabiliser of $\alpha$ and let $\rho\colon Z\ra\R$ be the homomorphism giving translation lengths, which is nontrivial. Since $Z/\ker\rho$ is free abelian, there exists a homomorphism $\overline\rho\colon Z\ra\Z$ such that $\ker\rho\leq\ker\overline\rho$. Note that $\overline\rho$ gives an HNN splitting of $Z$ over $\ker\overline\rho$, with stable letter in $Z$. Appealing to Lemma~\ref{refining splittings}, this results in an HNN splitting of $G$ over $\ker\overline\rho$ with the same stable letter. 

By Proposition~\ref{important addenda new}(a), we have $Z=Z_G(x)$ for some $x\in G$. By Proposition~\ref{important addenda new}(d) and Lemma~\ref{line centre lem}, the centre of $\ker\overline\rho$ is nontrivial. By Remark~\ref{kernel centre rmk}, the centre of $\ker\overline\rho$ is contained in the centre of $Z$, so it commutes with the stable letter of the HNN splitting of $G$. Any element in the centre of $\ker\overline\rho$ gives a twist with infinite order in $\out(G)$ by Lemma~\ref{non-inner from HNN}. In conclusion, we have constructed an automorphism as in Type~(3) in the statement of Theorem~\ref{main general}.

By the above discussion, we can assume in the rest of the proof that Cases~(i)--(iii) do not occur, i.e.\ that $T$ is not a line, that all arc-stabilisers are $G$--semi-parabolic, and that every line in $T$ is acted upon discretely by its stabiliser. Note however that not all arc-stabilisers might be centralisers.

Now, we can conclude via Proposition~\ref{exotic/surface prop} and Proposition~\ref{simplicial prop} as in the proof of Theorem~\ref{main cmp}. If Option~(3) of Proposition~\ref{simplicial prop} presents itself, then we obtain an automorphism as in Type~(3) of Theorem~\ref{main general} as above. In all other cases, we obtain a DLS automorphism that has infinite order in $\out(G)$ and is coarse-median preserving by Theorem~\ref{CMP GDT}. Thus, $\out_{\rm cmp}(G)$ is infinite and, appealing to Theorem~\ref{main cmp}, we obtain a DLS automorphism of the required form.

This concludes the proof.
\end{proof}

\begin{proof}[Proof of Corollary~\ref{infinite out finite cmp}]
Let $\mscr{H}$ be the collection of subgroups of $G$ of the form $Z_G(x)$ with $x\in G$. Suppose that, for every $\mc{H}\in\mscr{H}$, the $\out(G)$--orbit of $\mc{H}$ is infinite. We will show that $\out_{\rm cmp}(G)$ is infinite.

By Lemma~\ref{Neumann's lemma}, there exist automorphisms $\phi_n\in\out(G)$ such that, for every $\mc{H}\in\mscr{H}$, the sequence $\phi_n(\mc{H})$ eventually consists of pairwise distinct classes. Let us run the proof of Theorem~\ref{main cmp} for this sequence of automorphisms. The only place where we used that the automorphisms were coarse-median preserving was when applying Proposition~\ref{important addenda new}(c1), which we can now replace by Proposition~\ref{important addenda new}(c2). Thus, we obtain a coarse-median preserving DLS automorphism with infinite order in $\out(G)$, as required.
\end{proof}

\section{Coarse-median preserving DLS automorphisms.}\label{CMP GDT sect}

The goal of this section is the following result. The UCP condition is introduced in Subsection~\ref{UCP subsec}.

\begin{thm}\label{CMP GDT 2}
Let $G\acts X$ be a proper, cocompact, non-transverse action on a $\CAT$ cube complex. Suppose that $G$ splits as $G=A\ast_CB$ or $G=A\ast_C$, where $C$ is convex-cocompact, satisfies the UCP condition in $X$ and does not have any nontrivial finite normal subgroups. Then:
\begin{enumerate}
\item all partial conjugations and folds determined by this splitting are coarse-median preserving;
\item if $z\in Z_C(C)$ is such that $\langle z\rangle$ is convex-cocompact in $X$ and $Z_G(c)$ is contained in a conjugate of $A$ for every $c\in C$ such that $\langle c\rangle\cap\langle z\rangle\neq\{1\}$, then the twist determined by $z$ is coarse-median preserving;
\item more generally, if, for every infinite-order element $c\in C$ commuting with a finite-index subgroup of $C$, the centraliser $Z_G(c)$ is contained in a conjugate of $A$, then all transvections determined by the splitting $G=A\ast_C$ are coarse-median preserving.
\end{enumerate}
\end{thm}

The proof of Theorem~\ref{CMP GDT 2} is simpler if the cube complex $X$ contains a collection of pairwise disjoint hyperplanes such that their dual tree is precisely the Bass--Serre tree $T$ of the splitting of $G$. This is the situation that we consider in Subsection~\ref{earthquake cmp sect}.

The previous subsections reduce the proof to this setting. The main idea is to ``inflate'' a convex, $C$--invariant subcomplex of $X$ to a hyperplane. This is achieved by considering the $G$--action on the product $X\x T$, and recovering cocompactness by restricting to its ``cubical Guirardel core''. This is a generalisation of the Guirardel core of a product of two trees \cite{Guirardel-core} that we introduce in Subsection~\ref{Guirardel core sect}. Our construction can also be viewed as a broad generalisation of the idea of \emph{Salvetti blowups} from \cite{CSV}.

\subsection{Uniformly cocompact projections.}\label{UCP subsec}

Let $G\acts X$ be a proper cocompact action on a $\CAT$ cube complex. We denote Hausdorff distances by $d_{\rm Haus}(\cdot,\cdot)$.

\begin{lem}\label{Haus proj}
Consider a subgroup $H\leq G$ and $H$--invariant, convex subcomplexes $Z,W\cu X$ with $d_{\rm Haus}(Z,W)=D$. Then $d_{\rm Haus}(\pi_Z(gZ),\pi_W(gW))\leq 3D$ for all $g\in G$. 
\end{lem}
\begin{proof}
Since $\pi_Z$ is $1$--Lipschitz, we have $d_{\rm Haus}(\pi_Z(gZ),\pi_Z(gW))\leq D$. In addition, for every $x\in X$:
\[\mscr{W}(\pi_Z(x)|\pi_W(x))=\mscr{W}(x,\pi_Z(x)|\pi_W(x))\cup\mscr{W}(\pi_Z(x)|x,\pi_W(x))\cu\mscr{W}(\pi_Z(x)|W)\cup\mscr{W}(Z|\pi_W(x)),\]
hence $d(\pi_Z(x),\pi_W(x))\leq 2D$. It follows that $d_{\rm Haus}(\pi_Z(gZ),\pi_W(gW))\leq 3D$.
\end{proof}

\begin{defn}[Uniformly Cocompact Projections]\label{UCP defn}
Let $H\leq G$ be convex-cocompact in $X$. Let $Z\cu X$ be any $H$--invariant, $H$--cocompact convex subcomplex. We say that $H$ satisfies the \emph{UCP condition in $X$} if there exists $N\geq 1$ such that, for every $g\in G$, the action $H\cap gHg^{-1}\acts\pi_Z(gZ)$ has at most $N$ orbits of vertices.
\end{defn}

Since $X$ is uniformly locally finite, Lemma~\ref{Haus proj} shows that this property only depends on the subgroup $H\leq G$ and the action $G\acts X$, and not on the specific choice of $Z$.

Our interest in the UCP condition is exclusively related to Lemma~\ref{descending projections} below. In Lemma~\ref{special UCP}, we will show that convex-cocompact subgroups of special groups satisfy the UCP condition in any cospecial cubulation. However, even when restricting to special groups $G$, we will need this property within non-cospecial cubulations of $G$ (see Subsection~\ref{earthquake cmp sect}).

Recall that a sequence of hyperplanes $\mf{u}_1,\dots,\mf{u}_k$ is said to be a \emph{chain of hyperplanes} if, for each $2\leq i\leq k-1$, the hyperplane $\mf{u}_i$ separates $\mf{u}_{i-1}$ from $\mf{u}_{i+1}$.

\begin{lem}\label{descending projections}
Let $\mf{w}\in\mscr{W}(X)$ be a hyperplane. Let $H$ be its $G$--stabiliser and suppose that $H$ satisfies the UCP condition and acts non-transversely on $X$. Then there exists $K\geq 1$ such that, for every chain $\mf{w},g_1\mf{w},\dots,g_k\mf{w}$ of $G$--translates of $\mf{w}$ such that $\pi_{\mf{w}}(g_1\mf{w})\supsetneq\dots\supsetneq\pi_{\mf{w}}(g_k\mf{w})$, we have $k\leq K$.
\end{lem}
\begin{proof}
Since $H$ satisfies the UCP condition, there exists $N\geq 1$ such that, for every $g\in G$, the subgroup $H\cap gHg^{-1}$ acts on $\pi_{\mf{w}}(g\mf{w})$ with at most $N$ orbits of vertices. As in the claim during the proof of Lemma~\ref{almost normalisers}, there exists $N'\geq 1$ such that, for every $p\in X$, there are at most $N'$ subgroups of $G$ that act with at most $N$ orbits of vertices on a convex subcomplex of $X$ containing $p$. We will show that $k\leq N\cdot N'$.

Since the hyperplanes $\mf{w},g_1\mf{w},\dots,g_k\mf{w}$ form a chain and $H$ acts non-transversely on $X$, we have $g_1Hg_1^{-1}\cap H\geq\dots\geq g_kHg_k^{-1}\cap H$. By the previous paragraph, there are at most $N'$ distinct subgroups of $G$ among the $g_iHg_i^{-1}\cap H$. Thus, it suffices to assume that $g_iHg_i^{-1}\cap H$ is constant and show that $k\leq N$. The latter follows from the observation that, in this situation, the number of orbits in $\pi_{\mf{w}}(g_i\mf{w})$ is bounded above by $N$ and must strictly decrease as $i$ increases.
\end{proof}

The following implies that convex-cocompact subgroups of special groups satisfy the UCP condition in any cospecial cubulation.

\begin{lem}\label{special UCP}
Convex-cocompact subgroups of $\A_{\G}$ satisfy the UCP condition in $\X_{\G}$.
\end{lem}
\begin{proof}
Let $H\leq\A_{\G}$ be a convex-cocompact subgroup. Let $Z\cu\X_{\G}$ be an $H$--invariant, $H$--cocompact, convex subcomplex. Let $Z_0\cu Z$ be a finite subset meeting every $H$--orbit. 

Let $\mscr{P}$ be the set of parabolic subgroups of $\A_{\G}$ whose parabolic stratum meets $Z_0$. Note that $\mscr{P}$ is finite. If $P\in\mscr{P}$, recall that $\mc{W}_1(P)\cu\mscr{W}(\X_{\G})$ are the hyperplanes skewered by elements of $P$. 

\smallskip
{\bf Claim:} \emph{if $g\in\A_{\G}$ and $\pi_Z(gZ)\cap Z_0\neq\emptyset$, then there exist $\overline g\in\A_{\G}$ and $P\in\mscr{P}$ such that:}
\begin{enumerate}
\item \emph{$\overline gZ\cap Z_0\neq\emptyset$ and $\mscr{W}(\pi_Z(gZ))=\mscr{W}(Z)\cap\mscr{W}(\overline gZ)\cap\mc{W}_1(P)$;}
\item \emph{$H\cap gHg^{-1}=H\cap\overline gH\overline g^{-1}\cap P$.}
\end{enumerate}

\smallskip\noindent
\emph{Proof of claim.}
If $gZ\cap Z\neq\emptyset$, then $gZ$ meets $Z_0$, and we can take $\overline g=g$ and $P=\A_{\G}$. 

Otherwise, $\mscr{W}(Z|gZ)$ is nonempty and we define $P\leq\A_{\G}$ as the largest parabolic subgroup fixing $\mscr{W}(Z|gZ)$ pointwise. Since $\pi_Z(gZ)$ meets $Z_0$, we have $P\in\mscr{P}$. Note that $\mc{W}_1(P)\cu\mscr{W}(\X_{\G})$ coincides with the set of all hyperplanes transverse to $\mscr{W}(Z|gZ)$.

Choose a pair of gates $z\in Z$, $z'\in gZ$ for $Z$ and $gZ$, with $z\in Z_0$. Choose $g'\in\A_{\G}$ with $g'z'=z$. Observing that $H\cap gHg^{-1}\leq P$ and that $g'$ commutes with $P$ (e.g.\ by Lemma~\ref{commutation criterion}), we deduce that:
\[H\cap gHg^{-1}=H\cap gHg^{-1}\cap P=H\cap (g'g)H(g'g)^{-1}\cap P.\]
Setting $\overline g:=g'g$, condition~(2) is satisfied. We also have $z\in\overline gZ\cap Z_0$, hence $\overline gZ\cap Z_0\neq\emptyset$.

Since $g'$ fixes $\mc{W}_1(P)$ pointwise and $\overline g=g'g$, we have $\mscr{W}(\overline gZ)\cap\mc{W}_1(P)=\mscr{W}(gZ)\cap\mc{W}_1(P)$. Recalling that $\mc{W}_1(P)$ is the set of hyperplanes transverse to $\mscr{W}(Z|gZ)$, we obtain:
\[\mscr{W}(\pi_Z(gZ))=\mscr{W}(Z)\cap\mscr{W}(gZ)=\mscr{W}(Z)\cap\mscr{W}(gZ)\cap\mc{W}_1(P)=\mscr{W}(Z)\cap\mscr{W}(\overline gZ)\cap\mc{W}_1(P),\]
which completes the proof of the claim.
\hfill$\blacksquare$

\smallskip
Since the action $H\acts Z$ is cocompact, each point of $\X_{\G}$ lies in only finitely many pairwise-distinct $\A_{\G}$--translates of $Z$ (see Claim~1 in the proof of Lemma~\ref{cc intersection}). Moreover, $H$ has finite index in the $\A_{\G}$--stabiliser of $Z$. It follows that the set of elements $\overline g\in\A_{\G}$ such that $\overline gZ\cap Z_0\neq\emptyset$ is a finite union of left cosets of $H$.

Note that, in order to prove the lemma, it suffices to show that the actions $H\cap gHg^{-1}\acts\pi_Z(gZ)$ are uniformly cocompact when $\pi_Z(gZ)\cap Z_0\neq\emptyset$. By the claim and the previous paragraph, there are only finitely many options for such subgroups $H\cap gHg^{-1}$ and sets $\pi_Z(gZ)$. So it suffices to show that each action $H\cap gHg^{-1}\acts\pi_Z(gZ)$ is cocompact, which follows from Lemma~\ref{cc intersection}.
\end{proof}

\begin{proof}[Proof of Theorem~\ref{CMP GDT}]
Since $G$ is special, it admits a cospecial cubulation $G\acts X$. By Lemma~\ref{special UCP}, we can apply Theorem~\ref{CMP GDT 2} to this action. To reconcile the differences in parts~(2) and~(3) between Theorem~\ref{CMP GDT} and Theorem~\ref{CMP GDT 2}, it suffices to recall that $G$ is torsion-free and that elements of $G$ with commuting powers must themselves commute.
\end{proof}

\subsection{Panel collapse.}

Here, we record the following special case of the \emph{panel collapse} procedure of Hagen and Touikan \cite{Hagen-Touikan}, restricting ourselves to non-transverse actions. Under this assumption, panel collapse --- normally a fairly violent procedure --- does not alter the coarse median structure.

\begin{prop}\label{cmp panel collapse prop}
Let $G\acts X$ be a cocompact, non-transverse action on a $\CAT$ cube complex without inversions. Suppose that there exists a halfspace $\mf{h}\in\mscr{H}(X)$ that is minimal (under inclusion) among halfspaces transverse to a hyperplane $\mf{w}\in\mscr{W}(X)$. Then there exists $Y$ such that:
\begin{enumerate}
\item $Y$ is a $G$--invariant subcomplex of $X$ with $Y^{(0)}=X^{(0)}$;
\item $Y$ is a $\CAT$ cube complex (though not convex, nor a median subalgebra in $X$);
\item $Y$ has strictly fewer $G$--orbits of edges than $X$;
\item the identity map $X^{(0)}\ra Y^{(0)}$ is coarse-median preserving;
\item the intersection $\mf{w}\cap Y$ is nonempty and connected.
\end{enumerate}
\end{prop}
\begin{proof}
Say that an edge $e\cu X$ is \emph{bad} if there exists $g\in G$ such that $e$ is contained in $g\mf{h}$ and crosses $g\mf{w}$. Let $\mc{G}\cu X^{(1)}$ be the subgraph obtained by removing interiors of bad edges. Since the action $G\acts X$ is non-transverse and without inversions, then, for every cube $c\cu X$ with $\dim c\geq 2$, the intersection between $\mc{G}$ and the $1$--skeleton of $c$ is connected.
% thus, this the case of panel collapse where, in the Hagen-Touikan terminology, $\mc{D}(c)$ is connected for every cube $c$

We define $Y$ as the full subcomplex of $X$ with $Y^{(1)}=\mc{G}$. Parts~(1),~(3) and~(5) are immediate. Part~(2) is proved in \cite{Hagen-Touikan} (of which we are considering the simplest possible case, since $\mc{G}$ has connected intersection with $1$--skeletons of cubes of $X$).

We are left to prove~(4). We will speak of $X$--geodesics and $Y$--geodesics, depending on which of the two metrics we are considering. Let $m_X$ and $m_Y$ be the median operators on $X^{(0)}=Y^{(0)}$ induced by $X$ and $Y$, respectively. If $\alpha$ and $\beta$ are paths in $X^{(1)}$ (possibly containing edges outside $Y$), we write $\delta_Y(\alpha,\beta)$ for the Hausdorff distance in the metric of $Y$ between the two intersections $\alpha\cap X^{(0)}$ and $\beta\cap X^{(0)}$. 

\smallskip
{\bf Claim:} \emph{for every $X$--geodesic $\alpha\cu X^{(1)}$, there exists a $Y$--geodesic $\beta\cu Y^{(1)}$ with the same endpoints and with $\delta_Y(\alpha,\beta)\leq 2$.}

\smallskip
Assuming the claim, we prove part~(4). Consider three points $x,y,z\in X^{(0)}$. By the claim, the point $m_X(x,y,z)$ is at distance $\leq 2$ in $Y$ from a $Y$--geodesic between any two of these three points. Thus, at most $2$ hyperplanes of $Y$ separate $m_X(x,y,z)$ from any two among $x,y,z$. Hence at most $6$ hyperplanes of $Y$ separate $m_X(x,y,z)$ and $m_Y(x,y,z)$, which shows part~(4).

Now, in order to prove the claim, let $\alpha\cu X^{(1)}$ be an $X$--geodesic. Consider a bad edge $e\cu\alpha$, and let $g\in G$ be an element such that $e$ crosses $g\mf{w}$ and is contained in $g\mf{h}$. We say that $e$ is \emph{avoidable} if $g\mf{h}$ contains exactly one of the endpoints of $\alpha$.

\smallskip
{\bf Sub-claim:} \emph{there exists an $X$--geodesic $\alpha'\cu X^{(1)}$, with the same endpoints as $\alpha$, such that $\delta_Y(\alpha,\alpha')\leq 1$ and $\alpha'$ contains no avoidable bad edges.}

\smallskip\noindent
\emph{Proof of sub-claim.}
Let $e\cu\alpha$ be an avoidable bad edge, crossing a hyperplane $g\mf{w}$ and contained in a halfspace $g\mf{h}$. Let $\alpha_e\cu\alpha$ be the subsegment lying in the carrier of the hyperplane bounding $g\mf{h}$. Let $\alpha_e'$ be the $X$--geodesic, with the same endpoints as $\alpha_e$, that is entirely contained in $g\mf{h}^*$ except for its initial or terminal edge. Then $\delta_Y(\alpha_e,\alpha_e')=1$ and, by minimality of $\mf{h}$, no edge of $\alpha_e'$ is bad.

Replacing the segment $\alpha_e\cu\alpha$ with $\alpha_e'$, then repeating the procedure for the two geodesics forming $\alpha\setminus\alpha_e$ yields the sub-claim.
\hfill$\blacksquare$

\smallskip\noindent
\emph{Proof of claim.}
Now, let $e_1,\dots,e_k$ be the bad edges on $\alpha'$, in order of appearance along it. Let $g_i\in G$ be elements such that $e_i$ crosses $g_i\mf{w}$ and is contained in $g_i\mf{h}$. Since none of the $e_i$ is avoidable, we must have $\alpha'\cu g_i\mf{h}$ for every $i$. 

We define a new path $\beta\cu X^{(1)}$ as follows. Let $s$ be the highest index with $g_s\mf{h}=g_1\mf{h}$. Let $\g_1\cu\alpha'$ be the segment starting with $e_1$ and ending with $e_s$. We replace $\g_1$ with the path that immediately crosses into $g_1\mf{h}^*$, then crosses the same hyperplanes as $\g_1$, and finally crosses back into $g_1\mf{h}=g_s\mf{h}$. We deal in a similar way with all other halfspaces $g_i\mf{h}$ in order to avoid all bad edges on $\alpha'$.

Now, the path $\beta$ contains no bad edges. It is not an $X$--geodesic, but it is straightforward to check that it is a $Y$--geodesic. In addition, $\delta_Y(\alpha',\beta)\leq 1$, hence $\delta_Y(\alpha,\beta)\leq 2$. 
\hfill$\blacksquare$

\smallskip
This concludes the proof of the proposition.
\end{proof}

Note that, by part~(5), the intersection $\mf{w}\cap Y$ is a hyperplane of $Y$. We can only apply Proposition~\ref{cmp panel collapse prop} a finite number of times because of part~(3). Eventually, we obtain:

\begin{cor}\label{cmp panel collapse cor}
Let $G\acts X$ be a cocompact, non-transverse action on a $\CAT$ cube complex without inversions. Then there exists $Y\cu X$ such that:
\begin{enumerate}
\item $Y$ is a $G$--invariant subcomplex with $Y^{(0)}=X^{(0)}$;
% we're never asking for essentiality of the cube complex, so we won't shrink the vertex set
\item $Y$ is a $\CAT$ cube complex (though not convex, nor a median subalgebra in $X$);
\item the action $G\acts Y$ is hyperplane-essential;
\item the identity map $X^{(0)}\ra Y^{(0)}$ is coarse-median preserving.
\end{enumerate}
% clearly, properness of the action is preserved (subcomplex)
\end{cor}

In certain situations, it is convenient to prioritise connectedness of a certain hyperplane over essentiality of all other hyperplanes. This can be similarly achieved with a repeated application of Proposition~\ref{cmp panel collapse prop}.

\begin{cor}\label{cmp panel collapse cor 2}
Given $\mf{w}\in\mscr{W}(X)$, property~(3) in Corollary~\ref{cmp panel collapse cor} can be replaced with:
\begin{enumerate}
\item[(3')] the intersection $\mf{w}\cap Y$ is connected and the action $G_{\mf{w}}\acts\mf{w}\cap Y$ is essential.
\end{enumerate}
\end{cor}

\subsection{Cubical Guirardel cores.}\label{Guirardel core sect}

Guirardel's notion of \emph{core} for a product of actions on $\R$--trees $G\acts T_1\x T_2$ \cite{Guirardel-core} can be rephrased purely in median-algebra terms: it is (closely related to) the median subalgebra of $T_1\x T_2$ generated by a $G$--orbit. As such, this notion can be naturally extended to products of $\CAT$ cube complexes.

In this subsection, we are concerned with cocompactness of this notion of core. The main result is Proposition~\ref{cofinite subalgebra}, which we will only require in the special case of Corollary~\ref{core cor}.

We remark that cocompactness of the core can be achieved more generally, but one must allow the core to be a non--$\CAT$ cube complex, and thus abandon the setting of median algebras. This insight is explored in \cite{Hagen-Wilton}.

\begin{prop}\label{cofinite subalgebra}
Let $G$ act on $\CAT$ cube complexes $X$ and $Y$. Suppose that $G\acts X$ is proper and cocompact, while $G\acts Y$ is essential and has only finitely many orbits of hyperplanes. Then the following are equivalent:
\begin{enumerate}
\item every $G$--orbit in the $0$--skeleton of $X\x Y$ generates a $G$--cofinite median subalgebra;
\item some $G$--orbit in the $0$--skeleton of $X\x Y$ generates a $G$--cofinite median subalgebra;
\item the stabiliser of every hyperplane of $Y$ is convex-cocompact in $X$.
\end{enumerate}
\end{prop}

Before proving the proposition, we need to record a couple of observations.

\begin{defn}
Consider a group $\G$, a subgroup $H\leq \G$, and the action $H\acts \G$ by left multiplication. An \emph{$H$--AIS (Almost Invariant Set)} is a subset $A\cu \G$ such that:
\begin{enumerate}
\item $A$ is $H$--invariant;
\item both $A$ and its complement $\G\setminus A$ contain infinitely many $H$--orbits;
\item for every $g\in \G$, the symmetric difference $Ag\triangle A$ is $H$--cofinite.
\end{enumerate}
If $A$ is an $H$--AIS, then the set $A^*:=\G\setminus A$ is another $H$--AIS.
\end{defn}

\begin{lem}\label{realising AIS}
Let $G\acts X$ be a proper cocompact action on a $\CAT$ cube complex. Let $H\leq G$ be a convex-cocompact subgroup. Let $A\cu G$ be an $H$--AIS. Then, for every vertex $x_0\in X$, there exists a partition $X=C_-\sqcup C_0\sqcup C_+$ such that:
\begin{enumerate}
\item $C_0$ is an $H$--invariant convex subcomplex of $X$ on which $H$ acts cocompactly;
\item $C_-$ and $C_+$ are $H$--invariant unions of connected components of $X\setminus C_0$; 
\item $A\cdot x_0\cu C_0\cup C_+$ and $A^*\cdot x_0\cu C_0\cup C_-$.
\end{enumerate}
\end{lem}
\begin{proof}
Choose $R\geq 0$ such that $G\cdot x_0$ is $R$--dense in $X$. Observing that $G$ is finitely generated, we can fix a word metric $(G,d)$. Choose $r\geq 0$ so that $d(g,h)\leq r$ for all $g,h\in G$ with $d(gx_0,hx_0)\leq 2R$. Let $\Delta\cu G$ be the intersection between $A^*\cu G$ and the $r$--neighbourhood of $A$ in $G$. Since $A$ is an $H$--AIS, $\Delta$ is $H$--cofinite. 

Since $H$ is convex-cocompact, there exists an $H$--invariant convex subcomplex $K\cu X$ on which $H$ acts cocompactly. Since $\Delta$ is $H$--cofinite, there exists $L\geq 0$ such that the neighbourhood $N_L(K)$ contains the $R$--neighbourhood of $\Delta\cdot x_0$. We define $C_0$ as the convex hull of $N_L(K)$. This is clearly an $H$--invariant convex subcomplex of $X$. By \cite[Lemma~6.4]{Bow-cm}, $C_0$ is at finite Hausdorff distance from $K$, so the action $H\acts C_0$ is again cocompact.

Now, define $C_+$ as the union of the connected components of $X\setminus C_0$ that intersect $A\cdot x_0$. Since $A\cdot x_0$ is $H$--invariant, so is $C_+$. The set $C_-$ is defined analogously using $A^*\cdot x_0$. We are only left to show that a single connected component of $X\setminus C_0$ cannot intersect both $A\cdot x_0$ and $A^*\cdot x_0$.

Suppose for the sake of contradiction that there exists a path $\alpha\cu X\setminus C_0$ joining a point of $A\cdot x_0$ to a point of $A^*\cdot x_0$. Since every point of $\alpha$ is at distance $\leq R$ from the orbit $G\cdot x_0$, there exist a point $y\in\alpha$ and elements $a\in A$, $a'\in A^*$ with $d(y,ax_0)\leq R$ and $d(y,a'x_0)\leq R$. In particular, $d(ax_0,a'x_0)\leq 2R$, hence $d(a,a')\leq r$. It follows that $a'\in\Delta$, so $d(y,\Delta\cdot x_0)\leq R$. Since $y\in X\setminus C_0$, this contradicts the fact that $C_0$ contains the $R$--neighbourhood of $\Delta\cdot x_0$.
\end{proof}

\begin{rmk}
Let $G\acts Y$ be an action on a $\CAT$ cube complex. Consider a basepoint $y_0\in Y$, a halfspace $\mf{h}$ bounded by a hyperplane skewered by an element of $G$, and the subgroup $H\leq G$ stabilising $\mf{h}$. Then the set $\{g\in G\mid gy_0\in\mf{h}\}$ is an $H$--AIS.
\end{rmk}

\begin{proof}[Proof of Proposition~\ref{cofinite subalgebra}]
It is clear that (1)$\Ra$(2). Let us show that (2)$\Ra$(3). Assuming (2), let $M$ be a $G$--invariant, $G$--cofinite median subalgebra of the $0$--skeleton of $X\x Y$. Every hyperplane $\mf{w}\in\mscr{W}(Y)$ gives a hyperplane of $X\x Y$ skewered by some element of $G$, hence a wall $\mf{w}'\in\mscr{W}(M)$. By Chepoi--Roller duality, $M$ is the $0$--skeleton of a $\CAT$ cube complex $Z$ with a cocompact $G$--action. The stabiliser $G_{\mf{w}}$ acts cocompactly on the carrier of $\mf{w}'$ in $Z$, hence it is convex-cocompact in $Z$. The argument in the proof of Corollary~\ref{UCP and Guirardel}(1) below implies that $G_{\mf{w}}$ is convex-cocompact in $X$, as required.

We now prove the implication (3)$\Ra$(1), which is the main content of the proposition. Consider a vertex $p=(x_0,y_0)$ and let $M\cu X\x Y$ be the median algebra generated by the orbit $G\cdot p$.

By Lemma~\ref{realising AIS}, 
% here is where we use essentiality of Y
to every hyperplane $\mf{w}\in\mscr{W}(Y)$ bounding halfspaces $\mf{h}$ and $\mf{h}^*$, we can associate a partition $X=C(\mf{h}^*)\sqcup C(\mf{w})\sqcup C(\mf{h})$ with the following properties:
\begin{itemize}
\item $C(\mf{w})$ is a $G_{\mf{h}}$--invariant, $G_{\mf{h}}$--cocompact, convex subcomplex of $X$;
\item $C(\mf{h})$ and $C(\mf{h}^*)$ are $G_{\mf{h}}$--invariant unions of connected components of $X\setminus C(\mf{w})$;
\item if $g\in G$ and $gy_0\in\mf{h}$, then $gx_0\in C(\mf{w})\cup C(\mf{h})$; if $gy_0\in\mf{h}^*$, then $gx_0\in C(\mf{w})\cup C(\mf{h}^*)$;
\item if $g\in G$, we have $C(g\mf{w})=gC(\mf{w})$ and $C(g\mf{h})=gC(\mf{h})$.
\end{itemize}
These properties imply the following.

\smallskip
{\bf Claim:} \emph{Consider hyperplanes $\mf{v}\in\mscr{W}(X)\sqcup\mscr{W}(Y)$ and $\mf{w}\in\mscr{W}(Y)$ inducing transverse walls of $M$. Then $\mf{v}\cap C(\mf{w})\neq\emptyset$ if $\mf{v}\in\mscr{W}(X)$, and $C(\mf{v})\cap C(\mf{w})\neq\emptyset$ if $\mf{v}\in\mscr{W}(Y)$.}

\smallskip\noindent
\emph{Proof of claim.}
We only consider the situation with $\mf{v}\in\mscr{W}(Y)$, as the argument for the case when $\mf{v}\in\mscr{W}(X)$ is entirely analogous.

Let $\mf{v}^{\pm},\mf{w}^{\pm}\in\mscr{H}(Y)$ be the halfspaces bounded by $\mf{v}$ and $\mf{w}$. Suppose for the sake of contradiction that $C(\mf{v})$ and $C(\mf{w})$ are disjoint. Then $C(\mf{w})$, being connected, is contained in a single connected component of $X\setminus C(\mf{v})$. Without loss of generality, $C(\mf{w})\cu C(\mf{v}^+)$. It follows that the connected set $C(\mf{v}^-)\cup C(\mf{v})$ is disjoint from $C(\mf{w})$, hence contained in a single connected component of $X\setminus C(\mf{w})$. Thus, again without loss of generality, we have $C(\mf{v}^-)\cup C(\mf{v})\cu C(\mf{w}^+)$, hence the sets $C(\mf{v})\cup C(\mf{v}^-)$ and $C(\mf{w})\cup C(\mf{w}^-)$ are disjoint. 

However, since $\mf{v}$ and $\mf{w}$ induce transverse walls of $M=\langle G\cdot p\rangle$, there exists $g\in G$ such that $gp\in\mf{v}^-\cap\mf{w}^-$. Equivalently, $gy_0\in\mf{v}^-\cap\mf{w}^-$, hence $gx_0\in (C(\mf{v})\cup C(\mf{v}^-))\cap(C(\mf{w})\cup C(\mf{w}^-))$.
\hfill$\blacksquare$

\smallskip
Now, let $\mscr{T}(M)$ be the set of tuples of pairwise-transverse walls of $M$. Consider an element of $\mscr{T}(M)$, say induced by tuples of hyperplanes $\mf{u}_1,\dots,\mf{u}_k\in\mscr{W}(X)$ and $\mf{v}_1,\dots,\mf{v}_h\in\mscr{W}(Y)$. By the claim, the collection of all carriers of the $\mf{u}_i$ and all sets $C(\mf{v}_i)$ consists of pairwise intersecting convex subsets of $X$. By Helly's lemma, the intersection of these convex sets is nonempty.

Fix a compact fundamental domain $K\cu X$ for the $G$--action. By the previous paragraph, there exists $g\in G$ such that the carrier of each $g\mf{u}_i$ and every set $gC(\mf{v}_i)$ meets $K$. Note that only finitely many hyperplanes of $X$ have carrier meeting the compact set $K$. Similarly, only finitely many hyperplanes $\mf{v}\in\mscr{W}(Y)$ satisfy $C(\mf{v})\cap K\neq\emptyset$. This follows by combining the fact that the action $G\acts\mscr{W}(Y)$ is cofinite with Claim~1 in the proof of Lemma~\ref{cc intersection}.

The above discussion shows that the action $G\acts\mscr{T}(M)$ has only finitely many orbits. By Chepoi--Roller duality, $M$ is the $0$--skeleton of a $\CAT$ cube complex with finitely many $G$--orbits of maximal cubes. This shows that the action $G\acts M$ is cofinite, as required.
\end{proof}

For the next result, note that we can naturally extend the UCP property (Definition~\ref{UCP defn}) to actions on discrete median algebras $M$. This is entirely equivalent to the UCP property for the action on the $\CAT$ cube complex associated with $M$ by Chepoi--Roller duality.

We will also speak of \emph{carriers} and \emph{cubes} in $M$, always referring to (vertex sets of) carriers and cubes in the associated $\CAT$ cube complex.

\begin{cor}\label{UCP and Guirardel}
Let $G$ act on $X$ and $Y$ satisfying both the assumptions and the equivalent conditions in Proposition~\ref{cofinite subalgebra}. Let $M\cu X\x Y$ be the median subalgebra generated by a $G$--orbit.
\begin{enumerate}
\item A subgroup $H\leq G$ is convex-cocompact in $X$ if and only if $H$ is convex-cocompact in $M$.
\item If, in addition, $H$ satisfies the UCP condition in $X$, then it also satisfies it in $M$.
\end{enumerate}
\end{cor}
\begin{proof}
Let $p_X\colon M\ra X$ be the restriction of the factor projection. Since $G$ acts properly and cocompactly on both $M$ and $X$, and $p_X$ is a $1$--Lipschitz median morphism, we see that $M$ and $X$ induce the same coarse median structure on $G$. Along with Remark~\ref{cc vs qc}, this implies part~(1).

Let us prove part~(2). If $\mf{w}\in\mscr{W}(M)$, we denote by $C(\mf{w})\cu X$ the convex hull of the image under $p_X$ of the carrier of $\mf{w}$ in $M$. 
% note that p_X does not take convex sets to convex sets in general
Since $G_{\mf{w}}$ is convex-cocompact in $X$ by part~(1), and $C(\mf{w})$ is the convex hull of a $G_{\mf{w}}$--cocompact subset of $X$, we conclude that the action $G_{\mf{w}}\acts C(\mf{w})$ is cocompact. Since $G\acts\mscr{W}(M)$ is cofinite, there exists $N\geq 1$ such that every point of $X$ lies in $C(\mf{w})$ for at most $N$ walls $\mf{w}\in\mscr{W}(M)$ (see Claim~1 in the proof of Lemma~\ref{cc intersection}).

Let $Z\cu M$ be an $H$--invariant, $H$--cofinite, convex subset. As above, there exists an $H$--cocompact convex subcomplex $C(H)\cu X$ containing the projection $p_X(Z)$. Given $g\in G$, we denote by $\Pi_g$ the gate-projection of $gC(H)$ to $C(H)$. Since $H$ is UCP in $X$, there exists $N'\geq 1$ such that, for every $g\in G$, the group $H\cap gHg^{-1}$ acts on $\Pi_g$ with at most $N'$ orbits. Let $P_g\cu\Pi_g$ be a subset of cardinality $\leq N'$ meeting all these orbits.

Let $\mscr{T}(g)$ be the set of tuples of pairwise-transverse walls of $M$ that cross both $Z$ and $gZ$. We need to show that the number of orbits of $H\cap gHg^{-1}\acts\mscr{T}(g)$ is bounded independently of $g\in G$. This gives a uniform bound on the number of orbits of maximal cubes in $\pi_Z(gZ)$, hence on the number of vertices.

Consider $\underline{\mf{u}}=(\mf{u}_1,\dots,\mf{u}_k)\in\mscr{T}(g)$. The convex subcomplexes $C(\mf{u}_1),\dots,C(\mf{u}_k)\cu X$ pairwise intersect and they all meet both $C(H)$ and $gC(H)$. It follows that $C(\mf{u}_1),\dots,C(\mf{u}_k),\Pi_g$ pairwise intersect and, by Helly's lemma, their intersection contains a point $p\in\Pi_g$. 

Up to translating $\underline{\mf{u}}$ by an element of $H\cap gHg^{-1}$, we can assume that $p\in P_g$. Each point of $P_g$ lies in the set $C(\mf{w})$ for at most $N$ walls $\mf{w}\in\mscr{W}(M)$. Thus, for each $k$, there are at most $N^k\cdot N'$ orbits of $k$--tuples for the action of $(H\cap gHg^{-1})$ on $\mscr{T}(g)$. Observing that $\mscr{T}(g)$ contains $k$--tuples only for finitely many integers $k$, since $M$ is finite-dimensional, this concludes the proof.
\end{proof}

\begin{cor}\label{core cor}
Let $G\acts X$ be a non-transverse, proper, cocompact action on a $\CAT$ cube complex. Let $G\acts T$ be a minimal action on a simplicial tree such that all edge-stabilisers are convex-cocompact in $X$. Then there exists an action on a $\CAT$ cube complex $G\acts Z$ such that:
\begin{enumerate}
\item $G\acts Z$ is non-transverse, proper, cocompact and without inversions;
\item $G\acts Z$ and $G\acts X$ induce the same coarse median structure on $G$;
\item there exists a $G$--equivariant, surjective median morphism $Z\ra T$;
\item for every hyperplane $\mf{w}\in\mscr{W}(Z)$ obtained as preimage of the midpoint of an edge of $T$, the action $G_{\mf{w}}\acts\mf{w}$ is essential;
\item if $G$--stabilisers of edges of $T$ satisfy the UCP condition in $X$, they also do in $Z$.
\end{enumerate}
\end{cor}
\begin{proof}
Choose a basepoint $p\in X\x T$ and let $M$ be the median algebra generated by the orbit $G\cdot p$. Since the action $G\acts X\x T$ is proper and non-transverse, so is the action $G\acts M$. In addition, $G\acts M$ is cofinite by Proposition~\ref{cofinite subalgebra}. Thus, by Chepoi--Roller duality, there exists a non-transverse, proper, cocompact action on a $\CAT$ cube complex $G\acts Y$ such that the $0$--skeleton of $Y$ is $G$--equivariantly isomorphic to $M$ as a median algebra. 

The factor projection $X\x T\ra T$ gives the required $G$--equivariant median morphism $M\ra T$. The projection, $X\x T\ra X$, gives another $G$--equivariant median morphism $M\ra X$; this ensures that $Y$ and $X$ induce the same coarse median structure on $G$. Condition~(5) follows from Corollary~\ref{UCP and Guirardel}. Up to subdividing, we can assume that $G\acts Y$ is without inversions.

We are only left to ensure that condition~(4) is satisfied. By Corollary~\ref{cmp panel collapse cor 2}, it suffices to pass to a $G$--invariant (non-convex) $\CAT$ subcomplex $Z\cu Y$ inducing the same coarse median structure on $G$. It is immediate to check that the action $G\acts Z$ is again non-transverse, proper, cocompact and without inversions.

Let $\mf{v}\in\mscr{W}(Y)$ be the preimage of the midpoint of an edge of $T$. Since the intersection $\mf{v}\cap Z$ is connected, condition~(3) is not affected by passing to the subcomplex $Z$. 

Finally, the $G$--stabiliser of $\mf{v}\cap Z$ coincides with the stabiliser of $\mf{v}$. The set of vertices in $Z^{(0)}=Y^{(0)}$ that are adjacent to a hyperplane of $Z$ transverse to both $g\mf{v}\cap Z$ and $\mf{v}\cap Z$ is clearly a subset of the set of vertices adjacent to a hyperplane of $Y$ transverse to both $g\mf{v}$ and $\mf{v}$. Thus, condition~(5) remains satisfied in $Z$. This concludes the proof of the corollary.
\end{proof}

\subsection{Earthquake maps.}\label{earthquake cmp sect}

In this subsection, we prove Theorem~\ref{CMP GDT 2}. By Corollary~\ref{core cor}, we can assume that we are in the following setting:
\begin{enumerate}
\item $G\acts X$ is a proper cocompact action on a $\CAT$ cube complex without inversions;
\item $G\acts X$ is also non-transverse; 
\item $G\acts T$ is the action on a tree obtained as restriction quotient of $X$ associated to an orbit of hyperplanes $G\cdot\mf{w}\cu\mscr{W}(X)$;
\item $\mf{A}$ and $\mf{B}$ are the two connected components of $X\setminus G\cdot\mf{w}$ adjacent to $\mf{w}$, the subgroups $A,B\leq G$ are their stabilisers, and $C=A\cap B$ is the stabiliser of $\mf{w}$;
\item $C$ acts essentially on $\mf{w}$, satisfies the UCP condition in $X$, and has no nontrivial finite normal subgroups;
\item we fix an element $z\in Z_A(C)$;
\item depending on whether there are one or two $G$--orbits of vertices in $T$, we denote by $\tau$ and $\s$, respectively, the transvection and the partial conjugation induced by $z$, as defined in the Introduction (in the definition of $\tau$, we fix as stable letter an element $t\in G$ with $t\mf{A}=\mf{B}$).
\end{enumerate}
We emphasise that the element $z$ does not preserve the hyperplane $\mf{w}$ in general.

Though they will not be part of our standing assumptions, it is convenient to give a name to the conditions in parts~(2) and~(3) of Theorem~\ref{CMP GDT 2}:
\begin{enumerate}
\item[$(\ast)$] $z$ lies in $Z_C(C)$, the subgroup $\langle z\rangle$ is convex-cocompact in $X$, and $Z_G(c)$ fixes a point of $T$ for every $c\in C$ with $\langle c\rangle\cap\langle z\rangle\neq\{1\}$;
\item[$(\ast\ast)$] for every infinite-order element $c\in C$ commuting with a finite-index subgroup of $C$, the centraliser $Z_G(c)$ fixes a point of $T$.
\end{enumerate}

\smallskip
We begin with a few lemmas. If $\mf{u}$ is a hyperplane of $X$, we denote by $\mscr{T}(\mf{u})\cu\mscr{W}(X)$ the subset of hyperplanes transverse to $\mf{u}$. Recall that $\mf{u}$ has itself a structure of $\CAT$ cube complex whose hyperplanes are identified with hyperplanes of $X$ lying in $\mscr{T}(\mf{u})$.

Since $z$ acts non-transversely on $X$, note that every hyperplane in $\mc{W}_1(z)$ is skewered by $z$.

\begin{lem}\label{w splits}
The hyperplane $\mf{w}$ splits as a product of cube complexes $\mf{w}_0\x L_1\x\dots\x L_m$, where $m\geq 0$ and each $L_i$ is a quasi-line. All hyperplanes of $X$ corresponding to the factor $\mf{w}_0$ are preserved by $z$. All hyperplanes of $X$ corresponding to the factors $L_i$ are skewered by $z$.
\end{lem}
\begin{proof}
Since $z$ commutes with $C$, the convex subcomplex $\overline{\mc{C}}(z)$ introduced in Proposition~\ref{convex core prop} is $C$--invariant. It follows that every hyperplane in $\mc{W}_1(C)$ crosses $\overline{\mc{C}}(z)$, hence $\mc{W}_1(C)\cu\overline{\mc{W}}_0(z)\sqcup\mc{W}_1(z)$. Since $G$ acts non-transversely and without inversions on $X$, every element of $\mc{W}_1(C)$ is either skewered or preserved by $z$. Since $C$ acts essentially on $\mf{w}$, we have $\mscr{T}(\mf{w})=\mc{W}_1(C)$.

Since $\overline{\mc{W}}_0(z)$ is transverse to $\mc{W}_1(z)$, we have a transverse partition 
\[\mscr{T}(\mf{w})=(\mscr{T}(\mf{w})\cap\overline{\mc{W}}_0(z))\sqcup(\mscr{T}(\mf{w})\cap\mc{W}_1(z)),\] 
which gives rise to a splitting $\mf{w}=\mf{w}_0\x\mf{w}_1$ (see \cite[Lemma~2.5]{CS}). Every hyperplane of the factor $\mf{w}_0$ is preserved by $z$, and every hyperplane of the factor $\mf{w}_1$ is skewered by $z$. The cube complex $\mf{w}_1$ is a restriction quotient of the convex hull in $X$ of any axis of $z$ in $X$. By \cite[Theorem~3.6]{WW}, the latter splits as a product of quasi-lines. It follows that $\mf{w}_1$ is a product of quasi-lines and bounded cube complexes. However, since $C$ acts essentially on $\mf{w}$, there can be no bounded factors.
\end{proof}

The previous lemma yields a partition:
\[\mscr{T}(\mf{w})=\Om_0\sqcup\Om_1\sqcup\dots\sqcup\Om_m.\]
The sets $\Om_i$ are transverse to each other and, since $z$ acts non-transversely, they are all $\langle z\rangle$--invariant. In addition, $z$ fixes $\Om_0$ pointwise and it skewers all other elements of $\mscr{T}(\mf{w})$.

Let $\mf{w}_A,\mf{w}_B\in\mscr{H}(X)$ be the halfspaces bounded by $\mf{w}$ containing $\mf{A}$ and $\mf{B}$, respectively.

\begin{lem}\label{constant D}
There exists a constant $D\geq 0$ such that:
\begin{enumerate}
\item for every $y\in\mf{w}$, we have $d(y,zy)\leq D$;
\item for every $x\in X$, we have $d(\pi_{\mf{w}}(x),\pi_{\mf{w}}(zx))\leq D$ and $\Om_0\cap\mscr{W}(x|zx)=\emptyset$.
\end{enumerate}
\end{lem}
\begin{proof}
Part~(1) is immediate from the fact that $C$ acts cocompactly on $\mf{w}$ and commutes with $z$. Regarding part~(2), we need to bound uniformly the number of hyperplanes in $\mscr{T}(\mf{w})$ that separate $x\in X$ from $zx$. 

Recall that $\mscr{T}(\mf{w})=\bigcup_{j\geq 0}\Om_j$, where $z$ fixes $\Om_0$ pointwise. Since $G$ acts on $X$ without inversions, every halfspace bounded by a hyperplane in $\Om_0$ is left invariant by $z$. It follows that no element of $\Om_0$ can separate $x$ and $zx$. 

Let $H\cu X$ be the convex hull of an axis of $z$. Every hyperplane in $\mscr{T}(\mf{w})\setminus\Om_0$ lies in $\mc{W}_1(z)$, hence it crosses $H$. Denoting by $\pi_H$ the gate-projection to $H$, we conclude that:
\[d(\pi_{\mf{w}}(x),\pi_{\mf{w}}(zx))\leq d(\pi_H(x),\pi_H(zx))=d(\pi_H(x),z\pi_H(x)),\]
since $H$ is $\langle z\rangle$--invariant. Since $G$ acts non-transversely, we have $d(y,zy)=\ell_X(z)$ for every $y\in H$ (for instance by \cite[Proposition~3.17]{FFT} or \cite[Proposition~3.35]{Fio10b}). This proves part~(2) with $D=\ell_X(z)$.
\end{proof}

\begin{lem}\label{invariant separating hyperplanes}
Consider $x\in\mf{w}_B$ and $a\in A$. 
\begin{enumerate}
\item The projections $\pi_{\mf{w}}(ax)$ and $\pi_{\mf{w}}(zaz^{-1}\cdot x)$ are at distance at most $2D$.
\item If $\mf{u}\in\mscr{W}(X)$ is transverse to $a\mf{w}$ and separates $\pi_{\mf{w}}(ax)$ from $\pi_{\mf{w}}(zaz^{-1}\cdot x)$, then there exists an index $j\neq 0$ such that $\mf{u}\in\Om_j\cap a\Om_j$ and $\Om_j\cap a\Om_j\neq\Om_j$. 
\end{enumerate}
\end{lem}
\begin{proof}
We begin with part~(1). Set $y:=\pi_{\mf{w}}(x)$. Observing that the halfspaces $\mf{w}_B$ and $a\mf{w}_B$ are either equal or disjoint, we deduce that $\pi_{\mf{w}}(ax)=\pi_{\mf{w}}\pi_{a\mf{w}}(ax)=\pi_{\mf{w}}(ay)$. Similarly, $\pi_{\mf{w}}(az^{-1}x)=\pi_{\mf{w}}(az^{-1}y)$. Thus, Lemma~\ref{constant D} and the fact that gate-projections are $1$--Lipschitz yield:
\begin{align*}
d(\pi_{\mf{w}}(ax),\pi_{\mf{w}}(zaz^{-1}\cdot x))&\leq D+d(\pi_{\mf{w}}(ax),\pi_{\mf{w}}(az^{-1}x))=D+d(\pi_{\mf{w}}(ay),\pi_{\mf{w}}(az^{-1}y)) \\
&\leq D+d(ay,az^{-1}y)=D+d(y,z^{-1}y)\leq 2D.
\end{align*}

We now prove part~(2). Since $\mf{u}$ separates two points of $\mf{w}$, it lies in $\mscr{T}(\mf{w})$, hence $\mf{u}\in\Om_j$ for some $0\leq j\leq m$. Similarly, since $\mf{u}$ is transverse to $a\mf{w}$, we have $\mf{u}\in a\Om_{j'}$ for some $0\leq j'\leq m$. Since $G$ acts non-transversely on $X$, we must have $j=j'$.

If we had $j=0$, then $\mf{u}$ would be preserved by both $z$ and $aza^{-1}$. Since $G$ acts without inversions, these elements would also leave invariant the two halfspaces bounded by $\mf{u}$. This contradicts the fact that $\mf{u}$ must separate the points $ax$ and $zaz^{-1}\cdot x=z\cdot az^{-1}a^{-1}\cdot ax$.

Thus $\mf{u}\in\Om_j\cap a\Om_j$ for some $j\geq 1$. Suppose for the sake of contradiction that $\Om_j\cap a\Om_j=\Om_j$. Consider the restriction quotient of $X$ determined by the orbit $G\cdot\mf{u}$. This is a tree where $z$ and $aza^{-1}$ are loxodromics with the same translation length. Since $z$ acts non-transversely, we have $(G\cdot\mf{u})\cap\mc{W}_1(z)\cu\Om_j$ and $(G\cdot\mf{u})\cap\mc{W}_1(aza^{-1})\cu a\Om_j$. Thus, the fact that $\Om_j\cu a\Om_j$ implies that $z$ and $aza^{-1}$ have the same axis in the tree.
% not all elements of $\mc{W}_1(z)$ will lie in $\mscr{T}(\mf{w})$, hence in some $\Om_i$, but the above still works because the convex hull of the axis of $z$ is a product of quasi-lines.

It follows that, for every point $y$ in this tree, the points $y$ and $z\cdot az^{-1}a^{-1}\cdot y$ have the same projection to the shared axis of $z$ and $aza^{-1}$.  Since $\mf{u}$ projects to the midpoint of an edge of this axis, it cannot separate the points $ax$ and $zaz^{-1}\cdot x=z\cdot az^{-1}a^{-1}\cdot ax$, a contradiction.
\end{proof}

\begin{lem}\label{uniform constant lem}
There exists a constant $M\geq 0$ such that the following hold.
\begin{enumerate}
\item For all $g\in G$ and $j\neq 0$, either $\Om_j\cap g\Om_j=\Om_j$ or $\#(\Om_j\cap g\Om_j)\leq M$.
\item Suppose that either $(\ast)$ or $(\ast\ast)$ holds. If there exist $g\in G$ and $j\neq 0$ such that $\Om_j\cap g\Om_j=\Om_j$, then there are at most $M$ elements of $G\cdot\mf{w}$ separating $\mf{w}$ and $g\mf{w}$.
\end{enumerate}
\end{lem}
\begin{proof}
We begin with part~(1). Consider $g\in G$. Since $\pi_{\mf{w}}(g\mf{w})$ is convex in $\mf{w}$, the splitting $\mf{w}=\mf{w}_0\x L_1\x\dots\x L_m$ given by Lemma~\ref{w splits} determines a splitting $\pi_{\mf{w}}(g\mf{w})=Y_0\x\dots\x Y_m$. The set of hyperplanes of $X$ corresponding to the factor $Y_j$ is precisely $\Om_j\cap g\Om_j$. Indeed, recall that all intersections $\Om_j\cap g\Om_{j'}$ with $j\neq j'$ are empty because $G$ acts non-transversely on $X$.

Since $C$ satisfies the UCP condition in $X$, there exists a constant $N_1$ such that, for every $g\in G$, the subgroup $C\cap gCg^{-1}$ acts on $\pi_{\mf{w}}(g\mf{w})$ with at most $N_1$ orbits of vertices. This action preserves all factors in the above splitting of $\pi_{\mf{w}}(g\mf{w})$, since $C$ preserves the factors in the splitting of $\mf{w}$. Hence $C\cap gCg^{-1}$ acts on each factor $Y_j$ with at most $N_1$ orbits of vertices.

Since each $L_j$ is an essential quasi-line and $Y_j$ is a cocompact convex subcomplex, $Y_j$ is either the entire $L_j$ or a compact subset. Thus, if $\Om_j\cap g\Om_j\neq\Om_j$, then $C\cap gCg^{-1}$ fixes a point of $Y_j$, and it follows that the diameter of $Y_j$ is at most $N_1$. Since $Y_j$ is isomorphic to a subcomplex of $X$, which is locally finite, this results in a uniform bound on the number of vertices of $Y_j$, hence on the cardinality of $\Om_j\cap g\Om_j$. This proves part~(1).

We now prove part~(2), keeping the above notation. Since $C$ does not have any nontrivial finite normal subgroups, the action $C\acts\mf{w}$ is faithful and we can apply Proposition~\ref{splitting euclidean factor}. As a consequence, $C$ has a finite-index subgroup $C'=C_0\x\langle h_1,\dots,h_m\rangle$, where $\langle h_1,\dots,h_m\rangle\simeq\Z^m$, each $h_j$ acts trivially on $\mf{w}_0$, and each $L_j$ is acted upon trivially by $C_0$ and all $h_i$ with $i\neq j$. If $N_2$ is the index of $C'$ in $C$, then the action $C'\cap gC'g^{-1}\acts \pi_{\mf{w}}(g\mf{w})$ has $\leq N_1N_2^2$ orbits of vertices for every $g\in G$. Let $N_3$ be the highest order of a finite-order element of $C_0$.
% recall that CAT(0) groups have finitely many conjugacy classes of finite subgroups, so N_3 is well-defined

If $\Om_j\cap g\Om_j=\Om_j$ for some $j\neq 0$, then $Y_j=L_j$. It follows that there exist an element $h\in C'\cap gC'g^{-1}$ and a point $x\in \pi_{\mf{w}}(g\mf{w})$ such that $0<d(x,hx)\leq N_1N_2^2$ and $\mscr{W}(x|hx)\cu\Om_j$. The latter implies that $h=h_0h_j^n$ for some $h_0\in C_0$ that is elliptic in $\mf{w}_0$, while the former ensures that $n\leq N_1N_2^2$. Note that $h^{N_3}=h_j^{nN_3}$. In conclusion, $C'\cap gC'g^{-1}$ contains a power of $h_j$ of exponent at most $N_1N_2^2N_3$.

The cyclic subgroup $\langle h_j\rangle$ is convex-cocompact in $X$, since the convex hull of any of its axes is isomorphic to $L_j$. 
% note that, unike $z$, the element $h_j$ does fix $\mf{w}$ because it lies in $C$
By Lemma~\ref{almost normalisers}, there exist finite subsets $F_{j,n}\cu G$ such that:
\[\{g\in G \mid h_j^n\in gCg^{-1}\}=Z_G(h_j^n)\cdot F_{j,n}\cdot C,\]
for all $j\neq 0$ and $n\geq 1$.

Summing up, if $\Om_j\cap g\Om_j=\Om_j$ for some $j\neq 0$, then $g\mf{w}$ belongs to the set $Z_G(h_j^n)F_{j,n}\cdot\mf{w}$ for some $1\leq n\leq N_1N_2^2N_3$. Now, if $(\ast\ast)$ holds, then each subgroup $Z_G(h_j^n)$ is elliptic in the tree $T$. If instead $(\ast)$ holds, $\langle z\rangle$ is convex-cocompact and contained in $C$, so we have $m=1$ and $\mc{W}_1(z)=\Om_1$. In this case, a power of $z$ lies in $C'=C_0\x\langle h_1\rangle$ and its projection to $C_0$ must have finite order. Thus, $z$ and $h_1$ have a common power, hence $Z_G(h_1^n)$ is again elliptic in $T$ by $(\ast)$.

In both cases, this gives a uniform bound to the maximum possible distance between the edges of $T$ corresponding to $\mf{w}$ and $g\mf{w}$, as required by part~(2).
\end{proof}

Note that the three options in part~(b) of the next result correspond exactly to the three options in Theorem~\ref{CMP GDT 2}.

\begin{prop}\label{projection distance prop}
Under the assumptions listed at the beginning of this subsection:
\begin{enumerate}
\item[(a)] we have $\sup_{x\in\mf{w}}\sup_{g\in G}d(\pi_{\mf{w}}(gx),\pi_{\mf{w}}(\s(g)x))<+\infty$;
\item[(b)] we have $\sup_{x\in\mf{w}}\sup_{g\in G}d(\pi_{\mf{w}}(gx),\pi_{\mf{w}}(\tau(g)x))<+\infty$, provided that either $\langle z\rangle\perp C$, or one among $(\ast)$ and $(\ast\ast)$ holds.
\end{enumerate}
\end{prop}
\begin{proof}
Let $K$ be the constant provided by Lemma~\ref{descending projections} applied to $\mf{w}$. Let $m$, $D$, $M$ be the constants provided by Lemmas~\ref{w splits},~\ref{constant D} and~\ref{uniform constant lem}, respectively. 

Recall that $G\acts T$ has either one or two orbits of vertices. We will have to treat separately these two situations, which correspond to parts~(a) and~(b) of the proposition.

\smallskip
{\bf Case~(a).} \emph{The action $G\acts T$ gives an amalgamated product splitting $G=A\ast_CB$. We consider the partial conjugation $\s$ with $\s(a)=a$ for $a\in A$ and $\s(b)=zbz^{-1}$ for $b\in B$.}

\smallskip
It is actually more convenient to consider the automorphism $\overline\s$ satisfying $\overline\s(a)=zaz^{-1}$ for $a\in A$ and $\overline\s(b)=b$ for $b\in B$. This differs from $\s^{-1}$ by composition with an inner automorphism given by $z$. By Lemma~\ref{constant D}, we have $d(\pi_{\mf{w}}(\s^{-1}(g)x),\pi_{\mf{w}}(\overline\s(g)x))\leq 2D$ for every $x\in\mf{w}$ and $g\in G$, 
% the difference to Lemma~\ref{invariant separating hyperplanes}(1) is that $g$ might not lie in $A$, but instead $x$ lies in $\mf{w}$
so it suffices to prove the proposition for $\overline\s$.

We can write $g\in G$ as $g=a_0(b_1a_1\dots b_na_n)b_{n+1}$, with $n\geq 0$ and $a_i\in A\setminus C$, $b_i\in B\setminus C$, except for $a_0$ which is allowed to vanish, and $b_{n+1}$ which is allowed to lie in $C$. Consider a point $x\in\mf{w}$.

For $0\leq i\leq n+1$, we introduce the following hyperplanes and points of $X$: 
\begin{align*}
\mf{w}_i&:=a_0b_1a_1\dots a_{i-1}b_i\cdot\mf{w}, & y_i&:=a_0b_1a_1\dots a_{i-1}b_i\cdot\overline\s(a_ib_{i+1}\dots b_na_nb_{n+1})\cdot x.
\end{align*}
Thus $\mf{w}_0=\mf{w}$ and $\mf{w}_{n+1}=g\mf{w}$, while $y_0=\overline\s(g)x$ and $y_{n+1}=gx$. Observe that $\mf{w}_0,\mf{w}_1\dots,\mf{w}_{n+1}$ is a chain of hyperplanes. For $1\leq i\leq n$, the hyperplane $\mf{w}_i$ separates $y_i$ and $y_{i+1}$ from $\mf{w}$.

\smallskip
{\bf Claim~1:} \emph{at most $2KD$ elements of $\Om_0$ separate $gx$ and $\overline\s(g)x$.} 

\smallskip\noindent
\emph{Proof of Claim~1.} 
By our choice of $K$, there exists a subset $I\cu\{1,\dots,n\}$ such that $\#I\leq K$ and, for every $i\not\in I$, every hyperplane transverse to both $\mf{w}$ and $\mf{w}_i$ is also transverse to $\mf{w}_{i+1}$. 

Recall that, since $G$ acts non-transversely, an element of $\Om_0$ can only be transverse to $\mf{w}_i$ if it lies in the set $a_0b_1a_1\dots a_{i-1}b_i\cdot\Om_0$.

By Lemma~\ref{invariant separating hyperplanes}(1), the points $y_i$ and $y_{i+1}$ are separated by at most $2D$ hyperplanes in $\Om_0\cap\mscr{T}(\mf{w}_i)$. By Lemma~\ref{invariant separating hyperplanes}(2), none of these hyperplanes is transverse to $\mf{w}_{i+1}$ (since $a_0b_1a_1\dots a_{i-1}b_ia_i\cdot\mf{w}$ separates $\mf{w}_i$ and $\mf{w}_{i+1}$). So, if $y_i$ and $y_{i+1}$ are separated by an element of $\Om_0\cap\mscr{T}(\mf{w}_i)$, then $i\in I$. 

Since $\mf{w}_i$ separates $y_i$ and $y_{i+1}$ from $\mf{w}$, we deduce that $y_i$ and $y_{i+1}$ are separated by at most $2D$ elements of $\Om_0$. In addition, they can only be separated by at least one element of $\Om_0$ when $i\in I$. Since $\#I\leq K$, this shows that at most $2KD$ elements of $\Om_0$ separate $y_0$ from $y_{n+1}$, as required.
\hfill$\blacksquare$

\smallskip
{\bf Claim~2:} \emph{for every $j\neq 0$, at most $2D(K+1)+M$ elements of $\Om_j$ separate $gx$ and $\overline\s(g)x$.} 

\smallskip\noindent
\emph{Proof of Claim~2.} 
Fix $j\neq 0$. By Lemma~\ref{uniform constant lem}(1), there exists an index $0\leq k\leq n+1$ such that $\Om_j\cu\mscr{T}(\mf{w}_i)$ for $i\leq k$, while $\#(\Om_j\cap\mscr{T}(\mf{w}_i))\leq M$ for $i>k$. 

Let $I$ be as in the proof of Claim~1. If $i\leq k-1$ and $i\not\in I$, Lemma~\ref{invariant separating hyperplanes}(2) shows that no element of $\Om_j$ separates $y_i$ and $y_{i+1}$. Thus, $y_0$ and $y_k$ are separated by at most $2DK$ elements of $\Om_j$, using Lemma~\ref{invariant separating hyperplanes}(1) as in Claim~1. Similarly, at most $2D$ elements of $\Om_j$ separate $y_k$ and $y_{k+1}$.

Finally, every element of $\Om_j$ separating $y_{k+1}$ and $y_{n+1}$ is transverse to $\mf{w}_{k+1}$, hence there are at most $M$ such hyperplanes. This shows that at most $2D(K+1)+M$ elements of $\Om_j$ separate $y_0$ and $y_{n+1}$, as required.
\hfill$\blacksquare$

\smallskip
Combining the two claims with the fact that $\mscr{T}(\mf{w})=\bigcup_{j\geq 0}\Om_j$, we obtain:
\[d(\pi_{\mf{w}}(gx),\pi_{\mf{w}}(\overline\s(g)x))\leq 2KD+m(2D(K+1)+M).\]

\smallskip
{\bf Case~(b):} \emph{The action $G\acts T$ gives an HNN splitting $G=A\ast_C$. We fix $t\in G$ with $t\mf{A}=\mf{B}$ and consider the transvection $\tau$ with $\tau(a)=a$ for $a\in A$ and $\tau(t)=zt$.}

\smallskip
We can write $g\in G$ as $g=a_1t^{\eps_1}\dots a_nt^{\eps_n}a_{n+1}$ with $n\geq 0$ and $a_i\in A$, $\eps_i\in\{\pm 1\}$. In addition, we can require that this word be reduced in the following sense: 
\begin{itemize}
\item if $\eps_{i-1}=-1$ and $\eps_i=+1$, then $a_i\not\in C$;
\item if $\eps_{i-1}=+1$ and $\eps_i=-1$, then $a_i\not\in t^{-1}Ct$.
%\item if $\eps_1=+1$, then either $a_1=1$ or $a_1\not\in C$; 
%\item if $\eps_n=-1$, then either $a_{n+1}=1$ or $a_{n+1}\not\in C$.
\end{itemize}

Note that $\tau(g)=\overline a_1t^{\eps_1}\dots \overline a_nt^{\eps_n}\overline a_{n+1}$, where: 
\begin{itemize} 
\item $\overline a_i=a_i$ if $(\eps_{i-1},\eps_i)=(+1,-1)$;
\item $\overline a_i=a_iz$ if $(\eps_{i-1},\eps_i)=(+1,+1)$, or $i=1$ and $\eps_1=+1$;
\item $\overline a_i=z^{-1}a_i$ if $(\eps_{i-1},\eps_i)=(-1,-1)$, or $i=n+1$ and $\eps_n=-1$;
\item $\overline a_i=z^{-1}a_iz$ if $(\eps_{i-1},\eps_i)=(-1,+1)$.
\end{itemize}
Since $z$ normalises $C$, this word representing $\tau(g)$ is again reduced as defined above. The words $a_1t^{\eps_1}\dots a_it^{\eps_i}\overline a_{i+1}t^{\eps_{i+1}}\dots \overline a_nt^{\eps_n}\overline a_{n+1}$ are also reduced.

Consider a point $x\in\mf{w}$. For $0\leq i\leq n+1$, we introduce the following hyperplanes and points: 
\begin{align*}
\mf{w}_i&:=a_1t^{\eps_1}\dots a_it^{\eps_i}\cdot\mf{w}, & y_i&:=a_1t^{\eps_1}\dots a_it^{\eps_i}\overline a_{i+1}t^{\eps_{i+1}}\dots \overline a_nt^{\eps_n}\overline a_{n+1}\cdot x.
\end{align*}
Again, we have $\mf{w}_0=\mf{w}$ and $\mf{w}_{n+1}=g\mf{w}$, while $y_0=\tau(g)x$ and $y_{n+1}=gx$. The hyperplanes $\mf{w}=\mf{w}_0,\mf{w}_1\dots,\mf{w}_n,\mf{w}_{n+1}$ form a chain, possibly with $\mf{w}_n=\mf{w}_{n+1}$. For $1\leq i\leq n$, the hyperplane $\mf{w}_i$ separates $y_i$ and $y_{i+1}$ from $\mf{w}$, except if $i=n$ and $a_{n+1}\in C$ or $\overline a_{n+1}\in C$. 

\smallskip
{\bf Claim~3:} \emph{at most $2KD$ elements of $\Om_0$ separate $gx$ and $\tau(g)x$.} 

\smallskip\noindent
\emph{Proof of Claim~3.} 
This is proved exactly as in Claim~1. A little more care is only required when showing that $y_i$ and $y_{i+1}$ are separated by at most $2D$ elements of $\mscr{T}(\mf{w}_i)\cap\Om_0$, and no element of $\mscr{T}(\mf{w}_{i+1})\cap\Om_0$. We spend a few more words on this point.
% If $\eps_{i+1}=+1$, then $x\in\mf{w}_B$ and $\overline a=g^{-1}ag$ or $\overline a=ag$. 
% If $\eps_{i+1}=-1$, then $x\in t^{-1}\mf{w}_A$ and $\overline a=g^{-1}a$ or $\overline a=a$. 

If $\overline a_{i+1}=a_{i+1}$, this is obvious and, if $\overline a_{i+1}=z^{-1}a_{i+1}z$, we can repeat the argument in Claim~1. The cases when $\overline a_{i+1}=z^{-1}a_{i+1}$ or $\overline a_{i+1}=a_{i+1}z$ can be deduced from the previous two via Lemma~\ref{constant D}(2).
\hfill$\blacksquare$

\smallskip
If $\langle z\rangle\perp C$, then Lemma~\ref{pf lemma}(1) shows that $\mscr{T}(\mf{w})=\Om_0$. In this situation, Claim~3 immediately implies that $d(\pi_{\mf{w}}(gx),\pi_{\mf{w}}(\tau(g)x))\leq 2KD$, proving the proposition.

In the rest of the proof, we suppose that either $(\ast)$ or $(\ast\ast)$ is satisfied. 

\smallskip
{\bf Claim~4:} \emph{at most $mM+3D(M+2)$ elements of $\mscr{T}(\mf{w})\setminus\Om_0$ separate $gx$ and $\tau(g)x$.} 

\smallskip\noindent
\emph{Proof of Claim~4.} 
Lemma~\ref{uniform constant lem}(2) rules out the existence of some $j\neq 0$ such that $\mf{w}_{M+2}$ is transverse to every element of $\Om_j$. Lemma~\ref{uniform constant lem}(1) then shows that at most $M$ elements from each $\Om_j$ are transverse to $\mf{w}_{M+2}$. Since $\mf{w}_{M+2}$ separates $y_{M+2}$ and $y_{n+1}$ from $\mf{w}$, we deduce that $\pi_{\mf{w}}(y_{M+2})$ and $\pi_{\mf{w}}(y_{n+1})$ are separated by at most $M$ elements of each $\Om_j$. 

For every $i$, the projections of $y_i$ and $y_{i+1}$ to $\mf{w}_i$ are at distance at most $3D$. This can be deduced from Lemma~\ref{invariant separating hyperplanes}(1) and Lemma~\ref{constant D}(2). Thus, the projections $\pi_{\mf{w}}(y_i)$ and $\pi_{\mf{w}}(y_{i+1})$ are also at distance at most $3D$. It follows that $\pi_{\mf{w}}(y_0)$ and $\pi_{\mf{w}}(y_{M+2})$ are at distance at most $3D(M+2)$.

Summing up, the projections $\pi_{\mf{w}}(y_0)$ and $\pi_{\mf{w}}(y_{n+1})$ are separated by at most $mM+3D(M+2)$ elements of $\mscr{T}(\mf{w})\setminus\Om_0$.
\hfill$\blacksquare$

\smallskip
Combining Claims~3 and~4, we obtain:
\[d(\pi_{\mf{w}}(gx),\pi_{\mf{w}}(\tau(g)x))\leq 2KD+mM+3D(M+2).\]
This concludes the proof of the proposition.
\end{proof}

For simplicity, we introduce the notation $\varphi\in\aut(G)$ to refer to either the partial conjugation $\s$ or the transvection $\tau$.

\begin{defn}\label{earth defn}
Consider the setting described at the beginning of this subsection and $\varphi\in\aut(G)$ as above. The \emph{earthquake map} is the only bijection $\Phi\colon X^{(0)}\ra X^{(0)}$ that satisfies the following:
\begin{itemize}
\item $\Phi(gx)=\varphi(g)\Phi(x)$ for all $x\in X$ and $g\in G$;
\item $\Phi(p)=p$ for all $p\in\mf{A}$, and $\Phi(q)=zq$ for all $q\in\mf{B}$.
\end{itemize}
\end{defn}

We leave to the reader the straightforward check that $\Phi$ exists and is unique. Note that $\Phi$ descends to an automorphism of the tree $T$.

\begin{prop}\label{cmp earthquake}
Under the assumptions of Proposition~\ref{projection distance prop}, the earthquake map $\Phi$ is $(D+1)$--Lipschitz and coarse-median preserving.
\end{prop}
\begin{proof}
First, we prove that $\Phi$ is Lipschitz. It suffices to show that $d(\Phi(x),\Phi(y))\leq D+1$ whenever $x$ and $y$ are the endpoints of an edge of $X$. On each connected component of $X\setminus G\cdot\mf{w}$, the map $\Phi$ is an isometry, so it is enough to consider the case when $x$ and $y$ are in distinct components. 

Thus, suppose that there exist points $x'\in\mf{A}$, $y'\in\mf{B}$ and an element $g\in G$ such that $x=gx'$ and $y=gy'$. Now, since $\Phi(x)=\varphi(g)x'$ and $\Phi(y)=\varphi(g)zy'$, we have:
\[d(\Phi(x),\Phi(y))=d(x',zy')\leq 1+d(y',zy')\leq 1+D,\]
where the last inequality follows from Lemma~\ref{constant D}(1).

Before showing that $\Phi$ is coarse-median preserving, we need to obtain the following.

\smallskip
{\bf Claim~1:} \emph{we have $P:=\sup_{x\in X}d(\pi_{\mf{w}}(x),\pi_{\mf{w}}(\Phi(x)))<+\infty$.}

\smallskip\noindent
\emph{Proof of Claim~1.}
Fix a point $w$ in the intersection between $\mf{A}$ and the carrier of $\mf{w}$. Let $L\geq 0$ be a constant such that the orbit $G\cdot w$ is $L$--dense in $X$. Since $\Phi$ is $(D+1)$--Lipschitz and $\pi_{\mf{w}}$ is $1$--Lipschitz, we have:
\begin{align*}
\sup_{x\in X}d(\pi_{\mf{w}}(x),\pi_{\mf{w}}(\Phi(x)))&\leq \sup_{g\in G}d(\pi_{\mf{w}}(gw),\pi_{\mf{w}}(\Phi(gw)))+L+(D+1)L \\
&=\sup_{g\in G}d(\pi_{\mf{w}}(gw),\pi_{\mf{w}}(\varphi(g)w)))+L+(D+1)L.
\end{align*}
The last quantity is finite by Proposition~\ref{projection distance prop}, which proves the claim.
\hfill$\blacksquare$

\smallskip
{\bf Claim~2:} \emph{for every hyperplane $\mf{u}\in G\cdot\mf{w}$ bounding the region $\mf{A}$ and every $x\in X$, we have $d(\pi_{\mf{u}}(x),\pi_{\mf{u}}(\Phi(x)))\leq P+D$.}

\smallskip\noindent
\emph{Proof of Claim~2.}
Suppose first that $\mf{u}=a\mf{w}$ for some $a\in A$. Then, since $\varphi(a)=a$, we have:
\[d(\pi_{a\mf{w}}(x),\pi_{a\mf{w}}(\Phi(x)))=d(\pi_{\mf{w}}(a^{-1}x),\pi_{\mf{w}}(a^{-1}\Phi(x)))=d(\pi_{\mf{w}}(a^{-1}x),\pi_{\mf{w}}(\Phi(a^{-1}x)))\leq P,\]
by Claim~1. The only other option (only in the HNN case) is that $\mf{u}=at^{-1}\mf{w}$ for some $a\in A$. By the above equalities, it suffices to consider the case $a=1$. Then, we have:
\[d(\pi_{t^{-1}\mf{w}}(x),\pi_{t^{-1}\mf{w}}(\Phi(x)))=d(\pi_{\mf{w}}(tx),\pi_{\mf{w}}(t\Phi(x)))=d(\pi_{\mf{w}}(tx),\pi_{\mf{w}}(\Phi(z^{-1}tx))).\]
Since $\pi_{\mf{w}}(tx)$ and $\pi_{\mf{w}}(z^{-1}tx)$ are at distance at most $D$ by Lemma~\ref{constant D}(2), the above quantity is at most $P+D$, as required.
\hfill$\blacksquare$

\smallskip
Now, consider vertices $x,y,p\in X$ with $p=m(x,y,p)$. We will show that there are at most $4P+2D$ hyperplanes in the set $\mscr{W}(\Phi(p)|\Phi(x),\Phi(y))$. By Lemma~\ref{cmp criterion}, this shows that $\Phi$ is coarse-median preserving.

Since $\Phi$ is the restriction of an isometry on each connected component of $X\setminus G\cdot\mf{w}$, we can assume that $x,y,p$ do not all lie in the same component of $X\setminus G\cdot\mf{w}$. Thus, possibly swapping $x$ and $y$, the points $p$ and $y$ are separated by a hyperplane in the orbit $G\cdot\mf{w}$. Translating $x,y,p$ by an element of $G$ does not alter the size of the set $\mscr{W}(\Phi(p)|\Phi(x),\Phi(y))$ (by the first property in Definition~\ref{earth defn}), so we can assume that $\mf{w}\in\mscr{W}(p|y)$ and $p\in\mf{A}\cup\mf{B}$. 

We only treat the case when $p\in\mf{A}$. The other case is identical if we replace $\Phi$ with the map $z^{-1}\Phi$ and compose $\varphi$ with an inner automorphism of $G$ given by $z$.

By Claim~1, the projections $\pi_{\mf{w}}(\Phi(x))$ and $\pi_{\mf{w}}(\Phi(y))$ are at distance at most $P$ from the points $\pi_{\mf{w}}(x)$ and $\pi_{\mf{w}}(y)$, respectively. Since $x,p,y$ lie on a geodesic in this order, so do their projections $\pi_{\mf{w}}(x)$, $\pi_{\mf{w}}(p)$ and $\pi_{\mf{w}}(y)$. Hence at most $2P$ hyperplanes can separate $\pi_{\mf{w}}(p)$ from $\pi_{\mf{w}}(\Phi(x))$ and $\pi_{\mf{w}}(\Phi(y))$. In other words, at most $2P$ hyperplanes in $\mscr{W}(p|\Phi(x),\Phi(y))$ are transverse to $\mf{w}$.

In case $x\not\in\mf{A}$, let $\mf{u}\in G\cdot\mf{w}$ be a hyperplane adjacent to $\mf{A}$ and separating $\Phi(x)$ from $\mf{A}$. With the argument in the previous paragraph, Claim~2 implies that at most $2(P+D)$ hyperplanes in $\mscr{W}(p|\Phi(x),\Phi(y))$ are transverse to $\mf{u}$.

Since $\varphi$ is the identity on $A$, note that $y$ and $\Phi(y)$ are on the same side of $\mf{w}$ and, similarly, $x$ and $\Phi(x)$ are on the same side of $\mf{u}$. Thus $\mf{w}$ lies in $\mscr{W}(p,\Phi(x) | y,\Phi(y))$ and, when defined, $\mf{u}$ lies in $\mscr{W}(p,\Phi(y) | x,\Phi(x))$.

Now, since the set $\mscr{W}(p|x,y)$ is empty, we have:
\[\mscr{W}(p|\Phi(x),\Phi(y))=\mscr{W}(p,x|\Phi(x),\Phi(y))\cup\mscr{W}(p,y|\Phi(x),\Phi(y)).\]
The set $\mscr{W}(p,y|\Phi(x),\Phi(y))$ is transverse to the set $\mscr{W}(p,\Phi(x) | y,\Phi(y))$, which contains $\mf{w}$. Similarly, $\mscr{W}(p,x|\Phi(x),\Phi(y))$ is transverse to $\mscr{W}(p,\Phi(y) | x,\Phi(x))\ni\mf{u}$ (or it is empty, if $x\in\mf{A}$). In conclusion, every element of $\mscr{W}(p|\Phi(x),\Phi(y))$ is transverse to either $\mf{w}$ or $\mf{u}$, and so there are at most $4P+2D$ hyperplanes in $\mscr{W}(p|\Phi(x),\Phi(y))$. This concludes the proof of the proposition.
\end{proof}

\begin{proof}[Proof of Theorem~\ref{CMP GDT 2}]
By Corollary~\ref{core cor}, it suffices to prove the theorem under the assumptions of this subsection. Let $\varphi\in\aut(G)$ be our DLS automorphism, as above. Applying Proposition~\ref{cmp earthquake} to both $\varphi$ and $\varphi^{-1}$, we obtain a bi-Lipschitz, coarse-median preserving map $\Phi\colon X\ra X$ satisfying $\Phi(gx)=\varphi(g)\Phi(x)$ for all $g\in G$ and $x\in X$. This shows that $\varphi$ preserves the coarse median structure on $G$ induced by $X$.
\end{proof}

\bibliography{../mybib}
\bibliographystyle{alpha}

\end{document}